\documentclass{amsart}
\usepackage{amscd,amsmath,amssymb,mathtools,enumitem,combelow,comment,float}

\usepackage[normalem]{ulem}
\usepackage[dvipsnames]{xcolor}
\makeatletter
\def\squiggly{\bgroup \markoverwith{\textcolor{red}{\lower3.5\p@\hbox{\sixly \char58}}}\ULon}
\makeatother

\makeatletter
\@addtoreset{equation}{section}
\makeatother

\newtheorem{theorem}[subsection]{Theorem}

\newtheorem{proposition}[subsection]{Proposition}
\newtheorem{lemma}[subsection]{Lemma}

\newtheorem{definition}[subsection]{Definition}

\newtheorem{claim}[subsection]{Claim}

\newtheorem{remark}[subsection]{Remark}

\def\loccitt{\emph{loc. cit.}}
\def\loccit{\emph{loc. cit. }}

\def\fsl{{\mathfrak{sl}}}
\def\fgl{{\mathfrak{gl}}}

\def\hgl{{\widehat{\fgl}}}

\def\fC{\mathfrak{C}}
\def\fS{\mathfrak{S}}
\def\fT{\mathfrak{T}}

\def\BN{{\mathbb{N}}}
\def\BF{{\mathbb{F}}}

\def\BQ{{\mathbb{Q}}}
\def\BZ{{\mathbb{Z}}}

\def\CA{{\mathcal{A}}}
\def\CB{{\mathcal{B}}}
\def\DD{{\mathcal{D}}}
\def\CE{{\mathcal{E}}}

\def\CV{{\mathcal{V}}}

\def\ph{\varphi}

\def\e{\varepsilon}

\def\to{\tilde{o}}
\def\tp{\tilde{p}}

\def\vs{\varsigma}

\def\tP{\tilde{P}}

\def\and{\textrm{ }\&\textrm{ }}

\def\sym{\textrm{Sym}}

\def\tp{\widetilde{\ph}}

\def\CC{{{\mathcal{C}}}}

\def\tp{{\tilde{p}}}

\def\sym{\textrm{Sym}}

\def\nn{{{\BN}}^n}
\def\zz{{{\BZ}}^n}

\def\su{{U_q(\dot{\fsl}_n)}}
\def\sug{{U_q^\geq(\dot{\fsl}_n)}}
\def\sul{{U_q^\leq(\dot{\fsl}_n)}}
\def\sup{{U_q^+(\dot{\fsl}_n)}}
\def\summ{{U_q^-(\dot{\fsl}_n)}}
\def\supm{{U_q^\pm(\dot{\fsl}_n)}}

\def\uui{{U_q(\dot{\fgl}_1)}}
\def\uuig{{U_q^\geq(\dot{\fgl}_1)}}
\def\uuil{{U_q^\leq(\dot{\fgl}_1)}}
\def\uuip{{U_q^+(\dot{\fgl}_1)}}
\def\uuim{{U_q^-(\dot{\fgl}_1)}}
\def\uuipm{{U_q^\pm(\dot{\fgl}_1)}}

\def\uu{{U_q(\dot{\fgl}_n)}}

\def\uuo{{U_q^0(\dot{\fgl}_n)}}
\def\uup{{U_q^+(\dot{\fgl}_n)}}
\def\uum{{U_q^-(\dot{\fgl}_n)}}
\def\uul{{U_q^\leq(\dot{\fgl}_n)}}
\def\uupm{{U_q^\pm(\dot{\fgl}_n)}}
\def\uump{{U_q^\mp(\dot{\fgl}_n)}}

\def\uug{{U_q^\geq(\dot{\fgl}_n)}}

\def\UU{{U_{q,\oq}(\ddot{\fgl}_n)}}

\def\UUpm{{U^\pm_{q, \oq}(\ddot{\fgl}_n)}}

\def\A{{\CA}}

\def\B{\CB}

\def\bd{{\mathbf{d}}}
\def\bk{{\mathbf{k}}}
\def\bl{{\mathbf{l}}}

\def\bs{{\boldsymbol{\vs}}}

\def\la{{\lambda}}

\def\bde{{\boldsymbol{\delta}}}
\def\bla{{\boldsymbol{\la}}}
\def\bmu{{\boldsymbol{\mu}}}

\def\of{{\overline{f}}}

\def\oq{{\overline{q}}}

\def\bara{\bar{a}}
\def\barc{\bar{c}}
\def\bari{\bar{i}}

\def\be{\bar{e}}
\def\bf{\bar{f}}
\def\bA{\bar{A}}
\def\bB{\bar{B}}
\def\bE{\bar{E}}
\def\bF{\bar{F}}
\def\bG{\bar{G}}

\def\bY{\bar{Y}}
\def\bZ{\bar{Z}}

\def\hi{r}

\def\zzz{\frac {\BZ^2}{(n,n)\BZ}}
\def\lhs{\text{LHS}}
\def\rhs{\text{RHS}}

\def\hdeg{\text{hdeg }}
\def\vdeg{\text{vdeg }}

\begin{document}

\title[The PBW basis of $\UU$]{\Large{\textbf{The PBW basis of $\UU$}}}

\author[Andrei Negu\cb t]{Andrei Negu\cb t}
\address{MIT, Department of Mathematics, Cambridge, MA, USA}
\address{Simion Stoilow Institute of Mathematics, Bucharest, Romania}
\email{andrei.negut@gmail.com}

\maketitle

\begin{abstract} We consider the PBW basis of the quantum toroidal algebra of $\fgl_n$, which was developed in \cite{Tor}, and prove commutation relations between its generators akin to the ones studied in \cite{BS} for $n=1$. This gives rise to a new presentation of the quantum toroidal algebra of type $A$.

\medskip


\end{abstract}

\section{Introduction}

Let us fix $n>1$, and consider the type $A$ quantum affine algebra:
$$
\uu =\su \otimes \uui 
$$
(dots will replace hats in the present paper) whose generators are of two kinds: \\

\begin{itemize}
	
\item the simple root generators $\{x_i^\pm\}_{i \in \BZ/n\BZ}$ of $\su$ \\

\item the imaginary root generators $\{p_{\pm k}\}_{k \in \BN}$ of $\uui$ \\

\end{itemize}

\noindent  Our main object of study is the type $A$ quantum toroidal algebra, which was shown in \cite{Tor} to factor, \underline{as a vector space}, in terms of slope subalgebras:
\begin{equation}
\label{eqn:quant intro}
\UU \cong \bigotimes^\rightarrow_{\mu \in \BQ\sqcup \infty} \CB_\mu 
\end{equation}
where if $\frac ba$ is a reduced fraction with $b \in \BZ$ and $a \in \BN \sqcup 0$, we have:
\begin{equation}
\label{eqn:slope intro}
\CB_{\frac ba}\stackrel{\sim}\longrightarrow U_{q}(\dot{\fgl}_{\frac ng})^{\otimes g}
\end{equation}
where $g = \gcd(n,a)$. The isomorphisms \eqref{eqn:slope intro} imply that the simple and imaginary generators of quantum affine algebras give rise to the following elements of $\UU$:  \\

\begin{itemize}
	
\item simple root generators $P_{[i;j)}^{(k)}$ defined for all: 
\begin{equation}
\label{eqn:indexing 1}
(i, j) \in \zzz, \ i \not \equiv j \text{ mod }n, \ k \in \BZ \ \text{ s.t. } \gcd(k,j-i)=1
\end{equation}

\item imaginary root generators $P_{l \bde, \hi}^{(k')}$ defined for all:
\begin{equation}
\label{eqn:indexing 2}
\hi \in \BZ/g\BZ, \ k' \in \BZ, \ l \in \BZ 
\end{equation} 

\end{itemize}

\noindent We will often extend the notation \eqref{eqn:indexing 1} by setting:
\begin{equation}
\label{eqn:short}
P_{[i;i+nl)}^{(k)} := P_{l \bde, i}^{(k)}
\end{equation}
whenever $\gcd(k,nl) = 1$. \\

\noindent The elements $P_{[i;j)}^{(k)}$ and $P_{l \bde, \hi}^{(k')}$ lie in $\B_\mu \subset \UU$, where the slope $\mu$ is defined as:
$$
\mu = \frac k{j-i} \text{ for \eqref{eqn:indexing 1}} \quad \text{and} \quad \mu = \frac {k'}{nl} \text{ for \eqref{eqn:indexing 2}}
$$
and the subscript of the $P$'s refers to their grading as elements of $\UU$, where:
\begin{equation}
\label{eqn:degrees}
[i;j) = \begin{cases} \bs^i + ... + \bs^{j-1} &\text{if } i\leq j \\ -\bs^j - ... - \bs^{i-1} &\text{if } i > j \end{cases}, \ \bs^i = (\underbrace{0,...,0,1,0,...,0}_{1\text{ on }i-\text{th position}}), \ \bde = (1,...,1)
\end{equation}
lie in $\zz$. See Subsection \ref{sub:summarize} for the definition of the root generators introduced above, as well as for the precise combinatorics behind their indexing sets. As shown in \cite{Tor}, ordered products of the root generators described above give rise to a PBW basis of the quantum toroidal algebra. The main goal of the present paper is to compute commutation relations between the aforementioned root generators, which will allow us to give a new presentation of the quantum toroidal algebra: \\

\begin{theorem}
\label{thm:intro main}

The map $P_{[i;j)}^{(k)} \mapsto p_{[i;j)}^{(k)}$, $P_{l \bde, \hi}^{(k)} \mapsto p_{l \bde, \hi}^{(k)}$ yields an isomorphism:
\begin{equation}
\label{eqn:generators intro}
\UU \stackrel{\sim}\longrightarrow \boxed{ \CC := \Big \langle \CE_{\mu} \Big \rangle_{\mu \in \BQ \sqcup \infty} \Big/\text{relations \eqref{eqn:rel 1 intro}, \eqref{eqn:rel 2 intro}}}
\end{equation}
where for any coprime integers $a,b$ with $g = \gcd(n,a)$, we set:
\begin{equation}
\label{eqn:abstract subalgebra}
\CE_{\frac ba} := U_{q}(\dot{\fgl}_{\frac ng})^{\otimes g} 
\end{equation}
and we label the simple and imaginary root generators by:
\begin{align*}
&U_{q}(\dot{\fgl}_{\frac ng})^{\otimes g} \ni \underbrace{1 \otimes ... \otimes 1 \otimes x_i^\pm \otimes 1 \otimes ... \otimes 1}_{x_i^\pm \text{ on }r-\text{th position}} \quad \leadsto \quad p_{\pm [ai+r;a(i+1)+r)}^{(\pm b)} \in \CE_{\frac ba} \\
&U_{q}(\dot{\fgl}_{\frac ng})^{\otimes g} \ni \underbrace{1 \otimes ... \otimes 1 \otimes p_{\pm k} \otimes 1 \otimes ... \otimes 1}_{p_{\pm k} \text{ on }\hi-\text{th position}} \quad \leadsto \quad p_{\pm 
\frac {ak\bde}g, \hi}^{(\pm \frac {bkn}g)} \ \quad \qquad \in \CE_{\frac ba}
\end{align*}
The relations among the various $p_{[i;j)}^{(k)}, p_{l \bde, \hi}^{(k')} \in \CC$ are, for all indices, as follows:
\begin{equation}
\label{eqn:rel 1 intro}
\Big[ p_{[i;j)}^{(k)}, p_{l\bde, \hi}^{(k')} \Big] = x_{[i;j), l\bde,r}^{(k),(k')}
\end{equation}
if $\left|\det \begin{pmatrix} k & k' \\ j-i & nl \end{pmatrix} \right| = \gcd(k',nl)$, and:
\begin{equation}
\label{eqn:rel 2 intro}
p_{[i;j)}^{(k)} p_{[i';j')}^{(k')} q^{\circ} - p_{[i';j')}^{(k')} p_{[i;j)}^{(k)} q^{\bullet} = y_{[i;j),[i';j')}^{(k),(k')}
\end{equation} 
if $\det \begin{pmatrix} k & k' \\ j-i & j'-i' \end{pmatrix} = \gcd(k+k',j - i + j' - i')$. Above, the elements: 
$$
x_{[i;j), l\bde,r}^{(k),(k')} \in \CE_{\frac {k+k'}{j-i+nl}} \qquad \text{and} \qquad y_{[i;j),[i';j')}^{(k),(k')} \in \CE_{\frac {k+k'}{j-i+j'-i'}}
$$
as well as the numbers $\circ , \bullet \in \{-1,0,1\}$, will be defined explicitly in Theorem \ref{thm:main}. \\

\end{theorem}

\noindent One can better understand the intuition behind Theorem \ref{thm:intro main} by associating vectors: \\

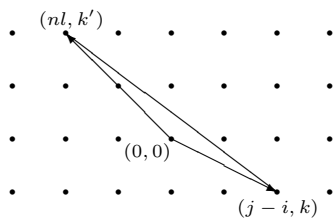
\begin{picture}(100,110)(-110,-27)

\put(0,0){\circle*{2}}\put(20,0){\circle*{2}}\put(40,0){\circle*{2}}\put(60,0){\circle*{2}}\put(80,0){\circle*{2}}\put(100,0){\circle*{2}}\put(120,0){\circle*{2}}\put(0,20){\circle*{2}}\put(20,20){\circle*{2}}\put(40,20){\circle*{2}}\put(60,20){\circle*{2}}\put(80,20){\circle*{2}}\put(100,20){\circle*{2}}\put(120,20){\circle*{2}}\put(0,40){\circle*{2}}\put(20,40){\circle*{2}}\put(40,40){\circle*{2}}\put(60,40){\circle*{2}}\put(80,40){\circle*{2}}\put(100,40){\circle*{2}}\put(120,40){\circle*{2}}\put(0,60){\circle*{2}}\put(20,60){\circle*{2}}\put(40,60){\circle*{2}}\put(60,60){\circle*{2}}\put(80,60){\circle*{2}}\put(100,60){\circle*{2}}\put(120,60){\circle*{2}}

\put(60,20){\vector(2,-1){40}}
\put(60,20){\vector(-1,1){40}}
\put(100,0){\line(-4,3){80}}

\put(42,13){\scriptsize{$(0,0)$}}
\put(85,-8){\scriptsize{$(j-i,k)$}}
\put(10,63){\scriptsize{$(nl,k')$}}

\end{picture}

\noindent $(j-i,k)$ and $(nl,k')$ to the generators $P_{[i;j)}^{(k)}$ and $P_{l\bde,r}^{(k')}$, respectively. Then \eqref{eqn:generators intro} states that generators whose associated vectors are parallel obey the relations in the algebras \eqref{eqn:slope intro}, while relations \eqref{eqn:rel 1 intro} and \eqref{eqn:rel 2 intro} state that generators whose associated vectors are ``close to each other" also obey a certain commutation relation. \\

\noindent Although we only consider $n>1$, the $n=1$ analogue of Theorem \ref{thm:intro main} would be precisely the main result of \cite{S}, where Schiffmann established the isomorphism between the Ding-Iohara-Miki algebra (the $n=1$ version of the quantum toroidal algebra) and the elliptic Hall algebra studied by Burban-Schiffmann in \cite{BS}. Indeed, the $n=1$ version of relations \eqref{eqn:rel 1 intro}, \eqref{eqn:rel 2 intro} were shown in \loccit to hold in the Hall algebra of the category $\fC_1$ of coherent sheaves over an elliptic curve over $\BF_q$. The dream would be to find an analogue $\fC_n$ of the aforementioned category, whose double Hall algebra is isomorphic to the quantum toroidal algebra, for all $n$. While we do not have a precise proposal for such a category, Theorem \ref{thm:intro main} gives hints as to what such a $\fC_n$ should be, by positing that the generators:
\begin{equation}
\label{eqn:new generators 2}
\left\{ P_{[i;j)}^{(k)},  P_{l\bde, \hi}^{(k')} \right\} \in \UU
\end{equation}
be related to the classes of indecomposable objects in $\fC_n$ and that relations \eqref{eqn:rel 1 intro}, \eqref{eqn:rel 2 intro} provide information on the extensions between these objects in $\fC_n$. \\

\noindent The structure of this paper is the following: \\

\begin{itemize}[leftmargin=*]
	
	\item in Section \ref{sec:quantum}, we review the quantum affine algebra $\uu$. \\
	
	\item in Section \ref{sec:shuf}, we review the shuffle algebra $\CA$ and its isomorphism with $\UU$. \\
	
	\item in Section \ref{sec:proof}, we properly state and prove Theorem \ref{thm:intro main}. \\
	
	\item in Section \ref{sec:final}, we state and prove an alternate version of Theorem \ref{thm:intro main}, which will yield a different presentation of $\UU$, that will be used in \cite{Tale}. \\ 
	
\item in Section \ref{sec:index}, we provide an index of notations. \\

\end{itemize}

\noindent I would like to thank Francesco Sala, Olivier Schiffmann and Alexander Tsymbaliuk for many wonderful discussions over the years. I gratefully acknowledge the support of NSF grants DMS--1600375, DMS--1760264 and DMS--1845034, as well as support from the Alfred P. Sloan Foundation. \\

\section{Quantum Algebras}
\label{sec:quantum}

\noindent The main purpose of the present Section is to review the quantum algebra $\uu$.  \\

\begin{itemize}[leftmargin=*]

\item in Subsection \ref{sub:drinfeldouble}, we review bialgebras, pairings and the Drinfeld double  \\

\item in Subsection \ref{sub:su}, we recall the definition of $\su$ as a bialgebra \\

\item in Subsection \ref{sub:heis}, we recall the definition of $\uui$ as a bialgebra; we also point out the non-degeneracy of the pairings \eqref{eqn:village} and \eqref{eqn:town} in Remark \ref{rem:non-degenerate} \\

\item in Subsection \ref{sub:uu}, we study $\uu = \su \otimes \uui$ and its generators $e_{\pm [i;j)}$ \\

\item in Subsection \ref{sub:pbw uu}, we discuss the PBW basis, the grading, and the classification of primitive elements of $\uu$ \\

\item in Subsection \ref{sub:primitive}, we connect the generators $x_i^\pm, p_{\pm k}$ of $\su \otimes \uui$ with the generators $e_{\pm [i;j)}$ of $\uu$ (see Definition \ref{def:correspondence}) \\

\item in Subsection \ref{sub:order n}, we define the order $n$ automorphism of $\uu$, and prove a Proposition that will be used later on \\

\item in Subsection \ref{sub:antipodes}, we consider elements $\be_{\pm [i;j)} \in \uu$, which are closely related to the antipodes of $e_{\pm [i;j)}$ \\

\end{itemize}

\subsection{}\label{sub:drinfeldouble}

All algebras $A$ considered in the present paper are bialgebras over a field $\BF$, meaning that they are endowed with a product and coproduct:
$$
A \otimes A \stackrel{*}\longrightarrow A \qquad \qquad A \stackrel{\Delta}\longrightarrow A \otimes A
$$
which are associative and coassociative, respectively. All our bialgebras will be endowed with a unit $1 : \BF \rightarrow A$ and a counit $\e : A \rightarrow \BF$. The most important property of the data above is compatibility between product and coproduct:
\begin{equation}
\label{eqn:compatibility}
\Delta(a * a') = \Delta(a) * \Delta(a')
\end{equation}
$\forall a,a' \in A$. We will often use Sweedler notation for the coproduct, namely:
\begin{equation}
\label{eqn:sweedler}
\Delta(a) = a_{1} \otimes a_{2}
\end{equation}
$\forall a \in A$, which implies the existence of a hidden summation sign in front of the tensor in the right-hand side (so the full notation would be $\Delta(a) = \sum_i a_{1,i} \otimes a_{2,i}$ with $i$ running over some indexing set). Then \eqref{eqn:compatibility} can be written as:
$$
(aa')_1 \otimes (aa')_2 = a_1a_1' \otimes a_2 a_2'
$$
Given bialgebras $A^+$ and $A^-$, a bialgebra pairing between them:
\begin{equation}
\label{eqn:bialgpair}
\langle \cdot, \cdot \rangle : A^+ \otimes A^- \rightarrow \BF
\end{equation}
is an $\BF$--linear pairing which satisfies the properties:
\begin{align}
&\left \langle a * a', b \right \rangle = \left \langle a\otimes a', \Delta^{\text{op}}(b) \right \rangle \label{eqn:bialg 1} \\
&\left \langle a,b * b' \right \rangle  = \left \langle \Delta(a), b \otimes b' \right \rangle \label{eqn:bialg 2}
\end{align}
for all $a,a' \in A^+$ and $b,b' \in A^-$ (where $\Delta^{\text{op}}$ denotes the opposite coproduct) and $\langle a, 1 \rangle = \e(a)$, $\langle 1, b \rangle = \e(b)$ for all $a$ and $b$. We will often abuse notation by writing:
$$
\langle b,a \rangle = \langle a,b \rangle 
$$
for all $a \in A^+$, $b \in A^-$. \\

\begin{definition}
\label{def:drinfeld}
	
(\cite{D1}) Given any two bialgebras $A^+$ and $A^-$ with a bialgebra pairing \eqref{eqn:bialgpair} between them, consider the vector space:
$$
A = A^+ \otimes A^-
$$ 
It can be made into a bialgebra by requiring that $A^+ \cong A^+ \otimes 1$ and $A^- \cong 1 \otimes A^-$ are sub-bialgebras of $A$, which freely generate $A$ subject to the relations:
\begin{equation}
\label{eqn:drinfeld}
a_1 b_1 \langle a_2,b_2 \rangle = \langle a_1,b_1 \rangle b_2 a_2
\end{equation}
$\forall a \in A^+, b\in A^-$, where in \eqref{eqn:drinfeld} we use Sweedler notation \eqref{eqn:sweedler} for $\Delta(a)$ and $\Delta(b)$. \\	

\end{definition}

\subsection{} 
\label{sub:su}

Throughout the present paper, Kronecker $\delta$ symbols are considered mod $n$, i.e. $\delta_j^i = 1$ if $i \equiv j$ mod $n$, and 0 otherwise. Consider the type $A$ quantum group: \footnote{Note that the algebra below is slightly larger than the usual type $A$ quantum group, as it contains $\psi_1,...,\psi_n$ and not just their ratios. We choose this definition for notational convenience.}
\begin{equation}
\label{eqn:special}
\su = \BQ(q) \Big \langle x_i^\pm, \psi_s^{\pm 1}, c^{\pm 1} \Big \rangle^{i \in \BZ/n\BZ}_{s \in \{1,...,n\}}
\end{equation}
modulo the fact that $c$ is central, as well as the following relations:
\begin{equation}
\label{eqn:special 0}
\psi_s \psi_{s'} = \psi_{s'} \psi_s
\end{equation}
\begin{equation}
\label{eqn:special 1}
\psi_s x^\pm_i  = q^{\pm (\delta_{s}^{i+1} - \delta^i_s)} x^\pm_i \psi_s 
\end{equation}
\begin{equation}
\label{eqn:special 2}
[x_i^\pm, x_j^\pm] = 0 \qquad \qquad \quad \text{if } j \notin \{i-1,i+1\}
\end{equation}
\begin{equation}
\label{eqn:special 3}
[x_i^\pm,[x_i^\pm, x_j^\pm]_q]_{q^{-1}} = 0 \quad \text{if } j \in \{i-1,i+1\}
\end{equation}
\begin{equation}
\label{eqn:special 4}
[x_i^+, x_j^-] =  \frac {\delta_i^j}{q-q^{-1}} \left(\frac {\psi_{i+1}}{\psi_i} - \frac {\psi_i}{\psi_{i+1}} \right)
\end{equation}
\footnote{As written, the $q$-Serre relation \eqref{eqn:special 3} assumes $n>2$. For $n=2$, one must replace it with a similar expression which is cubic in $x_i^\pm$.} for all $i,j \in \BZ/n\BZ$ and $s,s' \in \{1,...,n\}$, where we write:
\begin{equation}
\label{eqn:q comm}
[a,b]_q = ab - q ab
\end{equation}
We will extend the indexing set of the $\psi$'s to all integers by setting:
\begin{equation}
\label{eqn:quasi per}
\psi_{s+n} = c \psi_s
\end{equation}
$\forall s \in \BZ$. The counit given by $\e(x_i^\pm) = 0$, $\e(\psi_s) = 1$ and the coproduct given by:
\begin{align}
&\Delta(c) = c \otimes c, \ \Delta(\psi_s) = \psi_s \otimes \psi_s \label{eqn:cop special 1} \\
&\Delta(x_i^+) = \frac {\psi_{i+1}}{\psi_i} \otimes x_i^+ + x_i^+ \otimes 1 \label{eqn:cop special 2} \\
&\Delta(x_i^-) = 1 \otimes x_i^- + x_i^- \otimes \frac {\psi_i}{\psi_{i+1}} \label{eqn:cop special 3} 
\end{align}
make $\su$ into a bialgebra. It is easy to see that the ``half" subalgebras:
\begin{align}
&\sug = \BQ(q) \Big \langle x_i^+, \psi_s^{\pm 1}, c^{\pm 1} \Big \rangle^{i \in \BZ/n\BZ}_{s \in \{1,...,n\}} \subset \su \label{eqn:positive} \\
&\sul = \BQ(q) \Big \langle x_i^-, \psi_s^{\pm 1}, c^{\pm 1} \Big \rangle^{i \in \BZ/n\BZ}_{s \in \{1,...,n\}} \subset \su \label{eqn:negative}
\end{align}
are bialgebras. There is a bialgebra pairing:
\begin{equation}
\label{eqn:village}
\sug \otimes \sul \stackrel{\langle \cdot, \cdot \rangle}\longrightarrow \BQ(q)
\end{equation}
determined by properties \eqref{eqn:bialg 1}--\eqref{eqn:bialg 2}, the formulas:
\begin{equation}
\label{eqn:pairing village}
\Big \langle x_i^+, x_j^- \Big \rangle = \frac {\delta_j^i}{q^{-1}-q} \qquad \text{and} \qquad \langle \psi_s,\psi_{s'} \rangle = q^{-\delta_{s'}^s}
\end{equation}
(all other pairings between the generators are trivial). It is well-known that \eqref{eqn:special} is a Drinfeld double with respect to the pairing \eqref{eqn:village} (see Subsection 2.10 of \cite{Tor} for the precise statement, as in entails identifying the Cartan elements $\psi_s$ in the positive half \eqref{eqn:positive} with the same-named elements in the negative half \eqref{eqn:negative}). Let:
\begin{equation}
\label{eqn:notation su}
\supm \subset \su
\end{equation}
be the subalgebras generated by $\{x_i^\pm\}_{i \in \BZ/n\BZ}$. \\

\subsection{} \label{sub:heis} 

We also consider the $q$--deformed Heisenberg algebra:
\begin{equation}
\label{eqn:heisenberg}
\uui = \BQ(q) \Big \langle p_{\pm k}, c^{\pm 1} \Big \rangle_{k \in \BN}
\end{equation}
where $c$ is central and the $p_{\pm k}$ all commute, except for:
$$
[p_k, p_{-k}] = k \cdot \frac {c^k - c^{-k}}{q^k-q^{-k}} 
$$
The counit $\e(p_k) = 0$, $\e(c) = 1$ and the coproduct given by $\Delta(c) = c \otimes c$ and:
\begin{align}
&\Delta(p_k) = c^k \otimes p_k + p_k \otimes 1 \label{eqn:cop heis 1} \\
&\Delta(p_{-k}) = 1 \otimes p_{-k} + p_{-k} \otimes c^{-k} \label{eqn:cop heis 2} 
\end{align}
make $\uui$ into a bialgebra. As in the previous Subsection, it is easy to see that the Heisenberg algebra is the Drinfeld double of its two halves:
\begin{align}
&\uuig = \BQ(q) \Big \langle p_{k}, c^{\pm 1} \Big \rangle_{k \in \BN} \subset \uui \label{eqn:positive heis} \\
&\uuil = \BQ(q) \Big \langle p_{-k}, c^{\pm 1} \Big \rangle_{k \in \BN} \subset \uui \label{eqn:negative heis}
\end{align}
with respect to the bialgebra pairing:
\begin{equation}
\label{eqn:town}
\uuig \otimes \uuil \stackrel{\langle \cdot, \cdot \rangle}\longrightarrow \BQ(q)
\end{equation}
which is determined by properties \eqref{eqn:bialg 1}--\eqref{eqn:bialg 2} and the formula:
\begin{equation}
\label{eqn:pairing town}
\Big \langle p_k, p_{-k} \Big \rangle = \frac {k}{q^{-k}-q^k} 
\end{equation}
(all other pairings between the generators are trivial). We will write:
\begin{equation}
\label{eqn:notation heis}
\uuipm \subset \uui
\end{equation}
for the subalgebras generated by $\{p_{\pm k}\}_{k \in \BN}$. \\

\begin{remark}
\label{rem:non-degenerate}

The pairings \eqref{eqn:village} and \eqref{eqn:town} are non-degenerate: for the former this is proved in \cite{J}, while for the latter, non-degeneracy holds because the bases:
\begin{align*}
&\Big\{ p_{\lambda} = p_{\lambda_1} p_{\lambda_2} \dots \Big\}_{\lambda = (\lambda_1 \geq \lambda_2 \geq \dots)} \ \subset \ \uuip \\
&\Big \{p_{-\mu} = p_{-\mu_1} p_{-\mu_2}\dots  \Big\}_{\mu = (\mu_1 \geq \mu_2 \geq \dots)} \subset \uuim
\end{align*}
have the property that $\langle p_{\lambda}, p_{-\mu} \rangle = \delta_{\lambda}^{\mu} \cdot$ non-zero constant. This underscores a slight imprecision in our terminology: when we call the pairings \eqref{eqn:village} and \eqref{eqn:town} non-degenerate, what we are actually claiming (and using in the present paper) is the non-degeneracy of their restrictions to the $\pm$ subalgebras, namely:
\begin{align}
&\sup \otimes \summ \stackrel{\langle \cdot, \cdot \rangle}\longrightarrow \BQ(q) \label{eqn:restriction 1} \\
&\uuip \otimes \uuim \stackrel{\langle \cdot, \cdot \rangle}\longrightarrow \BQ(q) \label{eqn:restriction 2}
\end{align}
One can still make sense of the non-degeneracy of the pairings \eqref{eqn:village} and \eqref{eqn:town}, but this would require certain modifications, such as working over the ring of power series in $\log(q)$, replacing $\psi_s$ by $\log(\psi_s)$ in the definition of the quantum group, and adding more central elements. We will not need this extra layer of complexity. \\

\end{remark}

\subsection{} 
\label{sub:uu}

Let us consider the bialgebra:
\begin{equation}
\label{eqn:quantum as tensor}
\uu = \su \otimes \uui \Big/ (c \otimes 1 - 1 \otimes c)
\end{equation}
By combining the results of \cite{DF} and \cite{D2}, the algebra $\uu$ is generated by:
\begin{equation}
\label{eqn:collection of generators 1}
\Big \langle e_{\pm[i;j)}, \psi_s^{\pm 1}, c^{\pm 1} \Big \rangle^{s \in \{1,...,n\}}_{(i<j) \in \zzz}
\end{equation}
where the generators $e_{\pm [i;j)}$ satisfy quadratic relations (the famous $RTT = TTR$ relations of \cite{FRT, RS}, see \cite{Tor} for an overview), and their coproduct is given by:
\begin{align}
&\Delta(e_{[i;j)}) = \sum_{s=i}^j e_{[s;j)} \frac {\psi_s}{\psi_i} \otimes e_{[i;s)} \label{eqn:cop quant 1} \\
&\Delta(e_{-[i;j)}) = \sum_{s=i}^j e_{-[i;s)} \otimes e_{-[s;j)} \frac {\psi_i}{\psi_s} \label{eqn:cop quant 2}
\end{align}
We conclude that $\uu$ is generated by either the set of generators \eqref{eqn:collection of generators 1} or by:
\begin{equation}
\label{eqn:collection of generators 2}
\Big \langle x_i^\pm, p_{\pm k}, \psi^{\pm 1}_s, c^{\pm 1} \Big \rangle^{i \in \BZ/n\BZ}_{k \in \BN, s \in \{1,...,n\}}
\end{equation}
The connection between these two generating sets is given by:
\begin{align}
&e_{[i;i+1)} = x_i^+(q-q^{-1}) \label{eqn:simple correspondence 1} \\
&e_{-[i;i+1)} = x_i^-(q^{-2}-1) \label{eqn:simple correspondence 2}
\end{align}
We do not know an explicit formula for $p_{\pm k}$ in terms of the $e_{\pm [i;j)}$, but we will explain how to obtain an implicit connection in the following Subsections. \\

\subsection{} 
\label{sub:pbw uu}

The generators $e_{\pm [i;j)}$ give rise to PBW bases of the subalgebras:
\begin{equation}
\label{eqn:notation uu}
\uupm = \BQ(q) \Big \langle x_i^\pm, p_{\pm k} \Big \rangle_{i \in \BZ/n\BZ, k \in \BN} = \BQ(q) \Big \langle e_{\pm[i;j)} \Big \rangle_{(i<j) \in \zzz} \subset \uu
\end{equation}
by which we mean that $\uupm$ is spanned, as a $\BQ(q)$--vector space, by products of $e_{\pm [i;j)}$'s in a specific order (for example, ascending order of $j-i$, and then ascending order of $i$ mod $n$ to break ties). We note the triangular decomposition:
$$
\uu = \uup \otimes \uuo \otimes \uum
$$
where $\uuo = \BQ(q) [\psi_s^{\pm 1}, c^{\pm 1}]_{1\leq s \leq n}$. The algebra $\uu$ is graded by $\zz$, with:
\begin{equation}
\label{eqn:grading}
\deg e_{\pm[i;j)} = \pm [i;j), \qquad \deg p_{\pm k} = \pm k \bde, \qquad \deg \psi_s = 0
\end{equation}
where $[i;j)$ and $\bde = (1,\dots,1)$ are defined in \eqref{eqn:degrees}. For all $s$, consider:
\begin{equation}
\label{eqn:psi}
\ph_s = \frac {\psi_{s+1}}{\psi_s}
\end{equation}
It is easy to see that the coproduct of any element $x \in \uupm$ takes the form:
\begin{align}
&\Delta(x) = \prod_{i=1}^n\ph_i^{k_i} \otimes x + \underbrace{\dots}_{\text{intermediate terms}} + x \otimes 1 \qquad \text{if } x \in \uup \label{eqn:intermediate terms 1} \\
&\Delta(x) = 1 \otimes x + \underbrace{\dots}_{\text{intermediate terms}} + x \otimes \prod_{i=1}^n \ph_i^{k_i} \qquad \text{if } x \in \uum \label{eqn:intermediate terms 2}
\end{align}
if $\deg(x) = (k_1,\dots,k_n) \in \zz$, where the intermediate terms are all of the form $x' \otimes x''$ with $\deg x', \deg x'' \neq 0$. An element $x \in \uupm$ is called primitive if its coproduct has no intermediate terms. The following is a well-known result. \\

\begin{lemma}
\label{lem:primitive}

The only primitive elements of $\uupm$ are $\{x_i^\pm\}_{i \in \BZ/n\BZ}$, $\{p_{\pm k}\}_{k \in \BN}$ and their scalar multiples. \\

\end{lemma} 

\begin{proof} The product of the pairings \eqref{eqn:village} and \eqref{eqn:town} yields a bialgebra pairing:
\begin{equation}
\label{eqn:hamlet}
\uug \otimes \uul \stackrel{\langle \cdot, \cdot \rangle}\longrightarrow \BQ(q)
\end{equation}
such that its restriction to:
\begin{equation}
\label{eqn:restriction 3}
\uup \otimes \uum \stackrel{\langle \cdot, \cdot \rangle}\longrightarrow \BQ(q)
\end{equation}
is non-degenerate (see Remark \ref{rem:non-degenerate}). Thus, let us consider a primitive element: 
$$
a \in \uupm
$$
Because $\Delta(a)$ has no intermediate terms, properties \eqref{eqn:bialg 1} and \eqref{eqn:bialg 2} imply that $a$ pairs trivially with any product of two or more of the generators $\{x_i^\mp\}_{i \in \BZ/n\BZ}$ and $\{p_{\mp k}\}_{k \in \BN}$. By subtracting from $a$ appropriate multiples of $\{x_i^\pm\}_{i \in \BZ/n\BZ}$ and  $\{p_{\pm k}\}_{k \in \BN}$, we may replace the word ``two" in the previous sentence by ``one". Thus, $a$ pairs trivially with $\uump$, hence $a = 0$ by the non-degeneracy of \eqref{eqn:restriction 3}. \\

\end{proof}

\subsection{} \label{sub:primitive} 

Lemma \ref{lem:primitive} implies that $p_{\pm k}$ is the unique, up to constant multiple, sum of products of $e_{\pm [i;j)}$'s which is primitive and has degree $\pm k\bde$. To determine the $p_{\pm k}$ completely, we need to fix this constant multiple. To this end, consider the pairing:
\begin{equation}
\label{eqn:city}
\uug \otimes \uul \stackrel{\langle \cdot, \cdot \rangle}\longrightarrow \BQ(q)
\end{equation}
generated by properties \eqref{eqn:bialg 1}, \eqref{eqn:bialg 2} and the assignments:
\begin{equation}
\label{eqn:pair quantum 1}
\langle \psi_s, \psi_{s'} \rangle = q^{-\delta_{s'}^s}
\end{equation}
\begin{equation}
\label{eqn:pair quantum 2}
\langle e_{[i;j)}, e_{-[i;j)} \rangle = 1-q^{-2}
\end{equation}
and all other pairings among the generators $e_{\pm [i;j)}$ and $\psi_s$ are 0. On general grounds, the pairing \eqref{eqn:city} is the tensor product of the pairings \eqref{eqn:village} and \eqref{eqn:town}, although one would need to modify the latter by replacing the right-hand side of \eqref{eqn:pairing town} by some other non-zero constant which depends on $k$. This has to do with the ambiguity in defining $p_{\pm k}$ up to scalar multiple, which we will now fix. \\

\begin{definition} 
\label{def:correspondence}
	
Fix arbitrary parameters $\pi_+, \pi_-$ such that $\pi_+ \pi_- = q^{-1}$. Let: 
\begin{equation}
\label{eqn:def pk}
p_{\pm k} \in \uupm
\end{equation}
be the unique product of $e_{\pm [i;j)}$'s of degree $\pm k \bde$ which is primitive and satisfies:
\begin{equation}
\label{eqn:normalize}
\Big \langle p_{\pm k}, e_{\mp[u;u+nk)} \Big \rangle = \pm \pi_\pm^{nk}
\end{equation}
for all $u \in \BZ/n\BZ$. \\

\end{definition} 

\noindent Together with \eqref{eqn:simple correspondence 1}--\eqref{eqn:simple correspondence 2}, the Definition above completely determines the correspondence between the two systems of generators \eqref{eqn:collection of generators 1} and \eqref{eqn:collection of generators 2} of $\uu$. The existence of elements $p_{\pm k}$ defined in terms of $e_{\pm [i;j)}$'s which satisfy the conditions of Definition \ref{def:correspondence} was established in \cite{Tor}, where we also proved that:
\begin{equation}
\label{eqn:heisenberg commute}
[p_k, p_l] = k \delta_{k+l}^0 \frac {(c^k-c^{-k})(q^{nk} - q^{-nk})}{(q^k-q^{-k})^2}
\end{equation}

\subsection{} 
\label{sub:order n}

Let $\uu' \subset \uu$ denote the subalgebra generated by all $e_{\pm [i;j)}$ and all ratios \eqref{eqn:psi} and their inverses. It is easy to see that the assignments:
$$
e_{\pm [i;j)} \mapsto e_{\pm [i+1;j+1)}, \qquad \ph_s \mapsto \ph_{s+1}
$$
gives rise to an order $n$ automorphism: 
\begin{equation}
\label{eqn:tau}
\tau : \uu' \rightarrow \uu'
\end{equation} 
Since $\tau$ also preserves the coproduct and the pairing, the uniqueness part of Definition \ref{def:correspondence} implies that $p_{\pm k}$ is invariant under $\tau$. In fact, the uniqueness of primitive elements means that any:
\begin{equation}
\label{eqn:any element}
x^\pm \in \uupm 
\end{equation}
of degree $\notin \{\bs^1,...,\bs^n\}$ is completely determined by the intermediate terms of its coproduct and, if $\deg x^\pm = \pm k\bde$, the extra information of $\langle x^\pm, p_{\mp k} \rangle$. \\

\begin{proposition}
\label{prop:unique}

Suppose we are given elements:
\begin{equation}
\label{eqn:elements}
f_{\pm [i;j)} \in \uupm
\end{equation}
of degree $\pm [i;j)$, for all $(i\leq j) \in \zzz$. Moreover, assume that: \\

\begin{itemize}[leftmargin=*] 
	
\item the elements \eqref{eqn:elements} satisfy \eqref{eqn:cop quant 1}--\eqref{eqn:cop quant 2} with $f$ instead of $e$ \\

\item we have $\tau(f_{\pm [i;j)}) = f_{\pm [i+1;j+1)}$ for all $i<j$ \\

\item we have $f_{\pm [i;i)} = 1$ and $f_{\pm [i;i+1)} = e_{\pm [i;i+1)}$ for all $i$ \\

\end{itemize} 

\noindent Then there exist constants $\alpha_1,\alpha_2,... \in \BQ(q)$ such that:
\begin{equation}
\label{eqn:identity}
f_{\pm [i;j)} = \sum_{k = 0}^{\left \lfloor \frac {j-i}n \right \rfloor} e_{\pm [i;j-nk)} g_{\pm k}
\end{equation}
where $\sum_{k=0}^\infty g_{\pm k} x^k = \exp \left( \sum_{k=1}^\infty \alpha_k p_{\pm k} x^k  \right)$. \\

\end{proposition}

\begin{proof} Let us prove \eqref{eqn:identity} when $\pm = +$, and leave the analogous case where $\pm = -$ as an exercise to the interested reader. As a consequence of \eqref{eqn:cop heis 1}, we have:
\begin{equation}
\label{eqn:cop group-like}
\Delta(g_k) = \sum_{a+b = k} g_a c^b \otimes g_b
\end{equation}
for all $k \in \BN$. Since: 
$$
g_{k} = \alpha_k p_{k} + \Big(\text{a sum of products of more than one }p_l \Big)
$$
one can inductively define the scalars $\alpha_k$ by the condition that:
\begin{equation}
\label{eqn:identity 2}
\langle f_{[i;i+nk)}, p_{-k} \rangle = \left \langle \sum_{l = 0}^{k} e_{ [i;i+n(k-l))} g_{ l}, p_{- k} \right \rangle 
\end{equation} 
for all $i \in \BZ/n\BZ$ (indeed, because the automorphism $\tau$ permutes the elements $e_{[i;j)}$ and $f_{[i;j)}$ and preserves the elements $p_{-k}$ and $g_{l}$ and the pairing, if \eqref{eqn:identity 2} holds for a single $i$, then it holds for all $i$). Let us now prove \eqref{eqn:identity} by induction on $j-i$. The base case holds by the third bullet in the statement of the Proposition. As for the induction step, the first bullet in the statement of the Proposition implies that:
\begin{equation}
\label{eqn:identity 3}
\Delta(\text{LHS of \eqref{eqn:identity}}) = \sum_{s=i}^j f_{[s;j)} \frac {\psi_s}{\psi_i} \otimes f_{[i;s)}
\end{equation}
while \eqref{eqn:cop quant 1} and \eqref{eqn:cop group-like} imply:
$$
\Delta(\text{RHS of \eqref{eqn:identity}}) = \sum_{k = 0}^{\left \lfloor \frac {j-i}n \right \rfloor} \left( \sum_{s=i}^{j-nk}  f_{[s;j-nk)} \frac {\psi_s}{\psi_i} \otimes f_{[i;s)} \right) \left( \sum_{a+b = k}  g_a c^b \otimes g_b \right)
$$
\begin{equation}
\label{eqn:identity 4}
= \sum_{t = i}^j \left( \sum_{a = 0}^{\left \lfloor \frac {j-t}n \right \rfloor}  f_{[t;j-na)} g_a \frac {\psi_t}{\psi_i} \otimes \sum_{b = 0}^{\left \lfloor \frac {t-i}n \right \rfloor}  f_{[i;t-nb)} g_b \right)
\end{equation}
where between the first and second lines of \eqref{eqn:identity 4}, we changed variables according to $t = s + nb$. By the induction hypothesis of \eqref{eqn:identity}, the intermediate terms in the right-hand sides of \eqref{eqn:identity 3} and \eqref{eqn:identity 4} are equal to each other, thus:
$$
a = \text{LHS of \eqref{eqn:identity}} - \text{RHS of \eqref{eqn:identity}} 
$$
is primitive. Since $\deg a = [i;j)$, Lemma \ref{lem:primitive} gives us only two situations when the primitive element $a$ can be non-zero. The first of these is when $j = i+1$, and it cannot happen in the induction step, only in the base case. The second of these is when $j = i+nk$, in which case $a$ is a multiple of $p_k$. However, \eqref{eqn:identity 2} implies that: 
$$
\Big \langle a, p_{-k} \Big \rangle = 0
$$
from which \eqref{eqn:pairing town} yields $a = 0$. 

\end{proof}
 
\subsection{} 
\label{sub:antipodes} 

It is well-known that the bialgebra $\uu$ is in fact a Hopf algebra, and let $S$ be the antipode map. Recall the arbitrary parameters $\pi_+, \pi_-$ such that $\pi_+\pi_- = q^{-1}$ that we fixed in Definition \ref{def:correspondence}. The formula:
\begin{equation}
\label{eqn:e bare}
S^{\pm 1}(e_{\pm [i;j)}) = \frac {\psi^{\pm 1}_i}{\psi^{\pm 1}_j} \be_{\pm [i;j)} \cdot \pi_\mp^{2(j-i)}
\end{equation}
defines elements $\be_{\pm [i;j)} \in \uupm$. Then formulas \eqref{eqn:cop quant 1}--\eqref{eqn:cop quant 2} imply:
\begin{equation}
\label{eqn:identities antipode}
\sum_{s = i}^j \be_{\pm [s;j)} e_{\pm [i;s)} \cdot \pi_{\mp}^{2(i-s)} = 0 
\end{equation}
for all $i<j$. Moreover, we have the following coproduct formulas:
\begin{align}
&\Delta(\be_{[i;j)}) = \sum_{s=i}^j \frac {\psi_j}{\psi_s} \be_{[i;s)} \otimes \be_{[s;j)} \label{eqn:cop antipode 1} \\
&\Delta(\be_{-[i;j)}) = \sum_{s=i}^j \be_{-[s;j)} \otimes \frac {\psi_s}{\psi_j} \be_{-[i;s)} \label{eqn:cop antipode 2}
\end{align}
Moreover, as a consequence of \eqref{eqn:normalize} and  \eqref{eqn:identities antipode}, we note that:
\begin{equation}
\label{eqn:pairing antipode}
\Big \langle p_{\pm k}, \be_{\mp[u;u+nk)} \Big \rangle = \mp \pi_\pm^{-nk} 
\end{equation}
for all $u \in \BZ/n\BZ$. The reason for the formulas above is the fact that the pairing of $\be_{\mp [s;j)} e_{\mp [i;s)}$ with $p_{ \pm k}$ is trivial unless $s \in \{i,j\}$, due to the fact that $p_{\pm k}$ is primitive. \\

\section{The Shuffle Algebra}
\label{sec:shuf}

\noindent The main purpose of the present Section is to study the shuffle algebra $\CA$. \\

\begin{itemize}[leftmargin=*]

\item In Subsection \ref{sub:shufprod}, we introduce the rational functions $\zeta \left(\frac {z_{ia}}{z_{jb}} \right)$ \\

\item In Subsection \ref{sub:defshuf}, we use the aforementioned $\zeta$ to define the shuffle algebra $\CA^+$ \\

\item In Subsection \ref{sub:deg shuf}, we define the grading on $\CA^+$ and $\CA^- = (\CA^+)^{\text{op}}$ \\

\item In Subsection \ref{sub:extended}, we extend the algebras $\CA^+$, $\CA^-$ to $\CA^{\geq}$, $\CA^{\leq}$ \\

\item In Subsection \ref{sub:coproduct}, we define coproducts on the extended algebras $\CA^{\geq}$, $\CA^{\leq}$ \\

\item In Subsection \ref{sub:full}, we define the Drinfeld double $\CA = \CA^{\geq} \otimes \CA^{\leq}$ \\

\item In Subsection \ref{sub:slope}, we define slope subalgebras $\CB^{\pm}_{\mu} \subset \CA^{\pm}$ for any $\mu \in \BQ$ \\

\item In Subsection \ref{sub:diagrams}, we explain how to visualize slope subalgebras in terms of the degrees of their coproducts, using the notion of \underline{hinges} \\

\item In Subsection \ref{sub:double slope}, we construct the Drinfeld doubles $\CB_{\mu} \subset \CA$, recall the fact that they are isomorphic to tensor products of quantum groups in Proposition \ref{prop:sub}, and construct the important elements \eqref{eqn:a}--\eqref{eqn:bb} in $\CB_{\mu}$ \\

\item In Subsection \ref{sub:e's}, we rescale the previously defined elements as in \eqref{eqn:e plus}--\eqref{eqn:be minus}, with the goal of matching them with $e_{\pm [i;j)}$ and $\be_{\pm [i;j)}$ of the previous Section \\

\item In Subsection \ref{sub:alphabeta}, we define the linear maps $\alpha_{\pm [i;j)}$ on $\CA_{\pm [i;j)}$, which will help us fix elements of the shuffle algebra which are a priori determined up to scalar \\

\item In Subsection \ref{sub:alpha pairing}, we show that the restriction of $\alpha_{\pm [i;j)}$ to the slope subalgebras $\CB_\mu$ is given by pairing with the elements \eqref{eqn:e plus} and \eqref{eqn:e minus} \\

\item In Subsection \ref{sub:cop hinge computation}, we compute a part of the coproduct of the elements \eqref{eqn:e plus} \\

\item In Subsection \ref{sub:summarize}, we construct elements of the slope subalgebras $\CB_\mu$ which behave like the primitive generators of quantum groups (under Proposition \ref{prop:sub}) \\

\item In Subsection \ref{sub:cartan}, we perform the analogous treatment for the subalgebra $\CB_\infty$ \\

\end{itemize}

\subsection{}\label{sub:shufprod}

Let $\oq$ be a formal parameter. Consider the following bilinear form on $\zz$:
\begin{equation}
\label{eqn:bilinear form}
\langle \bk, \bl \rangle = \sum_{i=1}^n (k_il_i - k_i l_{i+1})
\end{equation}
(where we identify $l_{n+1} = l_1$) for any $\bk = (k_1,...,k_n)$ and $\bl = (l_1,...,l_n)$. Also let:
\begin{equation}
\label{eqn:length}
|\bk| = k_1+...+k_n
\end{equation}
for any $\bk \in \zz$. We will now recall the type $\widehat{A}_n$ trigonometric version of the shuffle algebra studied in \cite{FO}. For each $i \in \{1,...,n\}$, consider an infinite family of variables $z_{i1},z_{i2},...$. We call $i$ the color of the variable $z_{ia}$, and we call a rational function:
\begin{equation}
\label{eqn:rat func}
R(...,z_{ia},...)^{1 \leq i \leq n}_{1\leq a \leq k_i} 
\end{equation}
color-symmetric if it is symmetric in the variables $z_{i1},...,z_{ik_i}$ for each $i$ separately. Often, we will write explicit formulas for rational functions \eqref{eqn:rat func} that include $z_{ia}$ for any $i \in \BZ$, with the convention that:
\begin{equation}
\label{eqn:identify}
z_{ia} \text{ should be replaced with } z_{\bar{i}a} \oq^{-2\left \lfloor \frac {i-1}n \right \rfloor}
\end{equation}
where $\bar{i}$ is the residue class of $i$ in the set $\{1,...,n\}$. A particular example of this convention is the following color-dependent rational function:
\begin{equation}
\label{eqn:def zeta}
\zeta \left( \frac {z_{ia}}{z_{jb}} \right) = \left( \frac {z_{ia} q \oq^{2\left \lceil \frac {i-j}n \right \rceil}- z_{jb} q^{-1}}{z_{ia} \oq^{2\left \lceil \frac {i-j}n \right \rceil}- z_{jb}} \right)^{\delta_i^j - \delta_{i+1}^j}
\end{equation}
for any variables $z_{ia}$, $z_{jb}$ of colors $i,j \in \BZ$, respectively. Remember that Kronecker $\delta$ symbols are taken mod $n$ in the present paper, unless explicitly stated otherwise. \\

\subsection{} \label{sub:defshuf}

Let $\BF=\BQ(q,\oq^{\frac 1n})$ and consider the set of color-symmetric rational functions:
\begin{equation}
\label{eqn:big}
\CV = \bigoplus_{\bk \in \nn} \BF(...,z_{i1},...,z_{ik_i},...)^{\text{color symmetric}}
_{1 \leq i \leq n} 
\end{equation}
We make the above vector space into a $\BF-$algebra via the shuffle product:
\begin{equation}
\label{eqn:mult}
R(...,z_{i1},...,z_{ik_i},...) * R'(...,z_{i1},...,z_{ik'_i},...) = \frac 1{\bk! \cdot \bk'!} \cdot
\end{equation}
$$
\textrm{Sym} \left[ R(...,z_{i1},...,z_{ik_i},...) R'(...,z_{i,k_i+1},...,z_{i,k_i+k'_i},...) \prod_{i,i'=1}^{n} \prod^{a \leq k_i}_{a' > k'_{i'}} \zeta \left( \frac {z_{ia}}{z_{i'a'}} \right) \right] 
$$
for all rational functions $R$ and $R'$ in $\bk$ and $\bk'$ variables, respectively. In \eqref{eqn:mult}, $\sym$ denotes symmetrization with respect to the: 
\begin{equation}
\label{eqn:deffactorial}
(\bk+\bk')! := \prod_{i=1}^{n} (k_i+k_i')!
\end{equation}
permutations that preserve the color of the variables modulo $n$. \\

\begin{definition}
\label{def:shuf}
	
The shuffle algebra is the subspace $\CA^+ \subset \CV$ of rational functions:
\begin{equation}
\label{eqn:shuf}
R(...,z_{i1},...,z_{ik_i},...) = \frac {r(...,z_{i1},...,z_{ik_i},...)}{\prod_{i=1}^{n} \prod_{1\leq b \leq k_{i+1}}^{1\leq a \leq k_{i}} (z_{ia} q - z_{i+1,b} q^{-1})}
\end{equation}
where $r$ is a color-symmetric Laurent polynomial that satisfies the wheel conditions, i.e. the fact that for all $i\in \{1,...,n\}$ we have:
\begin{equation}
\label{eqn:wheel}
r(...,z_{ia},...) \Big |_{z_{i1} \mapsto w, z_{i2} \mapsto wq^{\pm 2}, z_{i \mp 1,1} \mapsto w} = 0
\end{equation}
	
\end{definition}

\subsection{} 
\label{sub:deg shuf}

It is straightforward to prove that $\CA^+$ is an algebra (see, for example, \cite{Tor}). To a rational function $R(...,z_{i1},...,z_{ik_i},...)$ of homogeneous degree $d$, we may associate its ``horizontal" and ``vertical" degrees, as follows:
\begin{align}
&\hdeg R = (k_1,...,k_n) \in \nn \label{eqn:hdeg} \\
&\vdeg R = d \in \BZ \label{eqn:vdeg}
\end{align}
Thus, the algebra $\CA^+$ is graded by $\nn \times \BZ$, and we will denote its graded pieces by:
\begin{align}
&\CA^+ = \bigoplus_{\bk \in \nn} \CA_\bk \label{eqn:a plus hdeg} \\
&\CA_\bk = \bigoplus_{d\in \BZ} \CA_{\bk,d} \label{eqn:a plus deg}
\end{align}
Let: 
\begin{equation}
\label{eqn:opposite shuffle}
\CA^- = \left(\CA^+\right)^{\text{op}}
\end{equation} 
so a rational function $R$ as in \eqref{eqn:shuf} may be regarded as either:
$$
R^+ \in \CA^+ \quad \text{or} \quad R^- \in \CA^-
$$ 
If $R^+$ has degree $(\bk,d)$, then we assign $R^-$ degree $(-\bk,d)$ and write:
\begin{align}
&\CA^- = \bigoplus_{\bk \in \nn} \CA_{-\bk} \label{eqn:a minus hdeg} \\
&\CA_{-\bk} = \bigoplus_{d\in \BZ} \CA_{-\bk,d} \label{eqn:a minus deg}
\end{align}
The \underline{naive slope} of an element $R^\pm \in \CA^\pm$ is defined as:
\begin{equation}
\label{eqn:naive slope}
\frac {\text{vdeg }R}{|\text{hdeg }R|}
\end{equation}

\subsection{}\label{sub:extended}

Define the extended shuffle algebras as:
\begin{align}
&\CA^\geq = \Big \langle \CA^+ , \psi_s^{\pm 1}, c^{\pm 1}, \barc^{\pm 1}, a_{s,1}, a_{s,2},... \Big \rangle_{s \in \{1,...,n\}} \label{eqn:a geq} \\
&\CA^\leq = \Big \langle \CA^- , \psi_s^{\pm 1}, c^{\pm 1}, \barc^{\pm 1}, a_{s,-1}, a_{s,-2},... \Big \rangle_{s \in \{1,...,n\}} \label{eqn:a leq}
\end{align}
modulo the fact that $c$ and $\barc$ are central, and the following relations:
\begin{equation}
\label{eqn:vertical 1}
[\psi_s, \psi_{s'}] = 0, \quad [a_{s,\pm d}, \psi_{s'}] = 0, \quad [a_{s,\pm d}, a_{s', \pm d'}] = 0
\end{equation}
\begin{equation}
\label{eqn:vertical 2}
\psi_s R^\pm  = q^{\pm(k_{s-1} - k_s)} R^{\pm}  \psi_s 
\end{equation}
\begin{align}
&[a_{s,d}, R^{+}] = R^+ \left( \sum_{t=1}^{k_{s-1}} z_{s-1,t}^{d} - \sum_{t=1}^{k_s} z_{st}^{d} \right) \label{eqn:vertical 3} \\
&[a_{s,-d}, R^{-}] = R^- \left(q^{-d} \sum_{t=1}^{k_{s-1}} z_{s-1,t}^{-d} -  q^d \sum_{t=1}^{k_s} z_{st}^{-d} \right) \label{eqn:vertical 4}
\end{align}
for all $s,s' \in \{1,...,n\}$, $d, d' > 0$ and $R^\pm (...,z_{i1},...,z_{ik_i},...)\in \CA^\pm$. In all formulas above and henceforth, we extend the indexing set of the $\psi$'s and the $a$'s by setting:
\begin{equation}
\label{eqn:extend}
\psi_{s+n} = c \psi_s, \qquad a_{s+n,d} = a_{s,d} \oq^{-2 d}
\end{equation}
for all $s,d \in \BZ$. \\

\begin{remark}
	
The algebras \eqref{eqn:a geq} and \eqref{eqn:a leq} were defined in \cite{Tor}, albeit in different notation. Explicitly, the generators $a_{s,\pm d}$ were packaged in \loccit as:
\begin{align}
&\ph_s^+(z) = \sum_{d=0}^\infty \frac {\ph_{s,d}}{z^d} = \frac {\psi_{s+1}}{\psi_s} \exp \left[\sum_{d=1}^\infty \frac {(q^d a_{s,d} - q^{-d} a_{s+1,d})(q^{-d}-q^{d})}{d z^{d}} \right] \label{eqn:convert 1} \\ 
&\ph_s^-(z) = \sum_{d=0}^\infty \frac {\ph_{s,-d}}{z^{-d}} = \frac {\psi_s}{\psi_{s+1}} \exp \left[\sum_{d=1}^\infty \frac {(a_{s,-d} - a_{s+1,-d})(q^{-d}-q^{d})}{d z^{-d}} \right] \label{eqn:convert 2} 
\end{align}
The discussion in \loccit had set $\barc=1$, but the general case is analogous. \\
	
\end{remark}

\subsection{} 
\label{sub:coproduct}

One of the main reasons for defining the extended shuffle algebras $\CA^\geq$ and $\CA^\leq$ is that they admit topological coproducts, as defined in \cite{Tor}: \footnote{Note that only the $\barc = 1$ case of the following formulas was defined in \loccitt, but the general definition below is analogous. In fact, we are simply modifying the coproduct of \loccit by the central element $\barc$ raised to the power equal to the vertical degree in one of the tensor factors, so all properties proved in \loccit (coassociativity, compatibility of product and coproduct etc) carry over to the present situation.}
\begin{equation}
\label{eqn:coproduct 0}
\Delta(\psi_s) = \psi_s \otimes \psi_s, \quad \Delta(c) = c \otimes c, \quad \Delta(\barc) = \barc \otimes \barc 
\end{equation}
\begin{equation}
\label{eqn:coproduct 1}
\Delta(a_{s,d}) = \barc^d \otimes a_{s,d} + a_{s,d} \otimes 1
\end{equation}
\begin{equation}
\label{eqn:coproduct 2}
\Delta(a_{s,-d}) = 1 \otimes a_{s,-d} + a_{s,-d} \otimes \barc^{-d}
\end{equation}
for all $s \in \{1,...,n\}$ and $d>0$, as well as (let $\barc_1 = \barc \otimes 1$ and $\barc_2 = 1 \otimes \barc$):
\begin{align}
&\Delta(R^+) = \sum_{\bl \in \nn}^{\bl\leq \bk}  \frac {\left[ \prod_{1 \leq i \leq n}^{a>l_i} \ph^+_i(z_{ia}\barc_1) \otimes 1 \right] R^+(z_{i, a \leq l_i} \otimes z_{i, a>l_i} \barc_1)}{\prod_{1 \leq i' \leq n}^{1 \leq i \leq n} \prod^{a \leq l_{i}}_{a' > l_{i'}} \zeta(z_{i'a'}\barc_1/z_{ia})} \label{eqn:coproduct 3} \\
&\Delta(R^-) = \sum_{\bl \in \nn}^{\bl\leq \bk}  \frac {R^-(z_{i, a \leq l_i}\barc_2 \otimes z_{i, a>l_i}) \left[ \prod_{1 \leq i \leq n}^{a \leq l_i} 1 \otimes \ph^-_i(z_{ia}\barc_2) \right]}{\prod_{1 \leq i' \leq n}^{1 \leq i \leq n} \prod^{a \leq l_{i}}_{a' > l_{i'}} \zeta(z_{ia} \barc_2/z_{i'a'})} \label{eqn:coproduct 4}
\end{align}
for all $R^\pm \in \CA_{\pm \bk}$. To think of \eqref{eqn:coproduct 3} as a tensor, we expand the right-hand side in non-negative powers of $z_{ia} / z_{i'a'}$ for $a\leq l_i$ and $a'>l_{i'}$, thus obtaining an infinite sum of monomials. In each of these monomials, we put the symbols $\ph_{i,d}$ to the very left of the expression, then all powers of $z_{ia}$ with $a\leq l_i$, then the $\otimes$ sign, and finally all powers of $z_{ia}$ with $a>l_i$. The powers of the central element $\barc_1 = c \otimes 1$ are placed in the first tensor factor. The resulting expression will be a power series, and therefore lies in a completion of $\CA^\geq \otimes \CA^\geq$. The same argument applies to \eqref{eqn:coproduct 4}. \\

\subsection{}\label{sub:full}

Finally, we showed in \cite{Tor} that there exists a bialgebra pairing:
\begin{equation}
\label{eqn:pairshuf}
\CA^\geq \otimes \CA^\leq \stackrel{\langle \cdot, \cdot \rangle}\longrightarrow \BF 
\end{equation}
given by:
\begin{equation}
\label{eqn:pairshuf0}
\langle \psi_s, \psi_{s'} \rangle = q^{- \delta_{s'}^s} , \qquad \langle a_{s,d}, a_{s',-d'} \rangle = \frac {\delta_{s'}^s \delta_{d'}^d d}{q^{d}-q^{-d}} 
\end{equation}
for $s,s' \in \{1,...,n\}$ and $d,d' \in \BN$, as well as:
\begin{equation}
\label{eqn:pairshuf1}
\left \langle R^+,R^- \right \rangle = \frac {(1-q^{-2})^{|\bk|}}{\bk!} \oint \frac {R^+(...,z_{ia},...)R^-(...,z_{ia},...)}{\prod_{i,j=1}^{n} \prod^{(i,a) \neq (j,b)}_{a\leq k_i, b \leq k_j} \zeta(z_{ia}/z_{jb})} \prod^{1 \leq i \leq n}_{1 \leq a \leq k_i} \frac {dz_{ia}}{2\pi i z_{ia}} \qquad 
\end{equation}
for any $R^+ \in \CA_{\bk}$ and $R^- \in \CA_{-\bk}$ (all other pairings are 0). In formula \eqref{eqn:pairshuf1}, the contour integral is set up in such a way that the variable $z_{ia}$ goes over a contour which surrounds only the poles at $z_{ib} q^2$, $z_{i-1,b}$, $z_{i+1,b} q^{-2}$ for all $i \in \{1,...,n\}$ and all $a,b$ (a particular choice of contours which achieves this is explained in Proposition 3.9 of \cite{Tor}). The pairing \eqref{eqn:pairshuf} allows us to construct the Drinfeld double:
\begin{equation}
\label{eqn:hmm}
\CA = \CA^\geq \otimes \CA^\leq \Big / (c \otimes 1 - 1 \otimes c, \barc \otimes 1  - 1 \otimes \barc, \psi_s \otimes 1 - 1 \otimes \psi_s)
\end{equation}
The main goal of setting up the above double shuffle algebra is the following: \\

\begin{theorem}
\label{thm:iso}
	
(\cite{Tor}) There is an isomorphism of bialgebras $\UU \cong \CA$. \footnote{The interested reader may find the definition of $\UU$ in \cite[Subsection 2.25]{Tor}, but we will not need it in the present paper.}\\
	
\end{theorem}

\noindent The construction of an algebra homomorphism $\UU \rightarrow \CA$ goes back to work of Enriquez (\cite{E}), so the main technical point of Theorem \ref{thm:iso} is that this homomorphism is surjective. In other words, we prove that the Feigin-Odesskii wheel conditions \eqref{eqn:wheel} are sufficient for describing the quantum toroidal algebra. \\

\begin{proposition}
\label{eqn:prop:double}
	
We have the following commutation relations in $\CA$:
\begin{equation}
\label{eqn:double 1}
[a_{s,d}, a_{s',d'}] = \delta_{s'}^s \delta_{d+d'}^0 d \frac {\barc^d - \barc^{-d}}{q^{-d}-q^{d}}
\end{equation}
\begin{align}
&[a_{s,d}, R^-] \ = R^- \barc^d \left(\sum_{t=1}^{k_s} z_{st}^{d} - \sum_{t=1}^{k_{s-1}} z_{s-1,t}^{d} \right) \label{eqn:double 2} \\ 
&[a_{s,-d}, R^+] = R^+ \barc^{-d} \left(q^d \sum_{t=1}^{k_s} z_{st}^{-d} - q^{-d} \sum_{t=1}^{k_{s-1}} z_{s-1,t}^{-d} \right) \label{eqn:double 3} 
\end{align}
for all $s,s' \in \{1,...,n\}$, $d, d' > 0$ and $R^\pm (...,z_{i1},...,z_{ik_i},...)\in \CA^\pm$. Moreover: 
\begin{equation}
\label{eqn:moreover}
\left[ (z_{i1}^{k} )^+, (z_{j1}^{-k'})^{-} \right] = \delta_j^i (q^{-2} - 1) \begin{cases} \ph_{i,k-k'} \barc^{k'} &\text{if } k>k' \\  \ph_{i,0} \barc^{k'} - \ph_{i,0}^{-1} \barc^{-k} &\text{if } k = k' \\ - \ph_{i,k-k'} \barc^{-k} &\text{if } k<k' \end{cases} 
\end{equation}
for all $i,j \in \{1,...,n\}$, where $\ph_{s,\pm d}$ are the series coefficients of $\ph_s^\pm$ of \eqref{eqn:convert 1}--\eqref{eqn:convert 2}. \\
	
\end{proposition}

\begin{proof} Let us apply \eqref{eqn:drinfeld} for $a = a_{s,d}$ and $b = a_{s',-d'}$:
$$
\barc^d \langle a_{s,d}, a_{s',-d'} \rangle + a_{s,d} a_{s',-d'} \langle 1, \barc^{-d'} \rangle = \langle \barc^d, 1 \rangle a_{s',-d'} a_{s,d} + \langle a_{s,d}, a_{s',-d'} \rangle \barc^{-d}
$$	
Then \eqref{eqn:pairshuf0} implies \eqref{eqn:double 1}. As for \eqref{eqn:double 2}, note that \eqref{eqn:coproduct 4} implies:
$$
\Delta(R^-) = 1 \otimes R^-(...,z_{i1},...,z_{ik_i},...)  + ... + R^-(...,z_{i1}\barc_2,...,z_{ik_i}\barc_2,...) \otimes \prod_{i=1}^n \prod_{a=1}^{k_i} \ph_i^-(z_{ij} \barc_2)
$$
where $...$ in the middle stand for terms which pair trivially with $a_{s,d}$ or powers of $\barc$, for degree reasons. Therefore, applying \eqref{eqn:drinfeld} for $a=a_{s,d}$ and $b = R^-$ gives us:
\begin{multline*}
\barc^d R^-(...,z_{ij}\barc_2,...) \left\langle a_{s,d}, \prod_{i=1}^n \prod_{a=1}^{k_i} \ph_i^-(z_{ij} \barc_2) \right\rangle + \\ + a_{s,d} R^-(...,z_{ij}\barc_2,...) \left \langle 1, \prod_{i=1}^n \prod_{a=1}^{k_i} \ph_i^-(z_{ij} \barc_2) \right\rangle = \langle \barc^d, 1 \rangle R^- a_{s,d}
\end{multline*}
Since $\barc$ is group-like and pairs trivially with anything, the terms denoted by $\barc_2$ in the formula above do not change the values of any of the pairings. Moreover, the pairings on the second line are both equal to 1. As for the pairing on the first line, a simple consequence of \eqref{eqn:convert 2} and \eqref{eqn:coproduct 1} is that
$$
\left\langle a_{s,d}, \prod_{i=1}^n \prod_{a=1}^{k_i} \ph_i^-(z_{ia}) \right\rangle = \sum_{d'=1}^\infty \sum_{i=1}^n \sum_{j=1}^{k_i} \frac {\langle a_{s,d},  a_{i,-d'} - a_{i+1,-d'}\rangle (q^{-d'}-q^{d'})}{d' z_{ij}^{-d'}} = 
$$
$$
= \sum_{j=1}^{k_{s-1}} z_{s-1,j}^d - \sum_{j=1}^{k_s} z_{sj}^d 
$$
which proves \eqref{eqn:double 2}. Formula \eqref{eqn:double 3} is proved analogously, so we leave it as an exercise. Relation \eqref{eqn:moreover} is straightforward, and we leave to the interested reader (a proof was also given in \cite{Tor}, where this relation was compared to the defining commutation relation between the two halves of the quantum toroidal algebra). 
	
\end{proof}

\subsection{}\label{sub:slope}

Let us now recall the factorization of $\A$ into slope subalgebras from \cite{Tor}:
\begin{equation}
\label{eqn:slope decomposition}
\A^\pm = \bigotimes_{\mu \in \BQ} \B^\pm_\mu
\end{equation}
where the product is taken in increasing order of $\mu$. We will now describe $\B_\mu^\pm$. \\

\begin{definition}
\label{def:slope}	
	
For any shuffle element $R^\pm\in \CA_{\pm \bk}$ and any $\mu \in \BQ$, if the limit:
\begin{equation}
\label{eqn:limit}
\lim_{\xi^{\pm 1} \rightarrow \infty} \frac {R^\pm(..., \xi z_{i1},..., \xi z_{i l_i}, z_{i, l_i+1},..., z_{ik_i},...)}{\xi^{\pm \mu |\bl|}}
\end{equation}
exists and is finite for all $\bl = (l_1,...,l_n)$, then we say that $R^\pm$ has \underline{slope} $\leq \mu$. \\
	
\end{definition}

\noindent It is elementary to show that the subspace of $\A^\pm$ of elements of slope $\leq \mu$ is a subalgebra (see \cite{Tor} for a proof), and we will denote it by:
\begin{equation}
\label{eqn:slope notation}
\A_{\leq \mu}^\pm \subset \A^\pm
\end{equation}
and its graded pieces by:
\begin{align}
&\A_{\leq \mu | \pm \bk} = \A_{\pm \bk} \cap \A^\pm_{\leq \mu} \label{eqn:slope notation hdeg} \\
&\A_{\leq \mu | \pm \bk,d} = \A_{\pm \bk,d} \cap \A^\pm_{\leq \mu} \label{eqn:slope notation deg}
\end{align}
Let us now connect the above slope property with the coproduct. We have:
\begin{equation}
\label{eqn:radu1}
\Delta(R^+) = \Delta_\mu(R^+) + (\textrm{anything}) \otimes (\text{naive slope} < \mu) 
\end{equation}
\begin{equation}
\label{eqn:radu2}
\Delta(R^-) = \Delta_\mu(R^-) + (\text{naive slope} < \mu) \otimes (\textrm{anything})
\end{equation}
for any $R^\pm \in \CA^\pm_{\leq \mu}$, where the leading terms $\Delta_\mu$ are defined by:
\begin{align}
&\Delta_\mu(R^+) = \sum^{\bl \in \nn}_{\bl\leq \bk} (\ph_{\bk-\bl} \otimes 1) \lim_{\xi \rightarrow \infty} \frac {R^+(z_{i,a\leq l_i} \otimes \xi\cdot z_{i,a>l_i} \barc_1 )}{\xi^{\mu|\bk - \bl|} q^{\langle \bk - \bl, \bl \rangle}} \label{eqn:cristian1} \\
&\Delta_\mu(R^-) = \sum^{\bl \in \nn}_{\bl\leq \bk} \lim_{\xi \rightarrow 0} \frac {R^-(\xi \cdot z_{i,a\leq l_i} \barc_2 \otimes z_{i,a>l_i})}{\xi^{-\mu|\bl|} q^{- \langle \bl, \bk - \bl \rangle}} (1 \otimes \ph_{-\bl}) \label{eqn:cristian2}
\end{align}
where $\ph_\bk = \prod_{i=1}^n \ph_i^{k_i}$ for all $\bk = (k_1,...,k_n)$, and $\ph_i = \frac {\psi_{i+1}}{\psi_i}$. The vector spaces:
\begin{equation}
\label{eqn:sub}
\CB^\pm_\mu = \bigoplus^{\mu|\bk| \in \BZ}_{\bk \in \nn} \CB_{\mu|\pm \bk} := \bigoplus^{\mu|\bk| = d}_{\bk \in \nn} \CA_{\leq \mu|\pm \bk, \pm d} \subset \CA^\pm
\end{equation}
are subalgebras, and their extended versions:
\begin{align}
&\CB^\geq_\mu = \Big \langle \CB^+_\mu, \psi_s^{\pm 1}, c^{\pm 1}, \barc^{\pm 1} \Big \rangle_{s \in \{1,...,n\}} \subset \CA^\geq  \label{eqn:sub1} \\
&\CB^\leq_\mu = \Big \langle \CB^-_\mu, \psi_s^{\pm 1}, c^{\pm 1}, \barc^{\pm 1} \Big \rangle_{s \in \{1,...,n\}} \subset \CA^\leq \label{eqn:sub2}
\end{align}
are bialgebras with respect to the coproducts \eqref{eqn:cristian1}, \eqref{eqn:cristian2}, respectively. Therefore, the decomposition \eqref{eqn:slope decomposition} does not preserve the bialgebra structure, but instead the coproduct of $x \in \B^\pm_\mu$ is the leading order term of the coproduct of $x \in \A^\pm$. 

\subsection{} 
\label{sub:diagrams} 

We may visualize the coproduct $\Delta$ as follows. To a shuffle element $R^\pm \in \CA_{\pm \bk, d}$, we associate the lattice point/vector $(\pm |\bk|, d)$. For any shuffle element $R^+ \in \CA^+$, consider its coproduct (in Sweedler notation):
\begin{equation}
\label{eqn:sweedler plus}
\Delta(R^+) = R^+_1 \otimes R^+_2 
\end{equation}
In any tensor $R^+_1 \otimes R^+_2$ that appears as a summand in $\Delta(R^+)$, the vectors associated to $R^+_1$ and $R^+_2$ are as in the picture below:

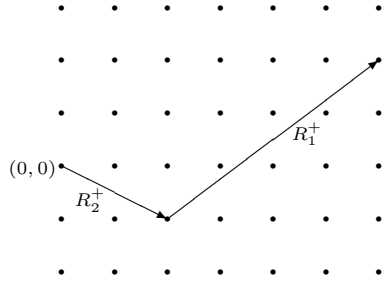
\begin{picture}(100,130)(-110,-35)

\put(0,-20){\circle*{2}}\put(20,-20){\circle*{2}}\put(40,-20){\circle*{2}}\put(60,-20){\circle*{2}}\put(80,-20){\circle*{2}}\put(100,-20){\circle*{2}}\put(120,-20){\circle*{2}}\put(0,0){\circle*{2}}\put(20,0){\circle*{2}}\put(40,0){\circle*{2}}\put(60,0){\circle*{2}}\put(80,0){\circle*{2}}\put(100,0){\circle*{2}}\put(120,0){\circle*{2}}\put(0,20){\circle*{2}}\put(20,20){\circle*{2}}\put(40,20){\circle*{2}}\put(60,20){\circle*{2}}\put(80,20){\circle*{2}}\put(100,20){\circle*{2}}\put(120,20){\circle*{2}}\put(0,40){\circle*{2}}\put(20,40){\circle*{2}}\put(40,40){\circle*{2}}\put(60,40){\circle*{2}}\put(80,40){\circle*{2}}\put(100,40){\circle*{2}}\put(120,40){\circle*{2}}\put(0,60){\circle*{2}}\put(20,60){\circle*{2}}\put(40,60){\circle*{2}}\put(60,60){\circle*{2}}\put(80,60){\circle*{2}}\put(100,60){\circle*{2}}\put(120,60){\circle*{2}}\put(0,80){\circle*{2}}\put(20,80){\circle*{2}}\put(40,80){\circle*{2}}\put(60,80){\circle*{2}}\put(80,80){\circle*{2}}\put(100,80){\circle*{2}}\put(120,80){\circle*{2}}

\put(0,20){\vector(2,-1){40}}
\put(40,0){\vector(4,3){80}}

\put(-20,17){\scriptsize{$(0,0)$}}
\put(87,30){\scriptsize{$R^+_1$}}
\put(5,5){\scriptsize{$R^+_2$}}

\end{picture}

\noindent The point where the arrows meet will be called the \underline{hinge} of the tensor $R_1^+ \otimes R_2^+$. Then a shuffle element $R^+$ has slope $\leq \mu$ if and only if all the summands in $\Delta(R^+)$ have hinge at slope $\leq \mu$ measured from the origin. \\

\noindent Similarly, consider any shuffle element $R^- \in \CA^-$ and its coproduct:
\begin{equation}
\label{eqn:sweedler minus}
\Delta(R^-) = R^-_1 \otimes R^-_2 
\end{equation}
In any tensor $R^-_1 \otimes R^-_2$ that appears as a summand in $\Delta(R^-)$, the vectors associated to $R^-_1$ and $R^-_2$ are as in the picture below:

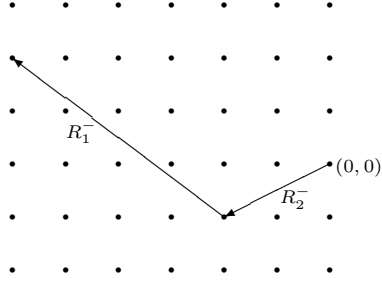
\begin{picture}(100,130)(-110,-35)

\put(0,-20){\circle*{2}}\put(20,-20){\circle*{2}}\put(40,-20){\circle*{2}}\put(60,-20){\circle*{2}}\put(80,-20){\circle*{2}}\put(100,-20){\circle*{2}}\put(120,-20){\circle*{2}}\put(0,0){\circle*{2}}\put(20,0){\circle*{2}}\put(40,0){\circle*{2}}\put(60,0){\circle*{2}}\put(80,0){\circle*{2}}\put(100,0){\circle*{2}}\put(120,0){\circle*{2}}\put(0,20){\circle*{2}}\put(20,20){\circle*{2}}\put(40,20){\circle*{2}}\put(60,20){\circle*{2}}\put(80,20){\circle*{2}}\put(100,20){\circle*{2}}\put(120,20){\circle*{2}}\put(0,40){\circle*{2}}\put(20,40){\circle*{2}}\put(40,40){\circle*{2}}\put(60,40){\circle*{2}}\put(80,40){\circle*{2}}\put(100,40){\circle*{2}}\put(120,40){\circle*{2}}\put(0,60){\circle*{2}}\put(20,60){\circle*{2}}\put(40,60){\circle*{2}}\put(60,60){\circle*{2}}\put(80,60){\circle*{2}}\put(100,60){\circle*{2}}\put(120,60){\circle*{2}}\put(0,80){\circle*{2}}\put(20,80){\circle*{2}}\put(40,80){\circle*{2}}\put(60,80){\circle*{2}}\put(80,80){\circle*{2}}\put(100,80){\circle*{2}}\put(120,80){\circle*{2}}

\put(120,20){\vector(-2,-1){40}}
\put(80,0){\vector(-4,3){80}}

\put(122,17){\scriptsize{$(0,0)$}}
\put(20,30){\scriptsize{$R^-_1$}}
\put(101,5){\scriptsize{$R^-_2$}}

\end{picture}

\noindent The point where the arrows meet will be called the \underline{hinge} of the tensor $R_1^- \otimes R_2^-$. Then a shuffle element $R^-$ has slope $\leq \mu$ if and only if all the summands in $\Delta(R^-)$ have hinge at slope $\geq \mu$ measured from the origin. \\

\subsection{} 
\label{sub:double slope} 

The bialgebra pairing between $\A^\geq$ and $\A^\leq$ restricts to a bialgebra pairing:
\begin{equation}
\label{eqn:pair b mu}
\CB^\geq_\mu \otimes \CB^\leq_\mu \stackrel{\langle \cdot , \cdot \rangle}\longrightarrow \BF
\end{equation}
and the Drinfeld double: 
\begin{equation}
\label{eqn:double sub}
\CB_\mu := \CB^\geq_\mu \otimes \CB^\leq_\mu \Big / (c \otimes 1 - 1 \otimes c, \barc \otimes 1  - 1 \otimes \barc, \psi_s \otimes 1 - 1 \otimes \psi_s)
\end{equation}
corresponding to the above data yields a homomorphism:
\begin{equation}
\label{eqn:subalgebra}
\CB_\mu \subset \CA
\end{equation}
as algebras. \\

\begin{proposition}
\label{prop:sub}
	
For any coprime $a \in \BN$ and $b \in \BZ$, we have an isomorphism:
\begin{equation}
\label{eqn:smalliso}
\Xi: U_{q}(\dot{\fgl}_{\frac ng})^{\otimes g} \stackrel{\sim}\longrightarrow \CB_{\frac ba}
\end{equation}
where $g = \gcd(n,a)$. The isomorphism $\Xi$ preserves the bialgebra structures. \\
	
\end{proposition}

\noindent We will now construct explicit elements of $\B^\pm_\mu$ which correspond under $\Xi$ to the generators \eqref{eqn:collection of generators 1} of quantum affine algebras. These elements were initially constructed in \cite{Tor} under the names $E$ and $F$, but we will relabel them as follows:
\begin{equation}
\label{eqn:a}
A_{\pm [i;j)}^{\mu} = \sym \left[ \frac {\prod_{a=i}^{j-1} (z_a \oq^{\frac {2a}n})^{\lceil \pm \mu(a-i+1) \rceil - \lceil \pm \mu(a-i) \rceil}}{\left(1 - \frac {z_{i}q^2}{z_{i+1}}  \right) ... \left(1 - \frac {z_{j-2}q^2}{z_{j-1}} \right)} \prod_{i\leq a < b < j} \zeta \left( \frac {z_b}{z_a} \right)  \right]^\pm  
\end{equation}
\begin{equation}
\label{eqn:ba}
\bA_{\pm [i;j)}^{\mu} = \sym \left[ \frac {\prod_{a=i}^{j-1} (z_a \oq^{\frac {2a}n})^{\lfloor \pm \mu(a-i+1) \rfloor - \lfloor \pm \mu(a-i) \rfloor}}{\left(1 - \frac {z_{i+1}}{z_{i}q^2}  \right) ... \left(1 - \frac {z_{j-1}}{z_{j-2}q^2} \right)} \prod_{i\leq a < b < j} \zeta \left( \frac {z_b}{z_a} \right)  \right]^\pm 
\end{equation}
\begin{equation}
\label{eqn:b}
B_{\pm [i;j)}^{\mu} = \sym \left[ \frac {\prod_{a=i}^{j-1} (z_a \oq^{\frac {2a}n})^{\lfloor \pm \mu(a-i+1) \rfloor - \lfloor \pm \mu(a-i) \rfloor}}{\left(1 - \frac {z_{i+1}}{z_{i}}  \right) ... \left(1 - \frac {z_{j-1}}{z_{j-2}} \right)} \prod_{i\leq a < b < j} \zeta \left( \frac {z_a}{z_b} \right) \right]^\pm 
\end{equation} 
\begin{equation}
\label{eqn:bb}
\bB_{\pm [i;j)}^{\mu} = \sym \left[ \frac {\prod_{a=i}^{j-1} (z_a \oq^{\frac {2a}n})^{\lceil \pm \mu(a-i+1) \rceil - \lceil \pm \mu(a-i) \rceil}}{\left(1 - \frac {z_{i}}{z_{i+1}}  \right) ... \left(1 - \frac {z_{j-2}}{z_{j-1}} \right)} \prod_{i\leq a < b < j} \zeta \left( \frac {z_a}{z_b} \right)  \right]^\pm 
\end{equation}
for all $(i<j) \in \BZ^2/(n,n)\BZ$ and $\mu \in \BQ$ such that $k = \mu(j-i) \in \BZ$. The superscript $+$ or $-$ indicates that the corresponding rational functions will be considered elements of either $\CA^+$ or $\CA^-$. We will often write:
\begin{equation}
\label{eqn:often write}
A_{\pm [i;j)}^{(\pm k)} = A_{\pm [i;j)}^{\mu} 
\end{equation}
to emphasize the fact that $\deg A_{\pm [i;j)}^{(\pm k)} = (\pm [i;j), \pm k)$. Moreover, we set:
\begin{equation}
\label{eqn:often set}
A_{\pm [i;j)}^{\mu} = 0
\end{equation}
if $\mu(j-i) \notin \BZ$. The analogous notations apply to $\bA$, $B$ and $\bB$ instead of $A$. \\

\begin{remark} In order to think of the right-hand sides of the expressions \eqref{eqn:a}--\eqref{eqn:bb} as shuffle elements, we regard each $z_a$ as a variable of color $a$ for all $a \in \{i,...,j-1\}$. By convention \eqref{eqn:identify}, this entails replacing:
\begin{equation}
\label{eqn:replacing}
z_{a} \ \leadsto \ z_{\bar{a}, \left \lceil \frac {a-i+1}n \right \rceil} \oq^{-2\left \lfloor \frac {a-1}n \right \rfloor}
\end{equation}
for all $a$, which puts $A_{\pm [i;j)}^{\mu}, \bA_{\pm [i;j)}^{\mu}, B_{\pm [i;j)}^{\mu}, \bB_{\pm [i;j)}^{\mu}$ in the form \eqref{eqn:shuf}. \\
	
\end{remark} 

\begin{proposition} 
\label{prop:belong}
	
(\cite{Tor}) For any $\mu \in \BQ$, the elements \eqref{eqn:a}--\eqref{eqn:bb} lie in $\CB^\pm_\mu$. \\

\end{proposition} 

\subsection{} 
\label{sub:e's}

For any $i < j$ and $\mu \in \BQ$, consider the following elements of $\A^\pm$:
\begin{align}
&E_{[i;j)}^{\mu} = \bA^\mu_{[i;j)} \label{eqn:e plus} \\
&\bE_{[i;j)}^\mu = A^\mu_{[i;j)} \cdot (-\oq_-^{\frac 2n})^{i-j} \label{eqn:be plus} \\
&E_{-[i;j)}^{\mu} = \bB^\mu_{-[i;j)} \label{eqn:e minus} \\
&\bE_{-[i;j)}^{\mu} = B^\mu_{-[i;j)} \cdot (-\oq_+^{\frac 2n})^{i-j} \label{eqn:be minus}
\end{align}
where:
\begin{equation}
\label{eqn:parameters}
\oq_+ = \oq \quad \text{and} \quad \oq_- = q^{-n} \oq^{-1}
\end{equation}
We will henceforth take the liberty to use the notations \eqref{eqn:often write}--\eqref{eqn:often set} with $E$ and $\bE$ instead of $A$. We have the following formulas for the coproduct $\Delta_\mu$: 
\begin{align}
&\Delta_\mu \left(E_{[i;j)}^\mu \right) = \sum_{s=i}^j E_{[s;j)}^\mu \frac {\psi_s}{\psi_i} \barc^{\mu(s-i)} \otimes E_{[i;s)}^\mu \label{eqn:cop1} \\
&\Delta_\mu \left(\bE_{[i;j)}^\mu \right) = \sum_{s=i}^j \frac {\psi_j}{\psi_s} \barc^{\mu(j-s)} \bE_{[i;s)}^\mu \otimes \bE_{[s;j)}^\mu  \label{eqn:cop1 anti} \\  
&\Delta_\mu \left(E_{-[i;j)}^\mu \right) = \sum_{s=i}^j E_{-[i;s)}^\mu \otimes E_{-[s;j)}^\mu \frac {\psi_i}{\psi_s} \barc^{\mu(i-s)} \label{eqn:cop2} \\
&\Delta_\mu \left(\bE_{-[i;j)}^\mu \right) = \sum_{s=i}^j \bE_{-[s;j)}^\mu \otimes \frac {\psi_s}{\psi_j} \barc^{\mu(s-j)} \bE_{-[i;s)}^\mu \label{eqn:cop2 anti} 
\end{align}
Formulas \eqref{eqn:cop1} and \eqref{eqn:cop2} were proved in \cite{Tor} when $\barc = 1$, but the general $\barc$ situation is analogous. Formulas \eqref{eqn:cop1 anti} and \eqref{eqn:cop2 anti} are proved analogously, so we leave them as exercises (alternatively, they follow from \eqref{eqn:cop1}, \eqref{eqn:cop2} and \eqref{eqn:antipode}). \\

\begin{proposition}
\label{prop:antipode}
	
The elements $E_{\pm [i;j)}^{\mu}$ and $\bE_{\pm [i;j)}^{\mu}$ are related by:
\begin{equation}
\label{eqn:antipode}
\sum_{s=i}^j \bE_{\pm [s;j)}^\mu E_{\pm [i;s)}^\mu \oq_\mp^{\frac {2(i-s)}{nh}}  = 0
\end{equation}
for all $i < j$, where we write $h = \frac {j-i}{\gcd(\mu(j-i),j-i)}$. \\
	
\end{proposition}

\begin{proof} Since $\lceil x \rceil = \lfloor x \rfloor + 1 - \delta_{x \in \BZ}$, it is elementary to see that:
\begin{equation}
\label{eqn:elementary}
\bE_{[i;j)}^\mu = - q^2 \oq^{\frac 2n} \cdot
\end{equation}
$$
\sym \left[ \frac {\prod^{\mu(s-i) \in \BZ}_{i < s < j} \frac {z_{s}\oq^{\frac 2n}}{z_{s-1}} \prod_{a=i}^{j-1} (z_a \oq^{\frac {2a}n})^{\lfloor \mu(a-i+1) \rfloor - \lfloor \mu(a-i) \rfloor}}{\left(1 - \frac {z_{i+1}}{z_{i}q^2}  \right) ... \left(1 - \frac {z_{j-1}}{z_{j-2}q^2} \right)} \prod_{i\leq a < b < j} \zeta \left( \frac {z_b}{z_a} \right)  \right]
$$
Therefore, we conclude that when $\pm = +$, the LHS of \eqref{eqn:antipode} is: 
$$
\sym \left[X \cdot \frac {\prod_{a=i}^{j-1}  (z_a \oq^{\frac {2a}n})^{\lfloor \mu(a-i+1) \rfloor - \lfloor \mu(a-i) \rfloor}}{\left(1 - \frac {z_{i+1}}{z_{i}q^2}  \right) ... \left(1 - \frac {z_{j-1}}{z_{j-2}q^2} \right)} \prod_{i \leq a < b \leq j}  \zeta \left( \frac {z_b}{z_a} \right) \right]
$$
where:
$$
X = 1 - (q^2q^{\frac 2n})^{1-\frac{j-i}h} \prod^{\mu(s-i) \in \BZ}_{i < s < j} \frac {z_{s}\oq^{\frac 2n}}{z_{s-1}} - \sum_{i < t < j}^{\mu(t-i) \in \BZ} (q^2q^{\frac 2n})^{1-\frac {j-t}h} \left(1- \frac {z_t}{z_{t - 1} q^2} \right) \prod^{\mu(s-i) \in \BZ}_{t < s < j} \frac {z_{s}\oq^{\frac 2n}}{z_{s-1}}
$$
Being a telescoping sum, it is clear that $X=0$, thus proving \eqref{eqn:antipode} when $\pm = +$. The case when $\pm = -$ is proved analogously, so we leave it as an exercise. 
	
\end{proof}

\subsection{} \label{sub:alphabeta} 

The isomorphism $\Xi$ of Proposition \ref{prop:sub} sends:
\begin{align*}
\underbrace{1 \otimes ... \otimes e_{\pm [s;t)} \otimes ... \otimes 1}_{e_{\pm [s;t)} \text{ on }r-\text{th position}} \in  U^\pm_{q}(\dot{\fgl}_{\frac ng})^{\otimes g} \ &\stackrel{\Xi}\leadsto \  E_{\pm [as+r,at+r)}^{(b(t-s))}  \\
\underbrace{1 \otimes ... \otimes \psi_s \otimes ... \otimes 1}_{\psi_s \text{ on }r-\text{th position}} \in  U^0_{q}(\dot{\fgl}_{\frac ng})^{\otimes g} \ &\stackrel{\Xi}\leadsto \ \barc^{bs} \psi_{r+sa}
\end{align*}
Therefore, the central element of $U_{q}(\dot{\fgl}_{\frac ng})^{\otimes g}$ is sent to $c^{\frac ag} \barc^{\frac {bn}g}$ under $\Xi$. As for:
\begin{equation}
\label{eqn:images}
\Xi(1 \otimes ... \otimes p_{\pm k} \otimes ... \otimes 1) \in \CB_\mu
\end{equation}
these are only determined up to constant multiple, as are primitive elements of quantum groups. To fix these multiples, we will consider the following linear maps:
\begin{equation}
\label{eqn:linear maps}
\alpha_{\pm [i;j)} : \CA_{\pm [i;j)} \rightarrow \BF
\end{equation}
given by (in what follows, we write $\pm h$ for the homogeneous degree of $R^\pm$):
\begin{align}
&\alpha_{[i;j)}(R^+) = \frac {R^+(1_i,...,1_{j-1})}{\prod_{i \leq a < b < j} \zeta\left(\frac {1_b}{1_a}\right)} \frac {(1 - q^{-2})^{j-i} \oq_+^{- \frac {\gcd(h,j-i)}n}}{\prod_{a=i}^{j-1} \oq_+^{\frac {2a}n \left( \left \lfloor \frac {h(a-i+1)}{j-i} \right \rfloor - \left \lfloor \frac {h(a-i)}{j-i} \right \rfloor \right)}} \label{eqn:alpha} \\
&\alpha_{-[i;j)}(R^-) = \frac {R^-(q^{2i},...,q^{2(j-1)})}{\prod_{i \leq a < b < j} \zeta\left( \frac {q^{2a}}{q^{2b}} \right)} \frac {(1-q^{-2})^{j-i} \oq_-^{-\frac {\gcd(h,j-i)}n}}{\prod_{a=i}^{j-1} \oq_-^{\frac {2a}n \left( \left \lceil \frac {h(a-i)}{j-i} \right \rceil - \left \lceil \frac {h(a-i+1)}{j-i} \right \rceil \right)}} \label{eqn:beta}
\end{align}
where in the RHS of either expression above, we plug in the number $1$ (respectively $q^{2a}$) into a variable of color $a$ of the rational function $R$, for all $a \in \{i,...,j-1\}$. Note that one needs to cancel the poles of $\zeta(q^{2a-2b})$ against the poles of the specialization $R(q^{2i},...,q^{2j-2})$ in order for the fraction \eqref{eqn:beta} to be well-defined. Formulas \eqref{eqn:alpha}, \eqref{eqn:beta} are constant multiples of the linear maps $\alpha_{[i;j)}$, $\beta_{[i;j)}$ of \cite{Tor}. \\

\begin{proposition} 
\label{prop:pseudo}
	
If we let $\pm h_i = \emph{vdeg } R_i^\pm$, then we have:
\begin{equation}
\label{eqn:pseudo}
\alpha_{\pm [i;j)}(R_1^\pm R^\pm_2) = \alpha_{\pm [s;j)}(R_1^\pm) \alpha_{\pm[i;s)}(R_2^\pm) \cdot \oq_{\pm}^{\frac {h_1(s-i) - h_2(j-s)}n}
\end{equation}
if $\exists \ s$ such that $\emph{hdeg } R_1^\pm R_2^\pm = \pm [i;j)$, $\emph{hdeg } R_1^\pm = \pm[s;j)$ and $\emph{hdeg } R_2^\pm = \pm[i;s)$. If such an $s$ does not exist, then the RHS of \eqref{eqn:pseudo} is set equal to 0, by convention. \\

\end{proposition}

\noindent Up to certain powers of $\oq_\pm$, Proposition \ref{prop:pseudo} is equivalent to Lemma 3.20 of \cite{Tor}. \footnote{We leave the details on establishing the equivalence to the interested reader, and note that it requires one to use the elementary identity:
\begin{equation}
\label{eqn:floor identity}
\sum_{a=i}^{j-1} a \left( \left \lfloor \frac {h(a-i+1)}{j-i} \right \rfloor - \left \lfloor \frac {h(a-i)}{j-i} \right \rfloor \right) = hi + \frac {h(j-i) - h + (j-i) - \gcd(h,j-i)}2 
\end{equation}} \\

\begin{lemma}
\label{lem:unique}

(\cite{Tor}) Any element of $\CB_\mu$ which is primitive for $\Delta_\mu$ (i.e. all intermediate terms in its coproduct vanish) and is annihilated by all the $\alpha_{\pm[i;j)}$, vanishes. \\

\end{lemma}

\subsection{} 
\label{sub:alpha pairing}

The following result connects the maps $\alpha_{\pm [i;j)}$ to the pairing \eqref{eqn:pair b mu}, and implies the non-degeneracy of the latter (as expected from the isomorphism \eqref{eqn:smalliso}): \\

\begin{proposition}
\label{prop:pairing}
	
(\cite{Tor}) For any $R^\pm \in \CB^\pm_\mu$ and any $i<j$, we have:
\begin{equation}
\label{eqn:bonnie}
\left \langle R^\pm , E_{\mp[i;j)}^\mu\right \rangle = \alpha_{\pm[i;j)}(R^\pm) \cdot \oq_\pm^{\frac {\gcd(\mu(j-i), j-i)}n}
\end{equation}
Moreover, for all $j'-i' = j-i \in \BN$ we have:
\begin{align}
&\alpha_{\pm [i';j')}(E_{\pm [i;j)}^{\mu}) = \delta_{(i',j')}^{(i,j)} (1 - q^{-2}) \oq_\pm^{-\frac {\gcd(\mu(j-i),j-i)}n} \label{eqn:main pair 1} \\
&\alpha_{\pm [i';j')}(\bE_{\pm [i;j)}^{\mu}) = \delta_{(i',j')}^{(i,j)} (1 - q^{2}) \oq_\pm ^{\frac {\gcd(\mu(j-i),j-i)}n} \label{eqn:main pair 2}
\end{align}
As a consequence of \eqref{eqn:bonnie} and \eqref{eqn:main pair 1}, we conclude that:
\begin{equation}
\label{eqn:pair1}
\left \langle E_{[i;j)}^\mu, E_{-[i';j')}^\mu \right \rangle = \delta_{(i',j')}^{(i,j)} \left( 1 - q^{-2} \right)
\end{equation}
The Kronecker delta symbol $\delta_{(i',j')}^{(i,j)}$ is 1 if and only if $(i,j) \equiv (i',j')$ mod $(n,n)\BZ$. \\
	
\end{proposition} 

\noindent In Definition \ref{def:correspondence}, we have fixed the primitive elements $p_{\pm k}$ of quantum groups by specifying their pairings with $e_{\pm [i;j)}$. Proposition \ref{prop:pairing} implies that the primitive elements inside $\CB_\mu$ can be equivalently determined by specifying their images under the linear maps $\alpha_{\pm [i;j)}$. We will specify these images in Subsection \ref{sub:summarize}. \\

\subsection{} 
\label{sub:cop hinge computation}

The following statement is our first computation of the coproduct of an element in the shuffle algebra that goes beyond the leading order term. \\

\begin{proposition}
\label{prop:coproduct}
	
Assume $\gcd(j-i,k) = 1$, and consider the lattice triangle $T$:
	
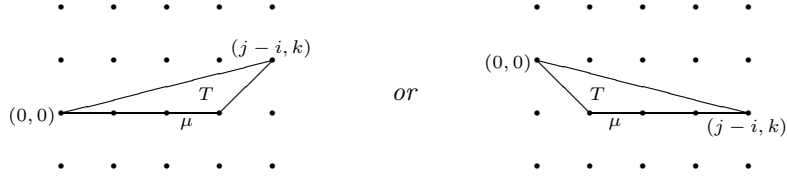
\begin{picture}(100,80)(-30,30)
	
\put(0,40){\circle*{2}}\put(20,40){\circle*{2}}\put(40,40){\circle*{2}}\put(60,40){\circle*{2}}\put(80,40){\circle*{2}}\put(0,60){\circle*{2}}\put(20,60){\circle*{2}}\put(40,60){\circle*{2}}\put(60,60){\circle*{2}}\put(80,60){\circle*{2}}\put(0,80){\circle*{2}}\put(20,80){\circle*{2}}\put(40,80){\circle*{2}}\put(60,80){\circle*{2}}\put(80,80){\circle*{2}}\put(0,100){\circle*{2}}\put(20,100){\circle*{2}}\put(40,100){\circle*{2}}\put(60,100){\circle*{2}}\put(80,100){\circle*{2}}
	
\put(0,60){\line(4,1){80}}
\put(0,60){\line(1,0){60}}
\put(60,60){\line(1,1){20}}
	
\put(-20,57){\scriptsize{$(0,0)$}}
\put(64,83){\scriptsize{$(j-i,k)$}}
\put(45,55){\scriptsize{$\mu$}}
\put(52,65){\scriptsize{$T$}}

\put(125,65){or}
	
\put(180,40){\circle*{2}}\put(200,40){\circle*{2}}\put(220,40){\circle*{2}}\put(240,40){\circle*{2}}\put(260,40){\circle*{2}}\put(180,60){\circle*{2}}\put(200,60){\circle*{2}}\put(220,60){\circle*{2}}\put(240,60){\circle*{2}}\put(260,60){\circle*{2}}\put(180,80){\circle*{2}}\put(200,80){\circle*{2}}\put(220,80){\circle*{2}}\put(240,80){\circle*{2}}\put(260,80){\circle*{2}}\put(180,100){\circle*{2}}\put(200,100){\circle*{2}}\put(220,100){\circle*{2}}\put(240,100){\circle*{2}}\put(260,100){\circle*{2}}
	
\put(180,80){\line(4,-1){80}}
\put(200,60){\line(1,0){60}}
\put(180,80){\line(1,-1){20}}
	
\put(160,77){\scriptsize{$(0,0)$}}
\put(244,53){\scriptsize{$(j-i,k)$}}
	
\put(207,55){\scriptsize{$\mu$}}
\put(200,65){\scriptsize{$T$}}
	
\end{picture}
	
\noindent uniquely determined as the triangle of maximal area situated completely below the vector $(j-i,k)$, which does not contain any lattice points inside. Let $\mu$ denote the slope of one of the edges of $T$, as indicated in the pictures above. Then:
$$
\Delta\left(E_{[i;j)}^{(k)} \right) = \frac {\psi_j}{\psi_i} \barc^{k} \otimes E_{[i;j)}^{(k)} + E_{[i;j)}^{(k)} \otimes 1 + \Big( \text{tensors with hinge strictly below } T \Big) +
$$
\begin{equation}
\label{eqn:new coproduct 1}
+ \sum_{i \leq s < t \leq j} \begin{cases} \displaystyle \frac {\psi_j}{\psi_t} E_{[s;t)}^{(\bullet)} \frac {\psi_s}{\psi_i} \barc^{k-\bullet} \otimes \bE_{[t,j)}^{\mu} E_{[i;s)}^{\mu} &\text{for the picture on the left} \\ \displaystyle E_{[t,j)}^{\mu} \frac{\psi_t}{\psi_s} \bE_{[i;s)}^{\mu} \barc^\bullet\otimes E_{[s;t)}^{(\bullet)} &\text{for the picture on the right}\end{cases} \qquad
\end{equation}
where $\bullet = k-(j-i+s-t)\mu$. Similarly, consider the lattice triangle $T$:
	
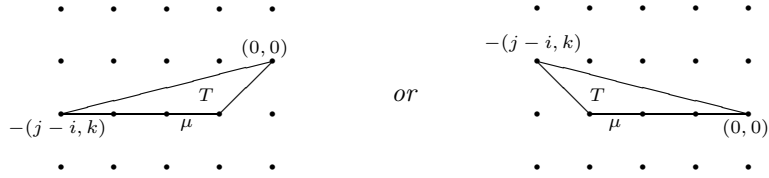
\begin{picture}(100,80)(-30,30)
	
\put(0,40){\circle*{2}}\put(20,40){\circle*{2}}\put(40,40){\circle*{2}}\put(60,40){\circle*{2}}\put(80,40){\circle*{2}}\put(0,60){\circle*{2}}\put(20,60){\circle*{2}}\put(40,60){\circle*{2}}\put(60,60){\circle*{2}}\put(80,60){\circle*{2}}\put(0,80){\circle*{2}}\put(20,80){\circle*{2}}\put(40,80){\circle*{2}}\put(60,80){\circle*{2}}\put(80,80){\circle*{2}}\put(0,100){\circle*{2}}\put(20,100){\circle*{2}}\put(40,100){\circle*{2}}\put(60,100){\circle*{2}}\put(80,100){\circle*{2}}
	
\put(0,60){\line(4,1){80}}
\put(0,60){\line(1,0){60}}
\put(60,60){\line(1,1){20}}
	
\put(-20,53){\scriptsize{$-(j-i,k)$}}
\put(68,83){\scriptsize{$(0,0)$}}
\put(45,55){\scriptsize{$\mu$}}
\put(52,65){\scriptsize{$T$}}
	
\put(125,65){or}
	
\put(180,40){\circle*{2}}\put(200,40){\circle*{2}}\put(220,40){\circle*{2}}\put(240,40){\circle*{2}}\put(260,40){\circle*{2}}\put(180,60){\circle*{2}}\put(200,60){\circle*{2}}\put(220,60){\circle*{2}}\put(240,60){\circle*{2}}\put(260,60){\circle*{2}}\put(180,80){\circle*{2}}\put(200,80){\circle*{2}}\put(220,80){\circle*{2}}\put(240,80){\circle*{2}}\put(260,80){\circle*{2}}\put(180,100){\circle*{2}}\put(200,100){\circle*{2}}\put(220,100){\circle*{2}}\put(240,100){\circle*{2}}\put(260,100){\circle*{2}}
	
\put(180,80){\line(4,-1){80}}
\put(200,60){\line(1,0){60}}
\put(180,80){\line(1,-1){20}}
	
\put(160,85){\scriptsize{$-(j-i,k)$}}
\put(250,53){\scriptsize{$(0,0)$}}
	
\put(207,55){\scriptsize{$\mu$}}
\put(200,65){\scriptsize{$T$}}
	
\end{picture}
	
\noindent of maximal area situated completely below the vector $(i-j,-k)$, which does not contain any lattice points inside. Then we have:
$$
\Delta\left(E_{-[i;j)}^{(-k)} \right) = 1 \otimes E_{-[i;j)}^{(-k)} + E_{-[i;j)}^{(-k)} \otimes \frac {\psi_i}{\psi_j} \barc^{-k}+ \Big( \text{tensors with hinge strictly below } T \Big) +
$$
\begin{equation}
\label{eqn:new coproduct 2}
+ \sum_{i \leq s < t \leq j} \begin{cases} \displaystyle \bE_{-[t,j)}^{\mu} E_{-[i;s)}^{\mu} \otimes \frac {\psi_t}{\psi_j} E_{-[s;t)}^{(\bullet)} \frac {\psi_i}{\psi_s} \barc^{-k-\bullet} &\text{for the picture on the left} \\  \displaystyle E_{-[s;t)}^{(\bullet)} \otimes E_{-[t,j)}^{\mu} \frac{\psi_s}{\psi_t} \bE_{-[i;s)}^{\mu} \barc^{\bullet} &\text{for the picture on the right} \end{cases}  \qquad \quad
\end{equation}
where $\bullet = -k + (j-i+s-t)\mu$. \\
	
\end{proposition}

\begin{proof} Let us recall that:
\begin{equation}
\label{eqn:einy}
E_{[i;j)}^{(k)} = \sym \left[ \frac {\prod_{a=i}^{j-1} (z_a \oq^{\frac {2a}n})^{\lfloor \mu(a-i+1) \rfloor - \lfloor \mu(a-i) \rfloor}}{\left(1 - \frac {z_{i+1}}{z_{i}q^2}  \right) ... \left(1 - \frac {z_{j-1}}{z_{j-2}q^2} \right)} \prod_{i\leq a < b < j} \zeta \left( \frac {z_b}{z_a} \right)  \right]	
\end{equation}
By \eqref{eqn:coproduct 3}, the coproduct $\Delta(E_{[i;j)}^{(k)})$ is computed by taking an arbitrary set:
$$
C \subset \{i,...,j-1\} 
$$
and expanding the rational function $E_{[i;j)}^{(k)}$ by sending the variables $\{z_a\}_{a\in C}$ to $\infty$, while keeping the variables $\{z_a\}_{a\in \bar{C}}$ fixed, where $\bar{C} = \{i,...,j-1\} \backslash C$. Therefore, the second tensor factor $R$ in any summand of $\Delta(E_{[i;j)}^{(k)})$ has the property that:
$$
\hdeg R = \# C
$$
Meanwhile, total homogeneous degree of $R$ in the variables $C$ satisfies (see \eqref{eqn:einy}):
\begin{equation}
\label{eqn:number 0}
\vdeg R \leq \sum_{a \in C} \left \lfloor \frac {k(a-i+1)}{j-i} \right \rfloor - \left \lfloor \frac {k(a-i)}{j-i} \right \rfloor - \# \Big \{a \in C \text{ s.t. } a-1 \in \bar{C} \Big \}
\end{equation}
Let us assume that $C$ is a union of blocks of consecutive integers:
$$
C = \{i_1,...,j_1-1\} \sqcup ... \sqcup \{i_v,...,j_v-1\}
$$
where $i \leq i_1 < j_1 < i_2 < j_2 < ... < i_v < j_v \leq j$. Therefore, we have:
\begin{equation}
\label{eqn:number 1}
\hdeg R = \sum_{u=1}^v (j_u - i_u)
\end{equation}
while \eqref{eqn:number 0} gives us:
$$
\vdeg R \leq \sum_{u=1}^v \left( \left \lfloor \frac {k(j_u-i)}{j-i} \right \rfloor - \left \lfloor \frac {k(i_u-i)}{j-i} \right \rfloor \right) - v + \delta_{i_1}^i =
$$
\begin{equation}
\label{eqn:number 2}
\stackrel{\gcd(k,j-i)=1}= \sum_{u=1}^v \left( \left \lfloor \frac {k(j_u-i)}{j-i} \right \rfloor - \left \lceil \frac {k(i_u-i)}{j-i} \right \rceil \right) 
\end{equation}
We conclude that all tensors appearing in \eqref{eqn:new coproduct 1} will have hinge $(x,y)$, where $x$ is the RHS of \eqref{eqn:number 0} and $y$ is less than or equal to the RHS of \eqref{eqn:number 2}. Such a hinge lies strictly below the triangle $T$, except in the following situations: \\
	
\begin{itemize}[leftmargin=*]
		
\item in the case depicted on the left of the figure, if:
\begin{equation}
\label{eqn:choice 1}
C = \{i,...,s-1\} \sqcup \{t,...,j-1\}
\end{equation}
for certain $i \leq s < t \leq j$. The reason for this is that each summand on the second row of the RHS of \eqref{eqn:number 2} is $\leq (j_u-i_u)\mu$ with equality if and only if $u=1, i_1 = i$ or $u=v, j_v = j$, by the fact that the triangle $T$ has no lattice points inside. Failure of equality to hold would force the vector $(x,y)$ to lie below the line of slope $\mu$. \\
		
\item in the case depicted on the right of the figure, if:
\begin{equation}
\label{eqn:choice 2}
C = \{s,...,t-1\}
\end{equation}
for certain $i \leq s < t \leq j$. The reason for this is that we may rewrite \eqref{eqn:number 2} as:
\begin{equation}
\label{eqn:number 3}
k - y = k - \text{vertical degree} \geq \sum_{u=0}^v \left(  \left \lceil \frac {k(i_{u+1}-i)}{j-i} \right \rceil - \left \lfloor \frac {k(j_u-i)}{j-i} \right \rfloor \right) 
\end{equation} 	
where we write $i_{v+1} = j$ and $j_0 = i$. Each summand in the right-hand side of \eqref{eqn:number 3} is $\geq (i_{u+1}-j_u)\mu$, with equality if and only if either $u=0$ or $u=v$, by the fact that the triangle $T$ has no lattice points inside. Failure of equality to hold would force the vector $(x,y)$ to lie below the line of slope $\mu$. \\
		
\end{itemize}  
	
\noindent The analysis in the two bullets above shows that the summands on the second line of \eqref{eqn:new coproduct 1} correspond to the choices \eqref{eqn:choice 1} and \eqref{eqn:choice 2}, respectively. More specifically, the second tensor factors of the summands in question consist of the leading order terms of $E_{[i;j)}^{(k)}$ when we send the variables $\{z_a, a \in C\}$ to $\infty$. Explicitly, in the case depicted on the left, we multiply the variables $z_i,...,z_{s-1},z_t,...,z_{j-1}$ in \eqref{eqn:einy} by $\xi$ and compute the leading order term as $\xi \rightarrow \infty$ is (we let $\nu(j-i) = k$):
$$
\sym \left[z_{t-1} \frac {\prod_{a=s}^{t-1} (z_a \oq^{\frac {2a}n})^{\lfloor \nu(a-i+1) \rfloor - \lfloor \nu(a-i) \rfloor}}{\left(1 - \frac {z_{s+1}}{z_{s}q^2}  \right) ... \left(1 - \frac {z_{t-1}}{z_{t-2}q^2} \right)} \prod_{s\leq a < b < t} \zeta \left( \frac {z_b}{z_a} \right) \right] 
$$
$$
\lim_{\xi \rightarrow \infty} \prod_{s \leq b < t}^{i \leq a < s} \zeta \left( \frac {z_b}{\xi z_a} \right)\prod_{t \leq b < j}^{s \leq a < t} \zeta \left( \frac {\xi z_b}{z_a} \right) 
$$
\begin{multline*}
\cdot \sym \left[ \frac {\prod_{a=i}^{s-1} (z_a \oq^{\frac {2a}n})^{\lfloor \nu(a-i+1) \rfloor - \lfloor \nu(a-i) \rfloor}}{\left(1 - \frac {z_{i+1}}{z_{i}q^2}  \right) ... \left(1 - \frac {z_{s-1}}{z_{s-2}q^2} \right)} \prod_{i\leq a < b < s} \zeta \left( \frac {z_b}{z_a} \right)\prod^{i \leq a <s }_{t \leq b < j} \zeta \left( \frac {z_b}{z_a} \right) \right.  \\ \left. \frac {(-q^2)\prod_{a=t}^{j-1} (z_a \oq^{\frac {2a}n})^{\lfloor \nu(a-i+1) \rfloor - \lfloor \nu(a-i) \rfloor}}{z_t\left(1 - \frac {z_{t+1}}{z_{t}q^2}  \right) ... \left(1 - \frac {z_{j-1}}{z_{j-2}q^2} \right)} \prod_{t\leq a < b < j} \zeta \left( \frac {z_b}{z_a} \right)   \right]
\end{multline*}
By \eqref{eqn:elementary} and \eqref{eqn:einy}, the symmetrizations above are responsible for the elements: 
$$
E_{[s;t)}^{(\bullet)} \quad \text{and} \quad \bE_{[t,j)}^{\mu} E_{[i;s)}^{\mu}
$$
in \eqref{eqn:new coproduct 1}, respectively, while the limit on the second line of the expression above is a power of $q$ which is precisely accounted for by the fact that some of the $\psi$'s are on the other side of $E$ in \eqref{eqn:new coproduct 1} when compared to \eqref{eqn:coproduct 3}. Similarly, in the case depicted on the right, we multiply the variables $z_s,...,z_{t-1}$ in \eqref{eqn:einy} by $\xi$ and send $\xi \rightarrow \infty$. Then the leading order term in $\xi$ is (we let $\nu(j-i) = k$):
$$
\sym \left[(-q^2) z_{s-1} \frac {\prod_{a=i}^{s-1} (z_a \oq^{\frac {2a}n})^{\lfloor \nu(a-i+1) \rfloor - \lfloor \nu(a-i) \rfloor}}{\left(1 - \frac {z_{i+1}}{z_{i}q^2}  \right) ... \left(1 - \frac {z_{s-1}}{z_{s-2}q^2} \right)} \prod_{i \leq a < b < s} \zeta \left( \frac {z_b}{z_a} \right) \prod_{t \leq b < j}^{i \leq a < s} \zeta \left( \frac {z_b}{z_a} \right) \right. 
$$
$$
\left. \frac {\prod_{a=t}^{j-1} (z_a \oq^{\frac {2a}n})^{\lfloor \nu(a-i+1) \rfloor - \lfloor \nu(a-i) \rfloor}}{\left(1 - \frac {z_{t+1}}{z_t q^2}  \right) ... \left(1 - \frac {z_{j-1}}{z_{j-2}q^2} \right)} \prod_{t \leq a < b < j} \zeta \left( \frac {z_b}{z_a} \right)  \right]  
$$
$$
\lim_{\xi \rightarrow \infty} \prod_{s \leq b < t}^{i \leq a < s} \zeta \left( \frac {\xi z_b}{z_a} \right)\prod_{t \leq b < j}^{s \leq a < t} \zeta \left( \frac {z_b}{\xi z_a} \right) $$
$$
\sym \left[ \frac {\prod_{a=s}^{t-1} (z_a \oq^{\frac {2a}n})^{\lfloor \nu(a-i+1) \rfloor - \lfloor \nu(a-i) \rfloor}}{z_s \left(1 - \frac {z_{s+1}}{z_s q^2}  \right) ... \left(1 - \frac {z_{t-1}}{z_{t-2}q^2} \right)} \prod_{s \leq a < b < t} \zeta \left( \frac {z_b}{z_a} \right)  \right] 
$$
By \eqref{eqn:elementary} and \eqref{eqn:einy}, the symmetrizations above are responsible for the elements: 
$$
E_{[t,j)}^{\mu} \bE_{[i;s)}^{\mu} \quad \text{and} \quad E_{[s;t)}^{(\bullet)}
$$
in \eqref{eqn:new coproduct 1}, respectively, while the limit on the third line of the expression above is a power of $q$ which is precisely accounted for by the fact that some of the $\psi$'s are on the other side of $E$ in \eqref{eqn:new coproduct 1} when compared to \eqref{eqn:coproduct 3}. Formula \eqref{eqn:new coproduct 2} is proved by an analogous argument to \eqref{eqn:new coproduct 1}, and is therefore left as an exercise to the reader.
	
\end{proof} 

\subsection{} 
\label{sub:summarize}

We are ready to define the shuffle elements featured in the Introduction. Let:
\begin{align}
&P_{[i;j)}^{(k)} = \frac {q\oq^{\frac 1n}}{q - q^{-1}} \cdot E_{[i;j)}^{(k)} = \frac {(q\oq^{\frac 1n})^{-1}}{q^{-1} - q} \cdot \bE_{[i;j)}^{(k)} \label{eqn:p plus} \\
&P_{-[i;j)}^{(-k)} = \frac {\oq^{-\frac 1n}}{q^{-1} - q} \cdot E_{-[i;j)}^{(-k)} = \frac {\oq^{\frac 1n}}{q - q^{-1}} \cdot \bE_{-[i;j)}^{(-k)} \label{eqn:p minus}
\end{align}
for all $i < j$ and $k$ such that $\gcd(k,j-i)=1$. We have $P_{\pm [i;j)}^{(\pm k)} \in \CB_{\frac k{j-i}}$, and:
\begin{align}
&\Delta_{\frac k{j-i}} \left(P^{(k)}_{[i;j)}\right) = \frac {\psi_j}{\psi_i} \barc^k \otimes P^{(k)}_{[i;j)} + P^{(k)}_{[i;j)} \otimes 1 \label{eqn:primitive 1} \\
&\Delta_{\frac k{j-i}} \left(P^{(-k)}_{-[i;j)}\right) = 1 \otimes P^{(-k)}_{-[i;j)} + P^{(-k)}_{-[i;j)} \otimes \frac {\psi_i}{\psi_j} \barc^{-k} \label{eqn:primitive 2}
\end{align}
By \eqref{eqn:main pair 1}, it is easy to see that:
\begin{equation}
\label{eqn:normalize simple}
\alpha_{\pm [u;v)}\left(P^{(\pm k)}_{\pm [i;j)}\right) = \pm \delta_{(u,v)}^{(i,j)} 
\end{equation}
$\forall (u,v) \in \zzz$. Proposition \ref{prop:sub} implies that there exist unique elements:
\begin{equation}
\label{eqn:def p}
P^{(\pm k')}_{\pm l \bde, \hi} \in \CB_{\frac {k'}{nl}} \quad \forall k' \in \BZ, \ l \in \BN, \ \hi \in \BZ/g\BZ
\end{equation}
(where $g = \gcd(n,\text{denominator of }\frac {k'}{nl})$) completely determined by the condition:
\begin{align}
&\Delta_{\frac {k'}{nl}} \left(P^{(k')}_{l \bde, \hi}\right) = c^l \barc^{k'} \otimes P^{(k')}_{l \bde, \hi} + P^{(k')}_{l \bde, \hi} \otimes 1 \label{eqn:condition 1} \\
&\Delta_{\frac {k'}{nl}} \left(P^{(-k')}_{- l \bde, \hi}\right) = 1 \otimes P^{(-k')}_{-l \bde, \hi} + P^{(-k')}_{-l \bde, \hi} \otimes c^{-l} \barc^{-k'} \label{eqn:condition 2}
\end{align}
and the normalization:
\begin{equation}
\label{eqn:normalize imaginary}
\alpha_{\pm [u;u+nl)} \left(P^{(\pm k')}_{\pm l \bde, \hi}\right) = \pm \delta_{u \text{ mod }g}^{\hi} 
\end{equation}
$\forall u \in \{1,...,n\}$. To summarize, under the isomorphism \eqref{eqn:smalliso}: \\

\begin{itemize}
	
\item $P_{\pm [i;j)}^{(\pm k)}$ correspond to the simple root generators $x^\pm_s$ \\
	
\item $P^{(\pm k')}_{\pm l \bde, \hi}$ correspond to the imaginary root generators $p_{\pm t}$ \\
	
\end{itemize}

\noindent We will often write $P_{\pm [i;i+nl)}^{(\pm k')} = P_{\pm l\bde, i}^{(\pm k')}$ if $\gcd(k',nl) = 1$. \\

\subsection{}
\label{sub:cartan} 

To complete the definition of the root generators of $\CA$, we set:
\begin{equation}
\label{eqn:vertical}
P_{0\bde, \hi}^{(\pm k)} = \pm a_{r,\pm k} \oq^{\frac {\pm (2r-1)k}n}
\end{equation}
for all $r \in \BZ/n\BZ$ and $k \in \BN$. Moreover, we define:
\begin{align}
&E^{\infty}_{[r;r)}(z) = \sum_{d=0}^\infty \frac {E_{[r;r)}^{(d)}}{z^{d}} = \exp \left[ \sum_{d=1}^{\infty} \frac {a_{r,d} (1-q^{-2d})}{dz^d} \oq^{\frac {2(r-1)d}n} \right] \label{eqn:cartan 1} \\
&\bE^{\infty}_{[r;r)}(z) = \sum_{d=0}^\infty \frac {\bE_{[r;r)}^{(d)}}{z^{d}} = \exp \left[ \sum_{d=1}^{\infty} \frac {a_{r,d}(1-q^{2d})}{dz^d} \oq^{\frac {2rd}n} \right] \label{eqn:cartan 2} \\
&E^{\infty}_{-[r;r)}(z) = \sum_{d=0}^\infty \frac {E_{-[r;r)}^{(-d)}}{z^{-d}} = \exp \left[ \sum_{d=1}^{\infty} \frac {a_{r,-d}(q^d-q^{-d})}{dz^{-d}} \oq^{-\frac {2(r-1)d}n} \right] \label{eqn:cartan 3} \\
&\bE^{\infty}_{-[r;r)}(z) = \sum_{d=0}^\infty \frac {\bE_{-[r;r)}^{(-d)}}{z^{-d}} = \exp \left[ \sum_{d=1}^{\infty} \frac {a_{r,-d}(q^{-d}-q^{d})}{dz^{-d}} \oq^{-\frac {2rd}n} \right] \label{eqn:cartan 4} 
\end{align}
Comparing the definitions above with \eqref{eqn:convert 1}--\eqref{eqn:convert 2}, we see that:
\begin{equation}
\label{eqn:convert cartan}
\ph^\pm_r (z) = \frac {\psi_{r+1}^{\pm 1}}{\psi_r^{\pm 1}} \bE^\infty_{\pm [r;r)} \left(z\oq^{\frac {2r}n} \right) E^\infty_{\pm [r+1;r+1)} \left(z \oq^{\frac {2(r-1)}n} \right)
\end{equation} 
This concludes the description of the generators considered in Theorem \ref{thm:main}. \\

\section{The statement and proof of Theorem \ref{thm:intro main}}
\label{sec:proof}

\noindent In the present Section, we will fully state and prove Theorem \ref{thm:intro main} (i.e. Theorem \ref{thm:main}). \\

\begin{itemize}[leftmargin=*]

\item In Subsection \ref{sub:connect slope}, we connect the slope subalgebras $\CB_\mu$ with the quantum groups of Section \ref{sec:quantum}, by establishing an explicit map between their generators \\

\item In Subsection \ref{sub:tilde}, we give the full statement of Theorem \ref{thm:intro main}, by giving the explicit form of relations \eqref{eqn:rel 1 intro} and \eqref{eqn:rel 2 intro}, in the guise of \eqref{eqn:case 1}, \eqref{eqn:case 2}, \eqref{eqn:case 3}, \eqref{eqn:case 4} \\

\item In Subsection \ref{sub:case 1}, we prove relations \eqref{eqn:case 1} and \eqref{eqn:case 3} \\

\item In Subsection \ref{sub:pairings intermezzo}, we work out the pairings between the $P^{(\pm k)}_{\pm [i;j)}$'s and $E^{(\mp k')}_{\mp [i';j')}$'s \\

\item In Subsection \ref{sub:analogue}, we present an analogue of the computation of Subsection \ref{sub:cop hinge computation} \\

\item In Subsection \ref{sub:case 2}, we prove relations \eqref{eqn:case 2} and \eqref{eqn:case 3} \\

\item In Subsection \ref{sub:main theorem}, we prove our main Theorem \\

\item In Subsection \ref{sub:rotate}, we consider a smaller set of generators of $\UU$ from the one of Theorem \ref{thm:main}, and deduce the relations among these generators \\

\end{itemize}

\subsection{} 
\label{sub:connect slope}

For any pair of coprime integers $(a,b) \in \BN \times \BZ \sqcup (0,1)$, we let: 
\begin{equation}
\label{eqn:abstract slope}
\CE_{\frac ba} := U_{q}(\dot{\fgl}_{\frac ng})^{\otimes g} 
\end{equation}
where $g = \gcd(n,a)$. We will write $p_{\pm [i;j)}^{(\pm k)}$, $p_{\pm l\bde,\hi}^{(\pm k')}$ for the primitive generators of $\CE_{\frac ba}$, where the sets of indices are such that the assignment:
\begin{equation}
\label{eqn:assignment 1}
P_{\pm [i;j)}^{(\pm k)} \mapsto p_{\pm [i;j)}^{(\pm k)}, \qquad P_{\pm l\bde,\hi}^{(\pm k')} \mapsto p_{\pm l\bde,\hi}^{(\pm k')}
\end{equation}
for all $\frac k{j-i} = \frac {k'}{nl} = \frac ba$ and $\hi \in \BZ/g\BZ$ yields an algebra isomorphism:
\begin{equation}
\label{eqn:mini}
\CB_{\frac ba} \stackrel{\sim}\longrightarrow \CE_{\frac ba}
\end{equation}
that is obtained by composing \eqref{eqn:smalliso} with \eqref{eqn:abstract slope}. We will write:
\begin{equation}
\label{eqn:assignment 2}
e_{\pm [i;j)}^{(\pm k)}, \be_{\pm [i;j)}^{(\pm k)} \in \CE_{\frac ba}
\end{equation}
for the images of the elements $E_{\pm [i;j)}^{(\pm k)}, \bE_{\pm [i;j)}^{(\pm k)}$ under the isomorphisms \eqref{eqn:mini}. \\

\begin{remark}

For any $\mu \in \BQ \sqcup \infty$, we extend the scalars from $\BQ(q)$ to $\BQ(q,\oq^{\frac 1n})$, and identify the central elements of the abstract algebras $\CE_\mu$ according to the rule:
$$
\Big(\text{central element of }\CE_{\frac ba}\Big) = c^{\frac ag} \barc^{\frac {bn}g} 
$$
where $g = \gcd(n,a)$, for two henceforth fixed central elements $c,\barc$. Moreover, we identify the Cartan elements of all the abstract algebras $\CE_\mu$ according to the rule:
$$
\Big(\psi_s \text{ on } p\text{--th factor of } U_{q}(\dot{\fgl}_{\frac ng})^{\otimes g} \cong \CE_{\frac ba} \Big) = \psi_{p+sa} \barc^{bs}
$$
where $\psi_1,...,\psi_n$ are the Cartan elements of $\CE_0 = \uu$. This implies the relations:
\begin{equation}
\label{eqn:rel 0}
c,\barc \text{ central}, \qquad \qquad \qquad \psi_s \psi_{s'} = \psi_{s'}\psi_s
\end{equation} 
\begin{equation}
\label{eqn:rel 1}
\psi_s p_{\pm [i;j)}^{(\pm k)} = q^{\pm (\delta_s^j - \delta_s^i)} p_{\pm [i;j)}^{(\pm k)} \psi_s, \qquad \psi_s p_{\pm l\bde,\hi}^{(\pm k)} = p_{\pm l\bde,\hi}^{(\pm k)} \psi_s
\end{equation} 
in the subalgebras $\CE_\mu$ for all applicable indices $i,j,s,s',k,l,r$. \\

\end{remark}

\subsection{} 
\label{sub:tilde}

Recall that we abbreviate $\oq_+ = \oq$ and $\oq_- = \oq^{-1} q^{-n}$. We write:
\begin{equation}
\label{eqn:kron mod}
\delta_{i \text{ mod }g}^j = \begin{cases} 1 &\text{if } i \equiv j \text{ mod }g \\ 0 &\text{otherwise} \end{cases}
\end{equation}
If $g = n$, we abbreviate the notation \eqref{eqn:kron mod} to $\delta_i^j$. For any $i,j,i',j' \in \BZ$, recall that:
\begin{equation}
\label{eqn:kron mod double}
\delta_{(i',j')}^{(i,j)} = \begin{cases} 1 &\text{if } (i,j) \equiv (i',j') \text{ mod } (n,n) \\ 0 &\text{otherwise} \end{cases}
\end{equation}
Let us consider the linear combination:
\begin{equation}
\label{eqn:tilde p} 
\tp_{\pm [i;j)}^{(\pm k)} =  \sum_{x \in \BZ/n\BZ} p_{\pm [i+x;j+x)}^{(\pm k)} \left[q^{-\delta_j^i} \delta_x^0 +  \delta_j^i (q-q^{-1}) \frac {\oq_\pm^{\frac 2n \cdot \overline{-kx}}}{\oq_\pm^2-1} \right]
\end{equation}
whenever $\gcd(k,j-i) = 1$. The following is result is elementary, so we leave its proof as an exercise to the interested reader: \\

\begin{lemma}
\label{lem:elem}

For any $k$ which is coprime with $n$, we have:
$$
M^+ M^- = \emph{Id}
$$
where $M^\pm$ is the $n \times n$ matrix whose $(a,b)$--th entry is $q^{-1} \delta_b^a + (q-q^{-1}) \frac {\oq_\pm^{\frac 2n \cdot \overline{\pm k(a-b)}}}{\oq_\pm^2-1}$. \\

\end{lemma}

\noindent Therefore, we conclude that \eqref{eqn:tilde p} is equivalent to the following:
\begin{equation}
\label{eqn:tilde p equivalent} 
p_{\pm [i;j)}^{(\pm k)} =  \sum_{x \in \BZ/n\BZ} \tp_{\pm [i+x;j+x)}^{(\pm k)} \left[q^{-\delta_j^i} \delta_x^0 +  \delta_j^i (q-q^{-1}) \frac {\oq_\mp^{\frac 2n \cdot \overline{kx}}}{\oq_\mp^2-1} \right]
\end{equation}
We are now ready to give the complete statement of Theorem \ref{thm:intro main}: \\

\begin{theorem}
\label{thm:main}
	
We have an algebra isomorphism:
\begin{equation}
\label{eqn:generators}
\UU \cong \boxed{\CC := \Big \langle \CE_{\mu} \Big \rangle_{\mu \in \BQ \sqcup \infty} \Big/\text{relations \eqref{eqn:case 1}--\eqref{eqn:case 4}}}
\end{equation}
where the defining relations in $\CC$ are of the four types below. \\

\noindent \underline{Type 1:} for all \footnote{We only allow $i = j$ in the following if $k > 0$ and $l = 0$ if $k' > 0$} $i \leq j$, $l \geq 0$ and $k,k' \in \BZ$ such that:
$$
d:=  \det \begin{pmatrix} k & k' \\ j-i & nl \end{pmatrix}
$$ 
satisfies $|d| = \gcd(k',nl)$, we set (letting $g = \gcd \left(n, \frac {nl}d \right)$) for any $\hi \in \BZ/g\BZ$:
\begin{equation}
\label{eqn:case 1}
\Big[ p_{\pm[i;j)}^{(\pm k)}, p_{\pm l\bde, \hi}^{(\pm k')} \Big] = \pm p_{\pm [i;j+ln)}^{(\pm k \pm k')} \left(\delta_{i \text{ mod }g}^r \oq_\pm^{\frac dn} - \delta_{j \text{ mod }g}^r \oq_\pm^{-\frac dn} \right)
\end{equation}
\text{ }

\noindent \underline{Type 2}: under the same assumptions as in Type 1, we set:
\begin{equation}
\label{eqn:case 2}
\begin{aligned}
&\Big[ p_{\pm [i;j)}^{(\pm k)}, p_{\mp l\bde, \hi}^{(\mp k')} \Big] = \pm (c^{\pm l} \barc^{\pm k'})^{\frac {d}{|d|}} p_{\pm [i;j-nl)}^{(\pm k \mp k')} \left(\delta_{i \text{ mod }g}^r \oq_\mp^{\frac dn} - \delta_{j \text{ mod }g}^r \oq_\mp^{-\frac dn} \right) \\
&\Big[ \tp_{\pm [i;j)}^{(\pm k)}, p_{\mp l\bde, \hi}^{(\mp k')} \Big] = \pm \left( \frac {\psi_j^{\pm 1} \barc^{\pm k}}{\psi_i^{\pm 1} q^{1-\delta_j^i}}\right)^{\frac {d}{|d|}} p^{(\pm k \mp k')}_{\mp[j-nl;i)} \left(\delta_{i \text{ mod }g}^r \oq_\mp^{\frac dn} - \delta_{j \text{ mod }g}^r \oq_\mp^{- \frac dn} \right)
\end{aligned} 
\end{equation}
with the first equation in \eqref{eqn:case 2} holding if $j-i > nl$ (or $j-i = nl$ and $k > k'$) and the second equation holding if $j-i<nl$ (or $j-i=nl$ and $k < k'$). \\

\noindent \underline{Type 3:} for all \footnote{We only allow $i = j$ in the following if $k > 0$, and $i' = j'$ if $k' > 0$ \label{foot}} $i \leq j$, $i' \leq j'$ and $k,k' \in \BZ$ such that:
$$
\det \begin{pmatrix} k & k' \\ j-i & j'-i' \end{pmatrix} = \gcd(k+k',j-i+j'-i')
$$
we set:
\begin{equation}
\label{eqn:case 3}
p_{\pm [i;j)}^{(\pm k)} \tp_{\pm [i';j')}^{(\pm k')} q^{\delta_{j'}^i - \delta_{i'}^i} - \tp_{\pm [i';j')}^{(\pm k')} p_{\pm [i;j)}^{(\pm k)} q^{\delta_{j'}^j - \delta_{i'}^j}  = \sum^{(s,t) \equiv (i',j')}_{i \leq t \text{ and } s \leq j} \frac {\be^\mu_{\pm [s,j)} e_{\pm [i;t)}^\mu}{q^{-1}-q} 
\end{equation}
where $\mu = \frac {k+k'}{j - i + j' - i'}$. \\

\noindent \underline{Type 4:} for all \textsuperscript{\emph{\ref{foot}}} $i \leq j$, $i' \leq j'$ and $k,k' \in \BZ$ such that:
$$
\det \begin{pmatrix} k & k' \\ j-i & j'-i' \end{pmatrix} = \gcd(k-k',j-i-j'+i')
$$
we set:
\begin{equation}
 \label{eqn:case 4}
\left[p_{\pm [i;j)}^{(\pm k)}, p_{\mp [i';j')}^{(\mp k')} \right] = \frac 1{q-q^{-1}}  \begin{cases} \displaystyle \sum^{(s,t) \equiv (i',j')}_{i \leq s \leq t \leq j} e^{\mu}_{\pm [t,j)} \frac {\psi_{j'}^{\pm 1} \barc^{\pm k'}}{\psi_{i'}^{\pm 1}} \be_{\pm [i;s)}^{\mu} &\text{if } j-i \geq j'-i' \\ \\
\displaystyle \sum^{(s,t) \equiv (i,j)}_{i' \leq s \leq t \leq j'} \be_{\mp [t,j')}^{\mu} \frac{\psi_j^{\pm 1}  \barc^{\pm k}}{\psi_i^{\pm 1}} e_{\mp [i';s)}^{\mu} &\text{if } j-i < j'-i'
\end{cases}
\end{equation}
where $\mu = \frac {k-k'}{j - i - j' + i'}$. \\ 

\end{theorem}

\subsection{} 
\label{sub:case 1}

Let us define the following analogues, in the algebra $\CA$, of formulas \eqref{eqn:tilde p}--\eqref{eqn:tilde p equivalent}:
\begin{align}
&\tP_{\pm [i;j)}^{(\pm k)} =  \sum_{x \in \BZ/n\BZ} P_{\pm [i+x;j+x)}^{(\pm k)} \left[q^{-\delta_j^i} \delta_x^0 +  \delta_j^i (q-q^{-1}) \frac {\oq_\pm^{\frac 2n \cdot \overline{-kx}}}{\oq_\pm^2-1} \right] \label{eqn:tilde P} \\
&P_{\pm [i;j)}^{(\pm k)} =  \sum_{x \in \BZ/n\BZ} \tP_{\pm [i+x;j+x)}^{(\pm k)} \left[q^{-\delta_j^i} \delta_x^0 +  \delta_j^i (q-q^{-1}) \frac {\oq_\mp^{\frac 2n \cdot \overline{kx}}}{\oq_\mp^2-1} \right] \label{eqn:tilde P equivalent}
\end{align}
The first of these formulas is a definition, and the second one is a property. \\

\begin{proposition}
\label{prop:plus plus 1}

Formula \eqref{eqn:case 1} holds in $\A \cong \UU$ with $p \leftrightarrow P$. \\

\end{proposition}

\begin{proof} Let us first assume $j-i > 0 $ and $l > 0$, and let:
$$
\lhs = \Big[ P_{\pm [i;j)}^{(\pm k)}, P_{\pm l\bde, \hi}^{(\pm k')} \Big] = P_{\pm [i;j)}^{(\pm k)} P_{\pm l\bde, \hi}^{(\pm k')} - P_{\pm l\bde, \hi}^{(\pm k')} P_{\pm [i;j)}^{(\pm k)}
$$
Formulas \eqref{eqn:primitive 1}/\eqref{eqn:primitive 2} and \eqref{eqn:condition 1}/\eqref{eqn:condition 2} imply that:
\begin{align}
&\Delta \left(P_{\pm [i;j)}^{(\pm k)}\right) = * \otimes P_{\pm [i;j)}^{(\pm k)} + P_{\pm [i;j)}^{(\pm k)} \otimes * + ... \label{eqn:zo1} \\
&\Delta \left( P_{\pm l\bde, \hi}^{(\pm k')}\right) = * \otimes  P_{\pm l\bde, \hi}^{(\pm k')} +  P_{\pm l\bde, \hi}^{(\pm k')} \otimes * + ... \label{eqn:zo2} 
\end{align}
where the $*$'s stand for various products of $\psi_s^{\pm 1}$, $c^{\pm 1}$, $\barc^{\pm 1}$, and the ellipses stand for tensors with hinge strictly below the vectors $\pm (j-i,k)$ and $\pm (nl,k')$, respectively (see Subsection \ref{sub:diagrams} for the definition of hinges). The assumption $|knl-k'(j-i)| = \gcd(k',nl)$ implies that there are no lattice points strictly inside the triangle spanned by the aforementioned vectors. Therefore, the commutator of \eqref{eqn:zo1} and \eqref{eqn:zo2} is:
\begin{equation}
\label{eqn:analysis}
\Delta(\lhs) = * \otimes \lhs + \lhs \otimes * + ...
\end{equation}
where the ellipsis denotes tensors with hinge situated strictly below $\pm (j-i+nl,k+k')$. As a consequence of our assumption, we have $\gcd(j-i+nl,k+k') = 1$, hence:
$$
\lhs \text{ is a primitive element of } \CB_\mu, \text{ for } \mu = \frac {k+k'}{j-i+nl}
$$
Since the right-hand side of relation \eqref{eqn:case 1} is also a primitive element of $\CB_\mu$, Lemma \ref{lem:unique} reduces our problem to proving that the two sides of \eqref{eqn:case 1} take the same values under the linear maps $\alpha_{\pm [u;u+j-i+nl)}$, for all $u \in \BZ/n\BZ$:
$$
\alpha_{\pm [u;u+j-i+nl)}(\lhs) = \alpha_{\pm [u;u+j-i+nl)}(P_{\pm [i;j)}^{(\pm k)} P_{\pm l\bde, \hi}^{(\pm k')}) - \alpha_{\pm [u;u+j-i+nl)}(P_{\pm l\bde, \hi}^{(\pm k')}  P_{\pm [i;j)}^{(\pm k)}) 
$$
\begin{multline*}
\stackrel{\eqref{eqn:pseudo}}= \alpha_{\pm [u+nl;u+j-i+nl)}(P_{\pm [i;j)}^{(\pm k)}) \alpha_{\pm [u;u+nl)}(P_{\pm l\bde, \hi}^{(\pm k')}) \oq_\pm^{\frac dn} - \\ - \alpha_{\pm [u+j-i;u+j-i+nl)}(P_{\pm l\bde, \hi}^{(\pm k')}) \alpha_{\pm [u;u+j-i)}(P_{\pm [i;j)}^{(\pm k)}) \oq_\pm^{- \frac dn} \stackrel{\eqref{eqn:normalize simple}, \eqref{eqn:normalize imaginary}}=
\end{multline*}
$$
= \delta_u^i \delta_{i \text{ mod }g}^r \oq_\pm^{\frac dn} - \delta_u^i \delta_{j \text{ mod }g}^r \oq_\pm^{- \frac dn} \stackrel{\eqref{eqn:normalize simple}}= \alpha_{[u;u+j-i+nl)}(\rhs)
$$
where $d = knl-k'(j-i)$. This proves \eqref{eqn:case 1} in the case $j-i > 0$ and $l > 0$. \\

\noindent When $j - i > 0$ and $l = 0$, our assumption requires $j = i+1$ and \eqref{eqn:case 1} reads:
\begin{equation}
\label{eqn:new 1}
\Big[ P_{\pm[i;i+1)}^{(\pm k)}, P_{\pm 0\bde, \hi}^{(\pm k')} \Big] = \pm P_{\pm [i;i+1)}^{(\pm k \pm k')} \left(\delta_{i}^r \oq_\pm^{- \frac {k'}n} - \delta_{i+1}^r \oq_\pm^{\frac {k'}n} \right)
\end{equation}
which is a simple application of \eqref{eqn:vertical 3}--\eqref{eqn:vertical 4} and \eqref{eqn:vertical}. \\

\noindent When $j-i = 0$ and $l > 0$, our assumption forces $k=1$, $nl | k'$, and we must prove:
\begin{equation}
\label{eqn:new 2}
\Big[ P_{\pm[i;i)}^{(\pm 1)}, P_{\pm l\bde, 0}^{(\pm k')} \Big] = \pm P_{\pm [i;i+ln)}^{(\pm 1 \pm k')} \left( \oq_\pm^{l} - \oq_\pm^{-l} \right)
\end{equation}
We will prove the case $\pm = +$, and leave the analogous case $\pm = -$ as an exercise to the interested reader. Moreover, to keep the notation simple we will assume $k' = 0$, as the general case can be obtained by simply multiplying all the rational functions below by the monomial $\prod_{i=1}^n \prod_{s=1}^l z_{is}^{\frac {k'}{nl}}$. Let us write:
\begin{equation}
\label{eqn:gits}
P_{l\bde, 0}^{(0)} = R(z_{11},...,z_{1l},...,z_{n1},...,z_{nl})
\end{equation}
where the rational function $R$ has total homogeneous degree $0$. By assumption, the hinge of any tensor in the coproduct $\Delta(R)$ lies strictly below the vector $(nl,0)$. Therefore, the hinge of any tensor in the coproduct of the element:
\begin{equation}
\label{eqn:vars}
\left [P_{[i;i)}^{(1)}, P_{l\bde, 0}^{(0)} \right] \stackrel{\eqref{eqn:vertical 3}}= \tilde{R} := \oq^{\frac {2i-1}n} \left(\sum_{s=1}^l z_{i-1,s} - \sum_{s=1}^l z_{is} \right) \cdot R
\end{equation}
lies strictly below the vector $(nl,1)$. This implies that the expression \eqref{eqn:vars} is some linear combination of the primitive shuffle elements $P_{[1;1+nl)}^{(1)},...,P_{[n;n+nl)}^{(1)}$, hence:
\begin{equation}
\label{eqn:cars}
\text{LHS of \eqref{eqn:new 2}} = \sum_{r=1}^n c_r \cdot P_{[r;r+nl)}^{(1)}
\end{equation}
To determine the constants $c_r$, we apply the linear maps $\alpha_{[r;r+nl)}$ of \eqref{eqn:alpha} to the expression above, for any $r \in \{1,...,n\}$. By \eqref{eqn:normalize simple}, we have:
$$
\tilde{R}\Big |_{z_r \mapsto 1,...,z_{r+nl-1} \mapsto 1} \cdot \frac {(1 - q^{-2})^{nl} \oq^{- \frac 1n}}{\oq^{\frac {2(r+nl-1)}n} \prod_{r \leq a < b < r+nl} \zeta(1)} = c_r
$$
for all $r \in \{1,...,n\}$. Recall that the specialization above involves regarding the variables in \eqref{eqn:gits} as having colors $r,r+1,...,r+nl-1$, and specializing them to the value 1 in this convention. This implies that $c_r = 0$ for $r \not \equiv i$, because $\tilde{R}$ is a multiple of the factor $\sum_s (z_{i-1,s} - z_{is})$, which goes to 0 under this specialization. However, when $r = i$, the given specialization corresponds to setting:
$$
z_{i,s} \mapsto \oq^{2s-2},...,z_{n,s} \mapsto \oq^{2s-2}, z_{1,s} \mapsto \oq^{2s},..., z_{i-1,s} \mapsto \oq^{2s}
$$
for all $s \in \{1,...,l\}$, according to convention \eqref{eqn:identify}. Thus, we obtain:
$$
(\oq^{2l}-1) \cdot R \Big |_{z_i \mapsto 1,...,z_{i+nl-1} \mapsto 1} \cdot \frac {(1 - q^{-2})^{nl} \oq^{\frac {2i-1}n} \oq^{- \frac 1n}}{\oq^{\frac {2(i+nl-1)}n}  \prod_{i \leq a < b < i+nl} \zeta(1)} = c_i
$$
To compute the left-hand side of the expression above, we recall that $\alpha_{[i;i+nl)}(R) = 1$ due to \eqref{eqn:normalize imaginary}, hence:
$$
R \Big |_{z_i \mapsto 1,...,z_{i+nl-1} \mapsto 1} \cdot \frac {(1 - q^{-2})^{nl} \oq^{- l}}{\prod_{i \leq a < b < i+nl} \zeta(1)} = 1
$$
By dividing out the previous equalities, we obtain $c_i = \oq^l - \oq^{-l}$. Plugging this formula (as well as $c_r = 0$ for $r \not \equiv i$ mod $n$) into \eqref{eqn:cars} precisely establishes \eqref{eqn:new 2}. 

\end{proof}

\begin{proposition}
\label{prop:plus plus 2}
	
Formula \eqref{eqn:case 3} holds in $\A \cong \UU$ with $p \leftrightarrow P$. \\
	
\end{proposition}

\begin{proof} We will only prove the case when $\pm = +$, as the case $\pm = -$ is analogous. We will first deal with the case when $j > i$ and $j' > i'$, and discuss the situation when we have equality at the end of the proof. To keep the notation simple, we will set $\barc = 1 $ in all formulas for the coproduct (this will have no bearing whatsoever on the validity of the argument). If we use \eqref{eqn:tilde P equivalent}, the relation we need to prove reads:
\begin{equation}
\label{eqn:need to prove}
P_{[i;j)}^{(k)} P_{[i';j')}^{(k')} q^{\delta_{j'}^i - \delta_{i'}^i} - P_{[i';j')}^{(k')} P_{[i;j)}^{(k)} q^{\delta_{j'}^j - \delta_{i'}^j} = 
\end{equation} 
$$
= \sum_{x \in \BZ/n\BZ} \sum^{(s,t) \equiv (i'+x,j'+x)}_{i \leq t \text{ and } s \leq j} \frac {\bE^\mu_{[s,j)} E_{[i;t)}^\mu}{q^{-1}-q} \left[q^{-\delta_{j'}^{i'}} \delta_x^0 +  \delta_{j'}^{i'} (q-q^{-1}) \frac {\oq_-^{\frac 2n \cdot \overline{k'x}}}{\oq_-^2-1} \right]
$$
We will prove this formula by induction on $\# = j - i + j' - i' \in \BN$ (the base case $\# = 1$ is vacuous). We may use \eqref{eqn:p plus} to rewrite the LHS of \eqref{eqn:need to prove} as:
\begin{equation}
\label{eqn:lhs}
\text{LHS} = \left(E_{[i;j)}^{(k)} E_{[i';j')}^{(k')} q^{\delta_{j'}^i - \delta_{i'}^i} - E_{[i';j')}^{(k')} E_{[i;j)}^{(k)} q^{\delta_{j'}^j - \delta_{i'}^j} \right) \frac {q^2 \oq^{\frac 2n}}{(q-q^{-1})^2}
\end{equation}
Our assumption implies that $\gcd(j-i,k) = \gcd(j'-i',k') = 1$, as depicted below:

\begin{picture}(100,65)(-110,-20)
\label{pic:pp}

\put(0,0){\circle*{2}}\put(20,0){\circle*{2}}\put(40,0){\circle*{2}}\put(60,0){\circle*{2}}\put(80,0){\circle*{2}}\put(100,0){\circle*{2}}\put(120,0){\circle*{2}}\put(40,20){\circle*{2}}

\put(0,0){\vector(1,0){120}}
\put(0,0){\vector(2,1){40}}
\put(40,20){\vector(4,-1){80}}

\put(-20,0){\scriptsize{$(0,0)$}}
\put(27,25){\scriptsize{$(j-i,k)$}}
\put(122,0){\scriptsize{$(j-i+j'-i',k+k')$}}

\put(0,15){\scriptsize{$E_{[i;j)}^{(k)}$}}
\put(75,15){\scriptsize{$E_{[i';j')}^{(k')}$}}

\end{picture}

\noindent Therefore, Proposition \ref{prop:coproduct} implies:
\begin{equation}
\label{eqn:louis 1}
\Delta \left(E_{[i;j)}^{(k)} \right) = \frac {\psi_j}{\psi_i} \otimes E_{[i;j)}^{(k)} + E_{[i;j)}^{(k)} \otimes 1 + \sum_{i \leq s < t \leq j} \frac {\psi_j}{\psi_t} E_{[s;t)}^{(\bullet)} \frac {\psi_s}{\psi_i} \otimes \bE_{[t,j)}^{\mu} E_{[i;s)}^{\mu} + ... 
\end{equation}
where $\bullet$ is shorthand for $k - \mu(j-i+s-t)$ and the ellipsis stands for tensors with hinge strictly below the vector $(j-i+j'-i',k+k')$. Similarly, \eqref{eqn:radu1} implies:
\begin{equation}
\label{eqn:louis 2}
\Delta \left(E_{[i';j')}^{(k')} \right) = \frac {\psi_{j'}}{\psi_{i'}} \otimes E_{[i';j')}^{(k')} + E_{[i';j')}^{(k')} \otimes 1 + ... 
\end{equation}
where the ellipsis stands for tensors with hinge strictly below the vector $(j'-i',k')$. Taking an appropriate $q$--commutator of \eqref{eqn:louis 1} and \eqref{eqn:louis 2} yields:
\begin{equation}
\label{eqn:up}
\Delta \left( \text{LHS} \right) = \frac {\psi_j\psi_{j'}}{\psi_i \psi_{i'}} \otimes \text{LHS} + \lhs \otimes 1 + ... \ + 
\end{equation}
$$
+ \left( \frac {\psi_j}{\psi_i} E_{[i';j')}^{(k')} q^{\delta_{j'}^i - \delta_{i'}^i} - E_{[i';j')}^{(k')} \frac {\psi_j}{\psi_i} q^{\delta_{j'}^j - \delta_{i'}^j} \right) \frac {q^2 \oq^{\frac 2n}}{(q-q^{-1})^2} \otimes E_{[i;j)}^{(k)} + 
$$
$$
\sum_{i \leq s < t \leq j}  \left( \frac {\psi_j}{\psi_t} E_{[s;t)}^{(\bullet)} \frac {\psi_s}{\psi_i} E_{[i';j')}^{(k')} q^{\delta_{j'}^i - \delta_{i'}^i} - E_{[i';j')}^{(k')} \frac {\psi_j}{\psi_t} E_{[s;t)}^{(\bullet)} \frac {\psi_s}{\psi_i} q^{\delta_{j'}^j - \delta_{i'}^j} \right) \frac {q^2 \oq^{\frac 2n}}{(q-q^{-1})^2} \otimes  \bE_{[t,j)}^{\mu} E_{[i;s)}^{\mu} 
$$
where the ellipsis (above and henceforth) stands for tensors with hinge strictly below the vector $(j-i+j'-i',k+k')$. By \eqref{eqn:vertical 2}, the second line of the expression above vanishes and hence $\lhs \in \CB_\mu$, while the third line equals:
$$
\sum_{i \leq s < t \leq j} \frac {\psi_j}{\psi_t} \left( E_{[s;t)}^{(\bullet)}  E_{[i';j')}^{(k')} q^{\delta_{j'}^s - \delta_{i'}^s} -  E_{[i';j')}^{(k')} E_{[s;t)}^{(\bullet)} q^{\delta_{j'}^t - \delta_{i'}^t} \right) \frac {q^2 \oq^{\frac 2n}}{(q-q^{-1})^2} \frac {\psi_s}{\psi_i}  \otimes \bE_{[t,j)}^{\mu} E_{[i;s)}^{\mu}
$$
By the induction hypothesis of \eqref{eqn:need to prove}, the expression above equals:
$$
\sum_{x \in \BZ/n\BZ} \sum^{(s',t') \equiv (i'+x,j'+x)}_{s \leq t' \text{ and } s' \leq t} \sum_{i \leq s < t \leq j} \frac {\psi_j}{\psi_t} \bE^{\mu}_{[s',t)} E_{[s;t')}^{\mu} \gamma_{i'j'k'}^-(x) \frac {\psi_s}{\psi_i} \otimes \bE_{[t,j)}^{\mu} E_{[i;s)}^{\mu} = 
$$
\begin{equation}
\label{eqn:multline}
\stackrel{\eqref{eqn:cop1}, \eqref{eqn:cop1 anti}}= \sum_{x \in \BZ/n\BZ} \sum^{(s',t') \equiv (i'+x,j'+x)}_{i \leq t' \text{ and } s' \leq j}  \Delta_\mu(\bE^{\mu}_{[s';j)}) \Delta_\mu(E^{\mu}_{[i;t')}) \gamma_{i'j'k'}^-(x)
\end{equation}
where (recall that $\oq_+ = \oq$ and $\oq_- = q^{-n} \oq^{-1}$, the symbol $\overline{a}$ denotes the residue class of $a$ in the set $\{1,...,n\}$, and $\delta_j^i$ is 1 if $i \equiv j$ mod $n$ and 0 otherwise):
\begin{equation}
\label{eqn:gamma}
\gamma_{ijk}^\pm(x) = \delta_x^0 \frac {q^{-\delta_j^i}}{q^{-1} - q} - \delta_j^i \frac {\oq_{\pm}^{\frac 2n \cdot \overline{kx}}}{\oq_\pm^2 - 1}
\end{equation}
Plugging formula \eqref{eqn:multline} into the third line of \eqref{eqn:up}, we conclude that the intermediate terms in $\Delta_\mu$ of either side of relation \eqref{eqn:need to prove} are equal. Thus, to establish this relation, all that remains to show is that its left and right-hand sides take the same values under the maps $\alpha_{[u;v)}$ for all $[i;j) + [i';j') = [u;v) \in \nn$:
$$
\alpha_{[u;v)}(\text{LHS}) = \left[\alpha_{[u;v)}( E_{[i;j)}^{(k)} E_{[i';j')}^{(k')} ) q^{\delta_{j'}^i - \delta_{i'}^i} -\alpha_{[u;v)} (E_{[i';j')}^{(k')} E_{[i;j)}^{(k)}) q^{\delta_{j'}^j - \delta_{i'}^j} \right] \frac {q^2 \oq^{\frac 2n}}{(q-q^{-1})^2} 
$$
\begin{equation}
\label{eqn:synth 1}
\stackrel{\eqref{eqn:pseudo}, \eqref{eqn:main pair 1}}= \delta_u^{i'} \delta_{j'}^i \oq^{\frac dn} q^{\delta_{j'}^i - \delta_{i'}^i} - \delta_u^i \delta_{i'}^j \oq^{-\frac dn} q^{\delta_{j'}^j - \delta_{i'}^j} 
\end{equation}
where $d = \gcd(k+k',j-i+j'-i')$. Meanwhile, \eqref{eqn:pseudo} and \eqref{eqn:main pair 1}--\eqref{eqn:main pair 2} imply:
$$
\alpha_{[u;v)} (\text{RHS}) = \sum_{x \in \BZ/n\BZ} \sum^{(s,t) \equiv (i'+x,j'+x)}_{i \leq t \text{ and } s \leq j} \alpha_{[u;v)} \left(\bE^{\mu}_{[s,j)} E_{[i;t)}^{\mu}\right) \gamma_{i'j'k'}^-(x) = 
$$
\begin{equation}
\label{eqn:synth 2}
=  (1-q^{-2}) \delta_u^i \oq^{-\frac dn} \gamma_{i'j'k'}^-(j-i') + (1-q^2) \delta_v^j \oq^{\frac dn} \gamma_{i'j'k'}^-(i-i') + 
\end{equation}
$$
+ (1-q^{-2})(1-q^2) \sum^{i < t, s < j, \mu(j-s) \in \BZ}_{[i;t)+[s;j) = [u;v)} \delta_u^i \delta_s^t \delta_j^v \oq^{\frac {\gcd(\mu(j-s), j-s) - \gcd(\mu(t-i), t-i)}n} \gamma_{i'j'k'}^-(s-i')
$$
Hence the proof of the Proposition is completed by the following identity: \\

\begin{claim}
\label{claim:combi}

The right-hand sides of \eqref{eqn:synth 1} and \eqref{eqn:synth 2} are equal. \\

\end{claim}

\begin{proof} Let us assume $i' \equiv j'$, $u \equiv i$, $v \equiv j$ mod $n$, since otherwise the last line of \eqref{eqn:synth 2} vanishes termwise, and the problem is trivial. Therefore:
$$
\text{RHS of \eqref{eqn:synth 2}} = (1-q^{-2}) \oq^{-\frac dn} \gamma_{i'j'k'}^-(j-i') + (1-q^2) \oq^{\frac dn} \gamma_{i'j'k'}^-(i-i') + 
$$
$$
+ (1-q^{-2})(1-q^2) \sum^{e \in \{1,...,d-1\}}_{s = j - (d-e)a, \ t = i+da} \oq^{\frac {d-2e}n} \gamma_{i'j'k'}^-(s-i')
$$
where we write $v-u = da$, $k+k' = db$ with $a$ and $b$ coprime. If we plug formula \eqref{eqn:gamma} into the expression above, then the right-hand side of \eqref{eqn:synth 2} equals:
$$
(1-q^{-2}) \oq^{-\frac dn} \left( \frac {\delta_j^{i'}}{1-q^2} + \frac {\oq_-^{\frac 2n \cdot \overline{k' (j - i')}}}{1 - \oq_-^{2n}} \right) + (1-q^2) \oq^{\frac dn} \left( \frac {\delta_i^{i'}}{1-q^2} + \frac {\oq_-^{\frac 2n \cdot \overline{k'(i - i')}}}{1 - \oq_-^{2n}} \right) + 
$$
\begin{equation}
\label{eqn:telescope}
+ (1-q^{-2})(1-q^2) \sum_{e \in \{1,...,d-1\}} \oq^{\frac {d-2e}n} \left( \frac {\delta_{i+ea}^{i'}}{1-q^2} + \frac {\oq_-^{\frac 2n \cdot \overline{k'(i+ea - i')}}}{1 - \oq_-^{2n}} \right)
\end{equation}
By assumption, we have $(k+k')(j'-i') - k'(v-u) = d \Rightarrow b (j'-i') - k'a = 1$. Since $j' \equiv i'$, this implies that $k'a \equiv -1$ modulo $n$, so we have the elementary identity:
$$
\delta_{i+ea}^{i'} = \oq_-^{-\frac 2n} F_{e-1} - F_e, \qquad \text{where } F_e =  \frac {\oq_-^{\frac 2n \cdot \overline{k'(i + ea - i')}}}{1 - \oq_-^{2n}}
$$
With this substitution, formula \eqref{eqn:telescope} for the right-hand side of \eqref{eqn:synth 2} reads:
$$
(1-q^{-2}) \oq^{-\frac dn} \left( \frac {\delta_j^{i'} \cdot q^2}{1-q^2} + \oq^{-\frac 2n}_- F_{d-1} \right) + (1-q^2) \oq^{\frac dn} \left( \frac {\delta_i^{i'}}{1-q^2} + F_0 \right) + 
$$
$$
+ (1-q^{-2})(1-q^2) \oq^{\frac dn} \sum_{e \in \{1,...,d-1\}} \left( \frac {\oq_-^{-\frac 2n} \oq^{-\frac {2}n}}{1-q^2} \cdot \oq^{-\frac {2(e-1)}n} F_{e-1} - \frac {q^2}{1-q^2} \cdot \oq^{-\frac {2e}n} F_e \right)
$$
Using $\oq_-^{-\frac 2n} \oq^{-\frac 2n} = q^2$, the formula above is easily seen to be a telescoping sum. After canceling the various $F_e$'s, we obtain the right-hand side of \eqref{eqn:synth 1}, as required. 
	
\end{proof}

\noindent The only remaining case is when either $i = j$ or $i' = j'$. In this case, our assumptions force $i = j$, $k = 1$ and $j'-i'|k'+1$. To keep the notation simple we will assume $k' = -1$, as the general case can be obtained by simply multiplying all the rational functions below by the monomial $\prod_{a=i'}^{j'-1} (z_a \oq^{\frac {2a}n})^{\frac {k'+1}{j'-i'}}$. We must prove that:
\begin{equation}
\label{eqn:space}
\left[ P_{[i;i)}^{(1)}, \tP_{[i';j')}^{(-1)} \right] = q^{\delta_{i'}^i - \delta_{j'}^i} \sum^{(s,t) \equiv (i',j')}_{s \leq i \leq t} \frac {\bE^0_{[s,i)} E_{[i;t)}^0}{q^{-1}-q}
\end{equation}
Recall that:
$$
P_{[i';j')}^{(-1)} = \frac {q\oq^{\frac 1n} E_{[i';j')}^{(-1)}}{q-q^{-1}} 
$$
Since the shuffle elements $\tP_{[i;j')}^{(-1)}$ are connected to the shuffle elements $P_{[i';j')}^{(-1)}$ by the linear transformation \eqref{eqn:tilde P}, then formulas \eqref{eqn:minimal e 1} and \eqref{eqn:minimal e 2} show that:
$$
\tP_{[i';j')}^{(-1)} = \frac {\oq^{\frac 1n}F_{[i';j')}^{(-1)}}{q-q^{-1}} 
$$
where, in notation analogous to that of Section \ref{sec:final}, we set:
\begin{align*}
&F_{[i';j')}^{(0) \text{ or } (-1)} = q^{j'-i'}  \cdot \sym \left[ \frac { \left( z_{i'} \oq^{\frac {2i'}n} \right)^{0 \text{ or } -1}}{\left(1 - \frac {z_{i'+1}}{z_{i'}}  \right) ... \left(1 - \frac {z_{j'-1}}{z_{j'-2}} \right)} \prod_{i'\leq a < b < j'} \zeta \left( \frac {z_a}{z_b} \right) \right] \\
&\bF_{[i';j')}^{(0) \text{ or } (-1)} = (-q\oq^{\frac 2n})^{j'-i'} \cdot \sym \left[ \frac { \left( z_{j'-1} \oq^{\frac {2(j'-1)}n} \right)^{0 \text{ or } -1}}{\left(1 - \frac {z_{i'}}{z_{i'+1}}  \right) ... \left(1 - \frac {z_{j'-2}}{z_{j'-1}} \right)} \prod_{i'\leq a < b < j'} \zeta \left( \frac {z_a}{z_b} \right) \right]
\end{align*}
By \eqref{eqn:vertical 3} and \eqref{eqn:vertical}, we therefore have:
$$
\left[ P_{[i;i)}^{(1)}, \tP_{[i';j')}^{(-1)} \right] = \frac {q^{j'-i'}}{q-q^{-1}}  \cdot \sym \left[ \frac {\sum_{i' \leq a < j'}^{a \equiv i-1} \frac {z_a}{z_{i'}} \oq^{\frac {2(a-i'+1)}n} -  \sum_{i' \leq a < j'}^{a \equiv i} \frac {z_a}{z_{i'}} \oq^{\frac {2(a-i')}n}}{\left(1 - \frac {z_{i'+1}}{z_{i'}}  \right) ... \left(1 - \frac {z_{j'-1}}{z_{j'-2}} \right)} \right. 
$$
$$
\left. \prod_{i'\leq a < b < j'} \zeta \left( \frac {z_a}{z_b} \right) \right] =  \frac {q^{j'-i'}}{q-q^{-1}}  \cdot \sym \left[ \frac {\delta_{j'}^i \frac {z_{j'-1}}{z_{i'}} \oq^{\frac {2(j'-i')}n} - \delta_{i'}^i + \sum_{i'<a<j'}^{a \equiv i} \frac {z_{a-1}-z_a}{z_{i'}} \oq^{\frac {2(a-i')}n}}{\left(1 - \frac {z_{i'+1}}{z_{i'}}  \right) ... \left(1 - \frac {z_{j'-1}}{z_{j'-2}} \right)} \right. 
$$
$$
\left. \prod_{i'\leq a < b < j'} \zeta \left( \frac {z_a}{z_b} \right) \right] = \frac 1{q-q^{-1}} \left(- \delta_{j'}^i \cdot \bF_{[i';j')}^{(0)}  - \delta_{i'}^i \cdot F_{[i';j')}^{(0)} - \sum_{i' < a < j'}^{a\equiv i} \bF_{[i';a)}^{(0)} F_{[a;j')}^{(0)} \right) 
$$
In the notation of \eqref{eqn:formula y} and \eqref{eqn:formula zz}, the expression above equals:
$$
\frac {\bZ_{[i';j'),[i;i)}^0}{q^{-1}-q} \stackrel{\eqref{eqn:identity yz}}= q^{\delta_{i'}^i - \delta_{j'}^i} \cdot \frac {Y_{[i;i),[i';j')}^0}{q^{-1}-q}
$$
(see Remark \ref{rem:works} for the reason why we are allowed to apply Proposition \ref{prop:identity 1} for $\mu = 0$). Since the right-hand side of the expression above is identical to the right-hand side of \eqref{eqn:space}, the proof is complete.

\end{proof}

\subsection{} 
\label{sub:pairings intermezzo}

For all collections of indices for which the two sides of the pairings below have complementary degrees, formulas \eqref{eqn:bonnie} and \eqref{eqn:normalize simple}/\eqref{eqn:normalize imaginary} imply:
\begin{align}
&\Big \langle P_{\pm [i;j)}^{(\pm k)}, E_{\mp [i';j')}^{(\mp k)} \Big \rangle = \pm \delta^{(i,j)}_{(i',j')} \cdot \oq_\pm^{\frac {\gcd(k,j-i)}n} \label{eqn:main pair 3} \\
&\Big \langle P_{\pm l\bde, \hi}^{(\pm k)}, E_{\mp [i';j')}^{(\mp k)} \Big \rangle = \pm \delta_{i' \text{ mod }g}^r \cdot \oq_\pm^{\frac {\gcd(k,nl)}n} \label{eqn:main pair 4}
\end{align}
where $g = \gcd(n,\text{denominator } \frac k{nl})$. Then formula \eqref{eqn:antipode} implies:
\begin{align}
&\Big \langle P_{\pm [i;j)}^{(\pm k)}, \bE_{\mp [i';j')}^{(\mp k)} \Big \rangle = \mp \delta^{(i,j)}_{(i',j')} \cdot \oq_\pm^{-\frac {\gcd(k,j-i)}n} \label{eqn:main pair 5} \\
&\Big \langle P_{\pm l\bde, \hi}^{(\pm k)}, \bE_{\mp [i';j')}^{(\mp k)} \Big \rangle = \mp \delta_{i' \text{ mod }g}^r \cdot \oq_\pm^{-\frac {\gcd(k,nl)}n} \label{eqn:main pair 6}
\end{align}
where we used the fact that the $P$'s pair trivially with products of two or more $E$'s or $\bE$'s of the same slope (because the $P$'s are primitive). Note that \eqref{eqn:p plus} implies:
\begin{equation}
\label{eqn:pair simple}
\Big \langle P_{[i;j)}^{(k)}, P_{-[i';j')}^{(-k)} \Big \rangle = \frac {\delta_{(i',j')}^{(i,j)}}{q^{-1}-q}
\end{equation}
whenever $\gcd(j-i,k) = 1$. A straightforward reformulation of Lemma \ref{lem:elem} implies:
\begin{equation}
\label{eqn:pair simple tilde}
\Big \langle \tP_{[i;j)}^{(k)}, \tP_{-[i';j')}^{(-k)} \Big \rangle = \frac {\delta_{(i',j')}^{(i,j)}}{q^{-1}-q}
\end{equation}
where $\tP_{\pm [i;j)}^{(\pm k)}$ is connected to $P_{\pm [i;j)}^{(\pm k)}$ by \eqref{eqn:tilde P}. \\ 

\subsection{}
\label{sub:analogue}

Consider the following analogue of Proposition \ref{prop:coproduct}: \\

\begin{proposition}
\label{prop:coproduct 2}
	
For any $k\in \BZ$ and $l \in \BN$, consider the diagram:
	
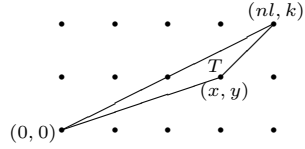
\begin{picture}(80,65)(-120,30)
	
\put(0,40){\circle*{2}}\put(20,40){\circle*{2}}\put(40,40){\circle*{2}}\put(60,40){\circle*{2}}\put(80,40){\circle*{2}}\put(0,60){\circle*{2}}\put(20,60){\circle*{2}}\put(40,60){\circle*{2}}\put(60,60){\circle*{2}}\put(80,60){\circle*{2}}\put(0,80){\circle*{2}}\put(20,80){\circle*{2}}\put(40,80){\circle*{2}}\put(60,80){\circle*{2}}\put(80,80){\circle*{2}}
	
\put(0,40){\line(2,1){80}}
\put(0,40){\line(3,1){60}}
\put(60,60){\line(1,1){20}}
	
\put(-20,37){\scriptsize{$(0,0)$}}
\put(70,83){\scriptsize{$(nl,k)$}}
\put(52,53){\scriptsize{$(x,y)$}}
\put(55,62){\scriptsize{$T$}}
	
\end{picture}
	
\noindent and let $T$ have minimal area among all lattice triangles contained strictly below the vector $(nl,k)$ (there are exactly $d := \gcd(nl,k)$ such triangles). Then we have:
$$
\Delta \left( P_{l\bde, \hi}^{(k)} \right) = c^l \barc^k \otimes P_{l\bde, \hi}^{(k)} + P_{l\bde, \hi}^{(k)} \otimes 1 + \Big(\text{tensors with hinge strictly below }T \Big) + 
$$
\begin{equation}
\label{eqn:new cop prim 1}
+ q^{\delta_x^0}(1-q^{-2}) \sum_{i = 1}^n P_{[i;i+nl-x)}^{(k-y)} \frac{\psi_{i}\barc^y}{\psi_{i-x}}  \otimes \tP_{[i-x;i)}^{(y)} \left(\delta_{i-x \text{ mod }g}^{\hi} \oq_+^{- \frac dn}   -  \delta_{i \text{ mod }g}^{\hi} \oq_+^{\frac dn} \right) 
\end{equation} 
where $g = \gcd \left(n,\frac {nl}d \right)$. Similarly, consider the diagram:
	
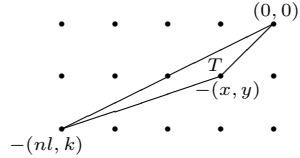
\begin{picture}(80,65)(-120,27)
	
\put(0,40){\circle*{2}}\put(20,40){\circle*{2}}\put(40,40){\circle*{2}}\put(60,40){\circle*{2}}\put(80,40){\circle*{2}}\put(0,60){\circle*{2}}\put(20,60){\circle*{2}}\put(40,60){\circle*{2}}\put(60,60){\circle*{2}}\put(80,60){\circle*{2}}\put(0,80){\circle*{2}}\put(20,80){\circle*{2}}\put(40,80){\circle*{2}}\put(60,80){\circle*{2}}\put(80,80){\circle*{2}}
	
\put(0,40){\line(2,1){80}}
\put(0,40){\line(3,1){60}}
\put(60,60){\line(1,1){20}}
	
\put(-20,32){\scriptsize{$-(nl,k)$}}
\put(72,83){\scriptsize{$(0,0)$}}
\put(50,53){\scriptsize{$-(x,y)$}}
\put(55,62){\scriptsize{$T$}}
	
\end{picture}
	
\noindent and let $T$ have minimal area as described above. Then we have:
$$
\Delta \left( P_{-l\bde, \hi}^{(-k)} \right) = 1 \otimes P_{-l\bde, \hi}^{(-k)} + P_{-l\bde, \hi}^{(-k)} \otimes c^{-l} \barc^{-k} + \Big(\text{tensors with hinge strictly below }T \Big) +
$$
\begin{equation}
\label{eqn:new cop prim 2}
+ q^{\delta_x^0}(1-q^{-2}) \sum_{i = 1}^n \tP_{-[i;i+nl-x)}^{(-k+y)} \otimes P_{-[i-x;i)}^{(-y)}  \frac{\psi_{i}\barc^{y-k}}{\psi_{i+nl-x}}  \left(\delta_{i - x \text{ mod }g}^{\hi} \oq_-^{\frac dn}  -  \delta_{i \text{ mod }g}^{\hi} \oq_-^{- \frac dn}  \right)
\end{equation}
	
\end{proposition}

\begin{proof} We will prove \eqref{eqn:new cop prim 1}, and leave the analogous formula \eqref{eqn:new cop prim 2} as an exercise to the interested reader. As a consequence of \eqref{eqn:condition 1} and the definition of hinges, we conclude that any summand $R_1^+ \otimes R_2^+$ that appears in the LHS of \eqref{eqn:new cop prim 1} has hinge below the vector $(nl,k)$. Now fix such a summand with hinge exactly equal to $(x,y)$ as depicted in the figure above. Since the coproduct is coassociative, the tensor factors $R_1^+$ and $R_2^+$ each have the property that any summand in their coproduct has hinge below the vector $(nk,l)$, and hence below the triangle $T$ by the minimality hypothesis. Therefore, $R_1^+$ and $R_2^+$ must be primitive:
\begin{multline} 
\label{eqn:north}
\Delta \left( P_{l\bde, \hi}^{(k)} \right) = c^l \barc^k \otimes P_{l\bde, \hi}^{(k)} + P_{l\bde, \hi}^{(k)} \otimes 1 + \\ + \sum_{i,i' = 1}^n \nu(i,i')  P_{[i;i+nl-x)}^{(k-y)} \frac{\psi_{i'+x}}{\psi_{i'}} \barc^y \otimes \tP_{[i';i'+x)}^{(y)} + ...
\end{multline}
for some scalars $\nu(i,i')$, where the ellipsis stands for summands with hinge strictly below $T$. It therefore remains to determine these scalars, and by \eqref{eqn:pair simple} we have:
$$
\frac {\nu(i,i')}{(q^{-1}-q)^2} = \left \langle \Delta \left( P_{l\bde, \hi}^{(k)} \right), P_{-[i;i+nl-x)}^{(y-k)} \otimes \tP_{-[i';i'+x)}^{(-y)} \right \rangle \stackrel{\eqref{eqn:bialg 2}}= \left \langle  P_{l\bde, \hi}^{(k)}, P_{-[i;i+nl-x)}^{(y-k)} \tP_{-[i';i'+x)}^{(-y)} \right \rangle
$$
The product of $P$'s in the right-hand side satisfies the hypotheses of \eqref{eqn:case 3}, hence:
\begin{multline*}
\frac {\nu(i,i')}{(q^{-1}-q)^2} = \Big \langle  P_{l\bde, \hi}^{(k)}, q^{\delta_{i'}^i - \delta_{i'}^{i-x} + \delta_{i'+x}^{i-x} - \delta_{i'+x}^{i}} \tP_{-[i';i'+x)}^{(-y)} P_{-[i;i+nl-x)}^{(y-k)} + \\ \left. + \frac {q^{\delta_{i'}^i - \delta_{i'+x}^i}}{q^{-1}-q} \sum_{(s,t) \equiv (i',i'+x)} \bE_{-[s,i+nl-x)}^{\frac k{nl}} E_{-[i;t)}^{\frac k{nl}} \right \rangle
\end{multline*}
Note that $P_{l\bde, \hi}^{(k)}$ pairs trivially with all products of more than one $P, \tP,  \bE, E$ in the formula above, as a consequence of \eqref{eqn:radu1} and \eqref{eqn:condition 1}. Therefore, we conclude that:
$$
\frac {\nu(i,i')}{q^{-1}-q} = q^{\delta_{i'}^i - \delta_{i'+x}^i} \left \langle  P_{l\bde, \hi}^{(k)} , E_{-[i;i+nl)}^{\frac k{nl}} \delta_{i'}^{i-x} + \bE_{-[i-x;i+nl-x)}^{\frac k{nl}} \delta_i^{i'+x} \right \rangle =
$$
$$
\stackrel{\eqref{eqn:main pair 4}, \eqref{eqn:main pair 6}}= \delta_{i'+x}^i q^{\delta_x^0 - 1} \left( \delta_{i \text{ mod }g}^{\hi} \oq^{\frac dn} -  \delta_{i-x \text{ mod }g}^{\hi} \oq^{- \frac dn} \right)
$$
Plugging this formula in \eqref{eqn:north} implies \eqref{eqn:new cop prim 1}. 

\end{proof}

\begin{proposition}
\label{prop:switch} 
	
In the second line of either \eqref{eqn:new cop prim 1} or \eqref{eqn:new cop prim 2}, one could move the tilde from the first $P$ to the second $P$, without changing the values of these formulas. \\
	
\end{proposition}

\begin{proof} We will prove the claim which pertains to \eqref{eqn:new cop prim 1}, since the case of \eqref{eqn:new cop prim 2} is analogous. The Proposition is only non-trivial if $n|x$, which we henceforth assume. The assumption on the minimal triangle $T$ implies that:
\begin{equation}
\label{eqn:gcd}
xk - ynl = d
\end{equation}
and it is an elementary exercise (which we leave to the interested reader) to show that $n|x$ implies that $g = \gcd(n,\frac {nl}{d})=1$. Also, since $n|x$ then \eqref{eqn:tilde P}--\eqref{eqn:tilde P equivalent} read:
\begin{align*} 
&P_{[i;i+nl-x)}^{(k-y)} = \sum_{s=1}^n \tP_{[s;s+nl-x)}^{(k-y)} \left[q^{-1} \delta_s^i + (q-q^{-1}) \frac {\oq_-^{\frac 2n \cdot \overline{(k-y)(s-i)}}}{\oq_-^2-1} \right] \\
&\tP_{[i-x;i)}^{(y)} = \sum_{s=1}^n P_{[s-x;s)}^{(y)} \left[q^{-1}\delta_s^i + (q-q^{-1}) \frac {\oq_+^{\frac 2n \cdot \overline{y(i-s)}}}{\oq_+^2-1} \right]
\end{align*} 	
Therefore, we have:
$$
\sum_{i=1}^n P_{[i;i+nl-x)}^{(k-y)} \otimes \tP_{[i-x;i)}^{(y)} = \sum_{s,s'=1}^n \tP_{[s;s+nl-x)}^{(k-y)} \otimes P_{[s'-x;s')}^{(y)} \cdot 
$$
$$
\sum_{i=1}^n \left[q^{-1} \delta_s^i + (q-q^{-1}) \frac {\oq_-^{\frac 2n \cdot \overline{(k-y)(s-i)}}}{\oq_-^2-1} \right] \left[q^{-1}\delta_{s'}^i + (q-q^{-1}) \frac {\oq_+^{\frac 2n \cdot \overline{y(i-s')}}}{\oq_+^2-1} \right]
$$
As a consequence of Lemma \ref{lem:elem} (which we may apply because $n|x \Rightarrow n|d \Rightarrow n|k$), the second line of the expression above is $\delta_{s'}^s$, as we needed to prove.

\end{proof} 

\subsection{} 
\label{sub:case 2}

We will now use the results proved in the previous Subsection in order to complete the proof of Theorem \ref{thm:main}. \\

\begin{proposition}
\label{prop:plus minus 1}

Formula \eqref{eqn:case 2} holds in $\CA \cong \UU$  with $p \leftrightarrow P$. \\
	
\end{proposition} 

\begin{proof} We will only prove the case when $\pm = +$, as the case $\pm = -$ is completely analogous. We divide the proof into several cases, depending on the relative sizes of $j-i$ and $nl$. In the case $j-i>nl > 0$, consider the triangles:
	
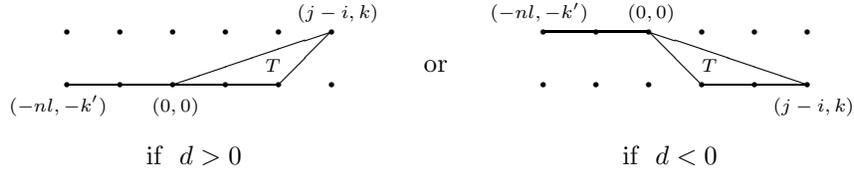
\begin{picture}(250,90)(-30,-46)

\put(0,0){\circle*{2}}\put(20,0){\circle*{2}}\put(40,0){\circle*{2}}\put(60,0){\circle*{2}}\put(80,0){\circle*{2}}\put(100,0){\circle*{2}}\put(0,20){\circle*{2}}\put(20,20){\circle*{2}}\put(40,20){\circle*{2}}\put(60,20){\circle*{2}}\put(80,20){\circle*{2}}\put(100,20){\circle*{2}}

\put(40,0){\line(-1,0){40}}
\put(40,0){\line(3,1){60}}
\put(40,0){\line(1,0){40}}
\put(80,0){\line(1,1){20}}

\put(75,5){\scriptsize{$T$}}
\put(32,-10){\scriptsize{$(0,0)$}}
\put(87,25){\scriptsize{$(j-i,k)$}}
\put(-22,-10){\scriptsize{$(-nl,-k')$}}

\put(135,5){\text{or}}
\put(30,-30){\text{if } $d>0$}
\put(210,-30){\text{if } $d<0$}

\put(180,0){\circle*{2}}\put(200,0){\circle*{2}}\put(220,0){\circle*{2}}\put(240,0){\circle*{2}}\put(260,0){\circle*{2}}\put(280,0){\circle*{2}}\put(180,20){\circle*{2}}\put(200,20){\circle*{2}}\put(220,20){\circle*{2}}\put(240,20){\circle*{2}}\put(260,20){\circle*{2}}\put(280,20){\circle*{2}}

\put(220,20){\line(-1,0){40}}
\put(220,20){\line(3,-1){60}}
\put(220,20){\line(1,-1){20}}
\put(240,0){\line(1,0){40}}

\put(240,5){\scriptsize{$T$}}
\put(212,25){\scriptsize{$(0,0)$}}
\put(267,-10){\scriptsize{$(j-i,k)$}}
\put(160,25){\scriptsize{$(-nl,-k')$}}

\end{picture}	

\noindent Denote $a = P_{[i;j)}^{(k)}$ and $b = P_{- l\bde, \hi}^{(-k')}$. If we let $\mu = \frac {k'}{nl}$, then \eqref{eqn:new coproduct 1} implies:
$$
\Delta(a) = \frac {\psi_j}{\psi_i} \barc^{k} \otimes a + a \otimes 1 + \sum_{i \leq s < t \leq j} \frac {\psi_j}{\psi_t} P_{[s;t)}^{(\bullet)} \frac {\psi_s}{\psi_i} \barc^{k-\bullet} \otimes \bE_{[t,j)}^{\mu} E_{[i;s)}^{\mu} + ... \quad \text{if } d>0
$$
$$
\Delta(a) = \frac {\psi_j}{\psi_i} \barc^{k} \otimes a + a \otimes 1 + \sum_{i \leq s < t \leq j} E_{[t,j)}^{\mu} \frac{\psi_t}{\psi_s} \bE_{[i;s)}^{\mu} \barc^\bullet\otimes P_{[s;t)}^{(\bullet)} + ... \qquad \ \ \quad \text{if } d<0
$$
where $\bullet = k-\mu(j-i+s-t)$ and the ellipsis denotes tensors with hinge below the triangle $T$.  Meanwhile, \eqref{eqn:radu2} and the fact that $b$ is primitive imply that:
$$
\Delta(b) = 1 \otimes b + b \otimes c^{-l} \barc^{-k'} + ...
$$		
where the ellipsis stands for tensors with hinge strictly below the vector $(-nl,-k')$. Then the only non-trivial pairings in relation \eqref{eqn:drinfeld} for our choice of $a$, $b$ are:
\begin{align*}
&ab + \sum_{i \leq s < t \leq j} \frac {\psi_j}{\psi_t} P_{[s;t)}^{(\bullet)} \frac {\psi_s}{\psi_i} \barc^{k-\bullet} \left \langle \bE_{[t,j)}^{\mu} E_{[i;s)}^{\mu} , b\right \rangle = ba \qquad \text{if } d>0 \\
&ab = ba + \sum_{i \leq s < t \leq j} c^{-l} \barc^{-k'} P_{[s;t)}^{(\bullet)} \left \langle E_{[t,j)}^{\mu} \frac{\psi_t}{\psi_s} \bE_{[i;s)}^{\mu} \barc^\bullet, b \right \rangle \quad \text{if } d<0
\end{align*}
The fact that $b$ is primitive implies that it pairs trivially with any non-trivial product between an $E^\mu$ and an $\bE^\mu$, so the only non-zero pairings above are those for $(s,t) \equiv (i,j-nl)$ and $(i+nl,j)$ modulo $(n,n)$. Using \eqref{eqn:main pair 4} and \eqref{eqn:main pair 6}, the formula above yields precisely \eqref{eqn:case 2}. \\

\noindent In the case $nl > j-i > 0$, consider the triangles:

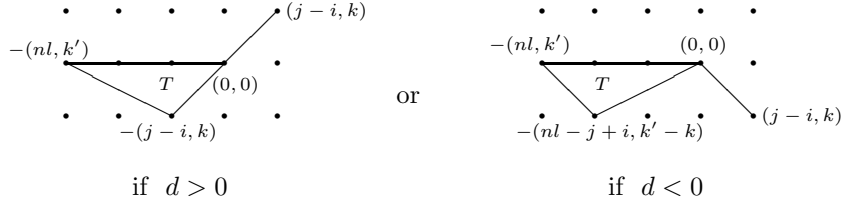
\begin{picture}(250,90)(-30,-41)

\put(0,0){\circle*{2}}\put(20,0){\circle*{2}}\put(40,0){\circle*{2}}\put(60,0){\circle*{2}}\put(80,0){\circle*{2}}\put(0,20){\circle*{2}}\put(20,20){\circle*{2}}\put(40,20){\circle*{2}}\put(60,20){\circle*{2}}\put(80,20){\circle*{2}}\put(0,40){\circle*{2}}\put(20,40){\circle*{2}}\put(40,40){\circle*{2}}\put(60,40){\circle*{2}}\put(80,40){\circle*{2}}

\put(0,20){\line(1,0){60}}
\put(0,20){\line(2,-1){40}}
\put(40,0){\line(1,1){40}}

\put(35,10){\scriptsize{$T$}}
\put(55,10){\scriptsize{$(0,0)$}}
\put(83,38){\scriptsize{$(j-i,k)$}}
\put(20,-8){\scriptsize{$-(j-i,k)$}}
\put(-22,23){\scriptsize{$-(nl,k')$}}

\put(125,5){\text{or}}
\put(25,-30){\text{if } $d>0$}
\put(205,-30){\text{if } $d<0$}

\put(180,0){\circle*{2}}\put(200,0){\circle*{2}}\put(220,0){\circle*{2}}\put(240,0){\circle*{2}}\put(260,0){\circle*{2}}\put(180,20){\circle*{2}}\put(200,20){\circle*{2}}\put(220,20){\circle*{2}}\put(240,20){\circle*{2}}\put(260,20){\circle*{2}}\put(180,40){\circle*{2}}\put(200,40){\circle*{2}}\put(220,40){\circle*{2}}\put(240,40){\circle*{2}}\put(260,40){\circle*{2}}

\put(180,20){\line(1,0){60}}
\put(180,20){\line(1,-1){20}}
\put(240,20){\line(1,-1){20}}
\put(200,0){\line(2,1){40}}

\put(200,10){\scriptsize{$T$}}
\put(232,25){\scriptsize{$(0,0)$}}
\put(263,-2){\scriptsize{$(j-i,k)$}}
\put(170,-8){\scriptsize{$-(nl-j+i,k'-k)$}}
\put(160,25){\scriptsize{$-(nl,k')$}}

\end{picture}	

\noindent Let $a = \tP_{[i;j)}^{(k)}$ and $b = P_{- l\bde, \hi}^{(-k')}$. By \eqref{eqn:radu1}, we have:
$$
\Delta(a) = \frac {\psi_j}{\psi_i} \barc^k \otimes a + a \otimes 1 + ...
$$
where the ellipsis denotes tensors with hinge strictly below the vector $(j-i,k)$. Meanwhile, \eqref{eqn:new cop prim 2} for $-(x,y)$ chosen as the bottom-most vertex of $T$ in the diagrams above yields:
$$
\Delta (b) = 1 \otimes b + b \otimes c^{-l} \barc^{-k} + ... + q^{\delta_j^i} (1-q^{-2}) \cdot 
$$ 
$$
\sum_{s = 1}^n  \begin{cases} P_{-[s;s+nl-j+i)}^{(k-k')} \otimes \tP_{-[s-j+i;s)}^{(-k)} \frac {\psi_{s}\barc^{k-k'}}{\psi_{s+nl-j+i}} \left(\delta_{s-j+i \text{ mod }g}^{\hi} \oq_-^{\frac dn}  -  \delta_{s \text{ mod }g}^{\hi} \oq_-^{- \frac dn}  \right) &\text{if } d>0 \\ \\  \tP_{-[s;s+j-i)}^{(-k)} \otimes P_{-[s-nl+j-i;s)}^{(k-k')} \frac {\psi_{s}\barc^{-k} }{\psi_{s+j-i}} \left(\delta_{s+j-i \text{ mod }g}^{\hi} \oq_-^{-\frac dn}  -  \delta_{s \text{ mod }g}^{\hi} \oq_-^{\frac dn}  \right) &\text{if } d<0  \end{cases}
$$
where the ellipsis denotes tensors with hinge strictly below the triangle $T$ (note that we invoked Proposition \ref{prop:switch} in the $d>0$ case above). Therefore, the only terms which appear non-trivially in relation \eqref{eqn:drinfeld} for our choice of $a$ and $b$ are:
\begin{align*} 
&ab - q^{\delta_j^i-1} \frac {\psi_j \barc^k}{\psi_i} P^{(k-k')}_{-[j;i+nl)} \left(\delta_{i \text{ mod }g}^{\hi} \oq_-^{\frac dn}  -  \delta_{j \text{ mod }g}^{\hi} \oq_-^{- \frac dn}  \right) =  ba \qquad \quad \text{if } d>0 \\
&ab = ba + q^{\delta_j^i-1} P_{-[j-nl;i)}^{(k-k')} \frac {\psi_{i}\barc^{-k} }{\psi_{j}} \left(\delta_{i \text{ mod }g}^{\hi} \oq_-^{\frac dn} - \delta_{j \text{ mod }g}^{\hi} \oq_-^{-\frac dn}\right) \qquad \ \text{if } d<0
\end{align*}
This is precisely equivalent to \eqref{eqn:case 2}, as we needed to prove. \\

\noindent When $j-i = nl > 0$, the assumption is only satisfied if $nl|k'$ and $k = k' \pm 1$. We will prove the case when $k = k' + 1$, and leave the analogous case of $k = k' - 1$ as an exercise to the interested reader. Moreover, to keep the notation simple we will assume $k' = 0 \Rightarrow k = 1$, as the general case can be obtained by simply multiplying all the rational functions below by the monomial $\prod_{a=i}^{j-1} (z_a \oq^{\frac {2a}n})^{\frac {k'}{nl}}$. Let us write:
$$
P_{[i;j)}^{(1)} = R(z_i,...,z_{j-1})
$$
as an element of the shuffle algebra. By \eqref{eqn:coproduct 3}--\eqref{eqn:coproduct 4} and \eqref{eqn:radu1}--\eqref{eqn:radu2}, we have: 
$$
\Delta(R) = R \otimes 1 + c^l \barc \otimes R  + (q^{-1}-q) \sum_{s=1}^n c^l a_{s,1} \otimes \tilde{R}_s + ...
$$
where $\tilde{R}_s = R \cdot \oq^{\frac {2s-1}n}\left[ \sum_{i \leq t < j}^{t \equiv s} q\oq^{\frac 1n} (z_t \oq^{\frac {2t}n})^{-1} - \sum_{i \leq t < j}^{t \equiv s-1} q^{-1}\oq^{-\frac 1n} (z_t \oq^{\frac {2t}n})^{-1}  \right]$, and:
$$
\Delta (P_{-l\bde, r}^{(0)}) = P_{-l\bde, r}^{(0)} \otimes c^{-l} + 1 \otimes P_{-l\bde, r}^{(0)} + ...
$$
where the ellipsis in the two formulas above denote terms whose hinges are too low to pair non-trivially with each other. Because of this, formula \eqref{eqn:drinfeld} therefore reads:
\begin{equation}
\label{eqn:pd}
R \cdot P_{-l\bde, r}^{(0)} + (q^{-1}-q )\sum_{s=1}^n c^l a_{s,1} \left \langle \tilde{R}_s, P_{-l\bde, r}^{(0)} \right \rangle = P_{-l\bde, r}^{(0)} \cdot R 
\end{equation}
The formula above implies \eqref{eqn:case 2}, once we invoke \eqref{eqn:vertical} and the following claim:
\begin{equation}
\label{eqn:uv}
\left \langle \tilde{R}_s, P_{-l\bde, r}^{(0)} \right \rangle = \delta_s^i \cdot \frac {\oq_-^l - \oq_-^{-l}}{q-q^{-1}} \cdot \oq^{\frac {2s-1}n}
\end{equation}
To prove \eqref{eqn:uv}, we will establish the following: \\

\begin{claim} 
\label{claim:tilde}
	
With the notation above, we have:
\begin{equation}
\label{eqn:db}
\tilde{R}_s = - \frac {\oq^{\frac {2s-1}n}}{q-q^{-1}} \left[ \bE_{[i;j)}^{(0)} +  E_{[i;j)}^{(0)} + \sum_{i < t < j}^{t \equiv s} \bE_{[t;j)}^{(0)} E_{[i;t)}^{(0)} \right]
\end{equation}

\end{claim} 

\noindent Then \eqref{eqn:uv} follows from relations \eqref{eqn:main pair 4} and  \eqref{eqn:main pair 6}, as well as the fact that $P_{-l\bde, r}^{(0)}$ is primitive, hence pairs trivially with any product of an $E$ and an $\bE$ in formula \eqref{eqn:db}. \\

\begin{proof} \emph{of Claim \ref{claim:tilde}:} We have:
$$
\tilde{R}_s = \frac {q^2 \oq^{\frac {2s-1}n}}{q-q^{-1}} \cdot \sym \left[ \frac {\sum_{i \leq t < j}^{t \equiv s} \frac {z_{j-1}}{z_t} \oq^{\frac {2(j-t)}n} - \sum_{i \leq t < j}^{t \equiv s-1} \frac {z_{j-1}}{z_t} q^{-2} \oq^{\frac {2(j-t-1)}n}}{\left(1 - \frac {z_{i+1}}{z_{i}q^2}  \right) ... \left(1 - \frac {z_{j-1}}{z_{j-2}q^2} \right)}  \prod_{i\leq a < b < j} \zeta \left( \frac {z_b}{z_a} \right)  \right] = 
$$
$$
= \frac {q^2 \oq^{\frac {2s-1}n}}{q-q^{-1}} \cdot \sym \left[ \frac {\delta_s^i \frac {z_{j-1}}{z_i} \oq^{\frac {2(j-i)}n} - \delta_s^j q^{-2} + \sum^{t \equiv s}_{i<t<j} \frac {z_{j-1}}{z_t} \oq^{\frac {2(j-t)}n} \left(1-\frac {z_t}{z_{t-1} q^2}\right)}{\left(1 - \frac {z_{i+1}}{z_{i}q^2}  \right) ... \left(1 - \frac {z_{j-1}}{z_{j-2}q^2} \right)}  \prod_{i\leq a < b < j} \zeta \left( \frac {z_b}{z_a} \right)  \right]
$$
According to formulas \eqref{eqn:a}, \eqref{eqn:ba}, \eqref{eqn:e plus} and \eqref{eqn:be plus}, the formula above is equal to the right-hand side of \eqref{eqn:db}.

\end{proof}

\noindent When $j-i > nl = 0$, the assumption of Proposition \ref{prop:plus minus 1} is only satisfied if $j=i+1$, in which case the desired relation reads:
$$
\left[ P_{[i;i+1)}^{(k)}, P_{-0\bde, \hi}^{-k'} \right] = \barc^{-k'} P_{[i;i+1)}^{(k-k')} \left( \delta_i^r (q\oq^{\frac 1n})^{k'} - \delta_{i+1}^r (q\oq^{\frac 1n})^{-k'} \right)
$$
If we make the substitution \eqref{eqn:vertical}, this relation is a special case of \eqref{eqn:double 3}. \\

\noindent When $j-i = 0$ and $l > 0$, our assumption forces $k=1$ and $nl | k'$. To keep the notation simple we will assume $k' = 0$, as the general case can be obtained by simply multiplying all the rational functions below by the monomial $\prod_{i=1}^n \prod_{s=1}^l z_{is}^{\frac {k'}{nl}}$. In this case, we have $g=1$, and the relation we must prove reads:
$$
\Big[ \tP_{[i;i)}^{(1)}, P_{- l\bde, 1}^{(0)} \Big] = \barc P^{(1)}_{-[i-nl;i)} \left(\oq_-^l - \oq_-^{-l} \right)
$$
As $P_{[i;i)}^{(1)} = \sum_{u=1}^n \tP_{[u;u)}^{(1)} \displaystyle \left(q^{-1} \delta_u^i +  (q-q^{-1}) \frac {\oq_-^{\frac 2n \cdot \overline{u-i}}}{\oq_-^2-1} \right)$, this is equivalent to:
\begin{equation}
\label{eqn:must prove}
\Big[ P_{[i;i)}^{(1)}, P_{- l\bde, 1}^{(0)} \Big] = \sum_{u=1}^n \barc P^{(1)}_{-[u-nl;u)} \left(\oq_-^l - \oq_-^{-l} \right)\left(q^{-1} \delta_u^i +  (q-q^{-1}) \frac {\oq_-^{\frac 2n \cdot \overline{u-i}}}{\oq_-^2-1} \right)
\end{equation}
To this end, let us write:
$$
P_{- l\bde, 1}^{(0)} = R(z_{11},...,z_{1l},...,z_{n1},...,z_{nl})
$$
where the rational function $R$ has total homogeneous degree $0$. By assumption, the hinge of any tensor in the coproduct $\Delta(R)$ lies strictly below the vector $(-nl,0)$. Therefore, the hinge of any tensor in the coproduct of the element:
\begin{equation}
\label{eqn:bern}
\Big[ P_{[i;i)}^{(1)}, P_{- l\bde, 1}^{(0)} \Big] \stackrel{\eqref{eqn:double 2}}= \barc \tilde{R} 
\end{equation}
where
\begin{equation}
\label{eqn:born}
\tilde{R} = \oq^{\frac {2i-1}n} \left(\sum_{s=1}^{l} z_{is} - \sum_{s=1}^{l} z_{i-1,s} \right) \cdot R 
\end{equation}
lies strictly below the vector $(-nl, 1)$. This implies that the expression \eqref{eqn:bern} is some linear combination of the shuffle elements $P_{-[1-nl;1)}^{(1)},...,P_{-[n-nl;n)}^{(1)}$, i.e.:
$$
\barc^{-1} \cdot \text{LHS of \eqref{eqn:must prove}} = \sum_{u=1}^n c_u \cdot P_{-[u-nl;u)}^{(1)}
$$
To determine the constants $c_u$, we apply the linear maps $\alpha_{-[u;u+nl)}$ of \eqref{eqn:beta} to the expression above, for any $u \in \{1,...,n\}$. We have:
$$
\left[\tilde{R}|_{z_u \mapsto q^{2u},...,z_{u+nl-1} \mapsto q^{2(u+nl-1)}} \right] \cdot \frac {(1-q^{-2})^{nl} \oq_-^{\frac {2u-1}n}}{\prod_{u \leq a < b < u+nl} \zeta\left( q^{2(a-b)} \right)} \stackrel{\eqref{eqn:normalize imaginary}}= - c_u
$$
In terms of the variables $z_{11},...,z_{nl}$, the specialization above corresponds to:
\begin{equation}
\label{eqn:eval}
z_{us} \mapsto q^{2u} \oq_-^{2-2s},...,z_{ns} \mapsto q^{2n} \oq_-^{2-2s}, z_{1s} \mapsto q^{2} \oq_-^{-2s},...,z_{u-1,s} \mapsto q^{2u-2} \oq_-^{-2s}
\end{equation}
for all $s \in \{1,...,l\}$. Let $\delta_{a<b}$ denote the number 1 if $a<b$ and 0 otherwise. Given the definition of $\tilde{R}$ in \eqref{eqn:born}, we obtain:
$$
\oq^{\frac {2i-1}n} \left(q^{2i} \oq_-^{\frac {2(i-\bari)}n -2\delta_{\bari<u}} - q^{2(i-1)} \oq_-^{\frac {2(i-1-\overline{i-1})}n -2\delta_{\overline{i-1}<u}} \right) \frac {1-\oq_-^{-2l}}{1-\oq_-^{-2}} \cdot R \Big|_{\text{evaluation }\eqref{eqn:eval}} \cdot 
$$
\begin{equation}
\label{eqn:1}
\cdot \frac {(1-q^{-2})^{nl} \oq_-^{\frac {2u-1}n}}{\prod_{u \leq a < b < u+nl} \zeta\left( q^{2(a-b)} \right)} = - c_u
\end{equation}
However, the fact that $\alpha_{-[u;u+nl)}(P_{-l\bde,1}^{(0)}) = \alpha_{-[u;u+nl)}(R) = - 1$ implies that:
\begin{equation}
\label{eqn:2}
R \Big|_{\text{evaluation }\eqref{eqn:eval}} \cdot \frac {(1-q^{-2})^{nl} \oq_-^{-l}}{\prod_{u \leq a < b < u+nl} \zeta\left( q^{2(a-b)} \right)} = - 1
\end{equation} 
Dividing \eqref{eqn:1} by \eqref{eqn:2}, and using the identity $\oq = q^{-n} \oq_-^{-1}$, yields:
$$
c_u =  \left(q \oq_-^{\frac {2(u-\bari)}n -2\delta_{\bari<u}} - q^{-1} \oq_-^{\frac {2(u-1-\overline{i-1})}n -2\delta_{\overline{i-1}<u}} \right) \frac {\oq_-^l-\oq_-^{-l}}{1-\oq_-^{-2}} 
$$
It is elementary to see that the RHS of the expression above equals:
$$
\left(\oq_-^l - \oq_-^{-l} \right)\left(q^{-1} \delta_u^i +  (q-q^{-1}) \frac {\oq_-^{\frac 2n \cdot \overline{u-i}}}{\oq_-^2-1} \right)
$$
which establishes \eqref{eqn:must prove}.  

\end{proof}

\begin{proposition}
\label{prop:plus minus 2}
	
Formula \eqref{eqn:case 4} holds in $\CA \cong \UU$  with $p \leftrightarrow P$. \\
	
\end{proposition} 

\begin{proof} We will only prove the case when $\pm = +$, as the case $\pm = -$ is analogous. We divide the proof into several cases, depending on the relative sizes of $j-i$ and $j'-i'$. \\
	
\noindent Let $a = P_{[i;j)}^{(k)}$ and $b = P_{-[i';j')}^{(-k')}$, and let us assume $j - i > j' - i' > 0$:
	
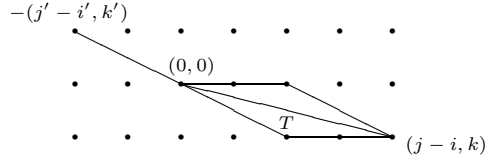
\begin{picture}(100,65)(-100,-10)

\put(0,0){\circle*{2}}\put(20,0){\circle*{2}}\put(40,0){\circle*{2}}\put(60,0){\circle*{2}}\put(80,0){\circle*{2}}\put(100,0){\circle*{2}}\put(120,0){\circle*{2}}\put(0,20){\circle*{2}}\put(20,20){\circle*{2}}\put(40,20){\circle*{2}}\put(60,20){\circle*{2}}\put(80,20){\circle*{2}}\put(100,20){\circle*{2}}\put(120,20){\circle*{2}}\put(0,40){\circle*{2}}\put(20,40){\circle*{2}}\put(40,40){\circle*{2}}\put(60,40){\circle*{2}}\put(80,40){\circle*{2}}\put(100,40){\circle*{2}}\put(120,40){\circle*{2}}

\put(40,20){\line(1,0){40}}
\put(40,20){\line(4,-1){80}}
\put(80,0){\line(1,0){40}}
\put(0,40){\line(2,-1){80}}
\put(80,20){\line(2,-1){40}}

\put(77,3){\scriptsize{$T$}}
\put(35,25){\scriptsize{$(0,0)$}}
\put(125,-5){\scriptsize{$(j-i,k)$}}
\put(-25,45){\scriptsize{$-(j'-i',k')$}}

\end{picture}

\noindent Note that \eqref{eqn:new coproduct 1} implies:
$$
\Delta(a) = \frac {\psi_j}{\psi_i} \barc^k \otimes a + a \otimes 1  + \sum_{i \leq s < t \leq j} E_{[t,j)}^{\mu} \frac{\psi_t}{\psi_s} \barc^\bullet \bE_{[i;s)}^{\mu} \otimes P_{[s;t)}^{(\bullet)} + ...
$$
where $\bullet = k-\mu(j-i+s-t)$ and the ellipsis stands for tensors with hinge strictly below the triangle $T$. Moreover, as a consequence of \eqref{eqn:radu2} we have:
$$
\Delta(b) = 1 \otimes b + b \otimes \frac {\psi_{i'}}{\psi_{j'}} \barc^{-k'} + ...
$$		
where the ellipsis stands for tensors with hinge strictly below the vector $-(j'-i',k')$. By the aforementioned discussion on the possible locations of the hinges, we conclude that the only non-zero terms in relation \eqref{eqn:drinfeld} applied to our choice of $a$ and $b$ are:
$$
ab + \sum_{i \leq s < t \leq j} E_{[t,j)}^{\mu} \frac{\psi_t}{\psi_s} \barc^\bullet \bE_{[i;s)}^{\mu} \cdot \left \langle P_{[s;t)}^{(\bullet)}, b \right \rangle = ba
$$
As a consequence of \eqref{eqn:pair simple}, we obtain precisely relation \eqref{eqn:case 4}. \\

\noindent Now let us assume $j' - i' > j - i > 0$, as in the picture below:

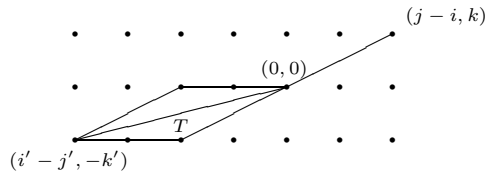
\begin{picture}(100,68)(-100,-15)

\put(0,0){\circle*{2}}\put(20,0){\circle*{2}}\put(40,0){\circle*{2}}\put(60,0){\circle*{2}}\put(80,0){\circle*{2}}\put(100,0){\circle*{2}}\put(120,0){\circle*{2}}\put(0,20){\circle*{2}}\put(20,20){\circle*{2}}\put(40,20){\circle*{2}}\put(60,20){\circle*{2}}\put(80,20){\circle*{2}}\put(100,20){\circle*{2}}\put(120,20){\circle*{2}}\put(0,40){\circle*{2}}\put(20,40){\circle*{2}}\put(40,40){\circle*{2}}\put(60,40){\circle*{2}}\put(80,40){\circle*{2}}\put(100,40){\circle*{2}}\put(120,40){\circle*{2}}

\put(80,20){\line(-4,-1){80}}
\put(40,0){\line(2,1){80}}
\put(0,0){\line(2,1){40}}
\put(0,0){\line(1,0){40}}
\put(40,20){\line(1,0){40}}

\put(37,3){\scriptsize{$T$}}
\put(70,25){\scriptsize{$(0,0)$}}
\put(125,45){\scriptsize{$(j-i,k)$}}
\put(-25,-10){\scriptsize{$(i'-j',-k')$}}

\end{picture}

\noindent Formula \eqref{eqn:radu1} implies:
$$
\Delta(a) = \frac {\psi_j}{\psi_i} \barc^k \otimes a + a \otimes 1  + ...
$$
where the ellipsis stands for tensors with hinge strictly below the vector $(j-i,k)$. Meanwhile, formula \eqref{eqn:new coproduct 2} gives us:
$$
\Delta(b) = 1 \otimes b + b \otimes \frac {\psi_{i'}}{\psi_{j'}} \barc^{-k'} + \sum_{i' \leq s < t \leq j'} \bE_{-[t,j')}^{\mu} E_{-[i';s)}^{\mu} \otimes \frac{\psi_t}{\psi_{j'}}  P_{-[s;t)}^{(\bullet)} \frac{\psi_{i'}}{\psi_s} \barc^{-k'-\bullet}  + ...
$$		
where $\bullet = -k'+\mu(j-i+s-t)$ and the ellipsis stands for tensors with hinge strictly below the triangle $T$. As a consequence of the above discussion on the locations of the hinges, the only non-trivial terms in relation \eqref{eqn:drinfeld} are:
$$
ab + \sum_{i' \leq s < t \leq j'} \frac {\psi_j}{\psi_i} \barc^k \bE_{-[t,j')}^{\mu} E_{-[i';s)}^{\mu} \cdot \left \langle a,\frac{\psi_t}{\psi_{j'}}  P_{-[s;t)}^{(\bullet)} \frac{\psi_{i'}}{\psi_s} \barc^{-k'-\bullet} \right \rangle = ba
$$
As a consequence of \eqref{eqn:pair simple}, we obtain precisely relation \eqref{eqn:case 4} (after collecting various powers of $q$ by commuting the various $\psi$ factors). \\

\noindent If $j-i = j'-i' > 0$, then our assumptions force us to have $j = i+1$, $j' = i'+1$ and $k > k'$. In this case, the required relation reads:
$$
\left[ P_{[i;i+1)}^{(k)}, P_{-[i';i'+1)}^{(-k')} \right] = \frac {\delta_{i'}^i}{q-q^{-1}} \sum_{a+b = k-k'} E_{[i+1;i+1)}^{(a)} \frac {\psi_{i+1}}{\psi_i} \barc^{k'} \bE_{[i;i)}^{(b)}
$$
Using the substitutions \eqref{eqn:p plus}--\eqref{eqn:p minus} and \eqref{eqn:cartan 1}--\eqref{eqn:cartan 2}, the relation above is equivalent to the particular case of \eqref{eqn:moreover} when $k > k'$. \\

\noindent The only other case covered by our assumptions is $j'-i' > 0 = j-i$, in which case we must have $k=1$ and $j'-i'|k'-1$. To keep the notation simple we will assume $k' = 1$, as the general case can be obtained by simply multiplying all the rational functions below by the monomial $\prod_{a=i'}^{j'-1} (z_a \oq^{\frac {2a}n})^{\frac {k'-1}{j'-i'}}$. The required relation reads:
\begin{equation}
\label{eqn:oj}
\left[P_{[i;i)}^{(1)}, P_{- [i';j')}^{(-1)} \right] = \frac 1{q-q^{-1}} \sum^{r \equiv i}_{i' \leq r \leq j'} \bE_{- [r,j')}^{0} E_{- [i';r)}^{0} \barc
\end{equation}
Let us write:
$$
P_{- [i';j')}^{(-1)} = R(z_{i'},...,z_{j'-1})
$$
where the rational function $R$ has total homogeneous degree $-1$. By assumption, the hinge of any tensor in the coproduct $\Delta(R)$ lies strictly below the vector $-(j'-i',1)$. Therefore, the hinge of any tensor in the coproduct of the element:
$$
\left[P_{[i;i)}^{(1)}, P_{- [i';j')}^{(-1)} \right] \stackrel{\eqref{eqn:double 2}}=  \tilde{R} \barc, \quad \text{where} \quad \tilde{R} = \left(\sum_{i' \leq a < j'}^{a \equiv i} z_a \oq^{\frac {2a-1}n} - \sum_{i' \leq a < j'}^{a \equiv i-1} z_a \oq^{\frac {2a+1}n} \right) \cdot R
$$
lies on or below the vector $-(j'-i',0)$. This implies that $\tilde{R} \in \CB_0$. Therefore, in order to prove the required identity:
\begin{equation}
\label{eqn:dc}
\tilde{R} = \frac 1{q-q^{-1}} \sum^{r \equiv i}_{i' \leq r \leq j'} \bE_{- [r,j')}^{0} E_{- [i';r)}^{0}
\end{equation}
we need to show that the two sides of \eqref{eqn:dc} have the same intermediate terms for the coproduct $\Delta_0$, and also the same value under the linear maps $\alpha_{-[u;v)}$ for any $[u;v) = [i';j')$. We will prove the statement about the coproduct by induction on $j'-i'$ (the base case is vacuous). By \eqref{eqn:new coproduct 2}, we have:
$$
\Delta\left(R \right) = 1 \otimes R + R \otimes \frac {\psi_{i'}}{\psi_{j'}} \barc^{-1}+ \sum_{i' \leq s < t \leq j'} \bE_{-[t,j')}^{0} E_{-[i';s)}^{0} \otimes \frac {\psi_t}{\psi_{j'}} P_{-[s;t)}^{(-1)} \frac {\psi_{i'}}{\psi_s}  + ...
$$
where the ellipsis denotes tensors with hinge which are more than one unit below the horizontal line. By commuting the expression above with the primitive element $P_{[i;i)}^{(1)}$, we obtain the following formula:
$$
\Delta(\tilde{R}) = 1 \otimes \tilde{R} + \tilde{R} \otimes \frac {\psi_{i'}}{\psi_{j'}} + \sum_{i' \leq s < t \leq j'} \bE_{-[t,j')}^{0} E_{-[i';s)}^{0} \otimes \frac {\psi_t}{\psi_{j'}} \left[P_{[i;i)}^{(1)}, P_{-[s;t)}^{(-1)}\right]\barc^{-1}  \frac {\psi_{i'}}{\psi_s} + ...
$$
By the induction hypothesis of \eqref{eqn:oj}, we may write the expression above as:
$$
\Delta(\tilde{R}) = 1 \otimes \tilde{R} + \tilde{R} \otimes \frac {\psi_{i'}}{\psi_{j'}} + ... + 
$$
$$
+ \frac 1{q-q^{-1}}\sum^{r \equiv i}_{i' \leq s \leq r \leq t \leq j'} \bE_{-[t,j')}^{0} E_{-[i';s)}^{0} \otimes \frac {\psi_t}{\psi_{j'}} \bE_{- [r,t)}^{0} E_{- [s;r)}^{0} \frac {\psi_{i'}}{\psi_s} + ...
$$ 
By \eqref{eqn:cop2} and \eqref{eqn:cop2 anti}, the last line of the expression above matches $\Delta_0$ applied to the right-hand side of \eqref{eqn:dc}. This proves that the two sides of relation \eqref{eqn:dc} have the same intermediate terms for the coproduct $\Delta_0$. Therefore, all that is left to show is that the two sides of the aforementioned relation have the same value under the linear maps $\alpha_{-[u;v)}$, for all $[u;v) = [i';j')$. To this end, formula \eqref{eqn:beta} reads:
$$
\alpha_{-[u;v)}(\text{LHS}) = \alpha_{-[u;v)}(\tilde{R}) = \frac {\tilde{R}(...,q^{2a},...) (1-q^{-2})^{v-u}}{\oq_-^{\frac {v-u}n} \prod_{i \leq a < b < j} \zeta(q^{2a-2b})} =
$$
$$
= \frac {R(...,q^{2a},...) \left(\sum_{i' \leq a < j'}^{a \equiv i} q^{2a} \oq^{\frac {2a-1}n} - \sum_{i' \leq a < j'}^{a \equiv i-1} q^{2a} \oq^{\frac {2a+1}n} \right) (1-q^{-2})^{v-u}}{\oq_-^{\frac {v-u}n} \prod_{i \leq a < b < j} \zeta(q^{2a-2b})}
$$
Since formula \eqref{eqn:normalize simple} gives us:
$$
\alpha_{-[u;v)}(R) = \frac {R(...,q^{2a},...)(1-q^{-2})^{v-u}}{\oq_-^{ \frac {2v-1}n} \prod_{u \leq a < b < v} \zeta(q^{2a-2b})} = - \delta_{(u,v)}^{(i',j')}
$$
we conclude that:
\begin{equation}
\label{eqn:lhs alpha}
\alpha_{-[u;v)}(\text{LHS}) = \delta_{(u,v)}^{(i',j')} \oq_-^{\frac {j'+i'-1}n}\left(\sum_{i' \leq a < j'}^{a \equiv i-1} q^{-1} \oq_-^{-\frac {2a+1}n} - \sum_{i' \leq a < j'}^{a \equiv i} q \oq_-^{-\frac {2a-1}n}\right)
\end{equation}
As for the right-hand side of \eqref{eqn:dc}, we have:
$$
\alpha_{-[u;v)}(\text{RHS}) \stackrel{\eqref{eqn:pseudo}}= \frac 1{q-q^{-1}} \sum^{r \equiv i}_{i' \leq r \leq j'} \alpha_{-[r;v)}(\bE_{- [r,j')}^{0}) \alpha_{-[u;r)}(E_{- [i';r)}^{0}) \stackrel{\eqref{eqn:main pair 1}, \eqref{eqn:main pair 2}}=
$$
\begin{equation}
\label{eqn:rhs alpha}
= \delta_{(u,v)}^{(i',j')} \left[\delta_{j'}^i q^{-1} \oq_-^{\frac {i'-j'}n} - \delta_{i'}^i q \oq_-^{\frac {j'-i'}n} + (q^{-1}-q)\sum_{i' < r < j'}^{r \equiv i} \oq_-^{\frac {j'+i'-2r}n} \right]
\end{equation}
It is easy to see that the right-hand sides of \eqref{eqn:lhs alpha} and \eqref{eqn:rhs alpha} are equal to each other, thus concluding the proof of \eqref{eqn:oj}.

\end{proof}

\subsection{}
\label{sub:main theorem}

We are now ready to prove our main Theorem. \\

\begin{proof}\emph{of Theorem \ref{thm:main}:} Propositions \ref{prop:plus plus 1}, \ref{prop:plus plus 2}, \ref{prop:plus minus 1} and \ref{prop:plus minus 2} imply that there exists an algebra homomorphism: 
$$
\CC \stackrel{\Upsilon}\longrightarrow \A \quad \text{which sends} \quad \CE_\mu \stackrel{\sim}\longrightarrow \CB_\mu
$$	
All that remains to prove is that $\Upsilon$ is an isomorphism of vector spaces. Since we have the triangular decomposition (\cite{Tor}):
\begin{equation}
\label{eqn:have}
\A \cong \A^+ \otimes \B_\infty \otimes \A^- \quad \text{where} \quad \A^\pm = \bigotimes^\rightarrow_{\mu \in \BQ} \B^\pm_\mu
\end{equation}
then it remains to show that:
\begin{equation}
\label{eqn:want}
\CC \cong \CC^+ \otimes \CE_\infty \otimes \CC^- \quad \text{and} \quad \CC^\pm = \bigotimes^\rightarrow_{\mu \in \BQ} \CE^\pm_\mu
\end{equation}
where $\CC^\pm \subset \CC$ are the subalgebras generated by:
$$
\Big \{p_{\pm [i;j)}^{(\pm k)}, p_{\pm l\bde, \hi}^{(\pm k')} \Big \}^{k,k' \in \BZ, \text{ any }\hi}_{i < j, l \in \BN}
$$
Both of the statements in \eqref{eqn:want} are immediate consequences of the following principle (following \cite{BS, S}). We call a ``generator" of $\CE_\mu^\pm$ any one of the symbols:
$$
p_{\pm [i;j)}^{(\pm k)} \quad \text{or} \quad p_{\pm l\bde, \hi}^{(\pm k')}
$$
with $\frac k{j-i} = \mu$ or $\frac {k'}{nl} = \mu$. \\

\begin{claim}
\label{claim:want}

For any generators $x \in \CE^\pm_\lambda$ and $y \in \CE^\pm_\mu$ with $\lambda > \mu$, we have:
\begin{equation}
\label{eqn:want 1}
xy \in \bigotimes^\rightarrow_{\nu \in [\mu, \lambda] \cap \BQ} \CE^\pm_\nu
\end{equation}
while for any generators $x \in \CE^-_\lambda$ and $y \in \CE^+_\mu$, we have:
\begin{equation}
\label{eqn:want 2}
xy \in \begin{cases} \bigotimes^\rightarrow_{\nu \in [\mu, \infty)} \CE^+_\nu \otimes \CE_\infty^+ \otimes \bigotimes^\rightarrow_{\nu \in (-\infty,\lambda]} \CE^-_\nu &\text{if } \lambda < \mu \\ \\ \bigotimes^\rightarrow_{\nu \in (- \infty,\mu]} \CE^+_\nu \otimes \CE_\infty^- \otimes \bigotimes^\rightarrow_{\nu \in [\lambda,\infty)} \CE^-_\nu &\text{if } \lambda > \mu\end{cases}
\end{equation}
(when $\lambda = \mu$, such formulas are implicit in the defining relations of $\CE_\mu$). \\

\end{claim}

\noindent Once one has Claim \ref{claim:want}, it is straightforward to show (akin to the proof of Corollary 5.1 of \cite{S}) that any product of the generators of $\CC$ can be written as a sum of products of elements of $\CE^\pm_\mu$ in counterclockwise order of the slope $\pm \mu$. Moreover, there can be no non-trivial linear relations among such ordered products, because passing such a relation through the homomorphism $\Upsilon$ would violate \eqref{eqn:have}. \\

\begin{proof} \emph{of Claim \ref{claim:want}.} The proof follows the idea of \cite{BS, S}. To keep the notation simple, we will only prove \eqref{eqn:want 1}, as \eqref{eqn:want 2} is an analogous exercise that we leave to the interested reader. Moreover, we will only prove the $\pm = +$ case of \eqref{eqn:want 1}, as the $\pm = -$ case is proved similarly. We we use induction on the number:
$$
\# = kd' - k'd \in \BN
$$
where $k = \vdeg x$, $d = |\hdeg x|$ and $k' = \vdeg y$, $d' = |\hdeg y|$. If $\# = 1$, then \eqref{eqn:want 1} is a particular case of either relation \eqref{eqn:case 1} or \eqref{eqn:case 3}, which establishes the base case of the induction. For the induction step, pick any $N > 1$. Let us assume that \eqref{eqn:want 1} holds for all pairs of generators $x,y$ with $\# < N$, and let us prove it for those pairs such that $\# = N$. We note that the lattice triangle $T$ obtained by placing the vectors $(d, k)$ and $(d',k') \in \BN \times \BZ$ in succession has area precisely $N/2$. \\
	
\noindent \underline{Case 1:} Suppose there are no lattice points contained strictly inside $T$. If there are no lattice points on at least two of the edges of $T$ (except for the vertices), then \eqref{eqn:want 1} once again reduces to either \eqref{eqn:case 1} or \eqref{eqn:case 3}. If there are lattice points on two of the edges of $T$, then there are lattice points on all three edges. Then we run the argument in Case 2 below, with the triangle $T'$ therein defined to have the point $(d+d',k+k')/g$ as a vertex, where $g = \gcd(d+d',k+k')$. \\
	
\noindent \underline{Case 2}: Suppose there exists a lattice point inside the triangle $T$. Then \loccit considers a lattice triangle $T' \subset T$ as in the picture below:

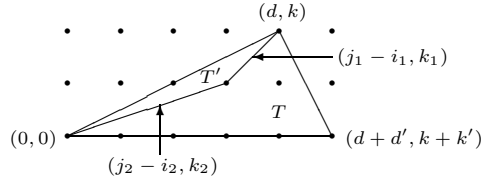
\begin{picture}(100,100)(-110,-10)

\put(0,20){\circle*{2}}\put(20,20){\circle*{2}}\put(40,20){\circle*{2}}\put(60,20){\circle*{2}}\put(80,20){\circle*{2}}\put(100,20){\circle*{2}}\put(0,40){\circle*{2}}\put(20,40){\circle*{2}}\put(40,40){\circle*{2}}\put(60,40){\circle*{2}}\put(80,40){\circle*{2}}\put(100,40){\circle*{2}}\put(0,60){\circle*{2}}\put(20,60){\circle*{2}}\put(40,60){\circle*{2}}\put(60,60){\circle*{2}}\put(80,60){\circle*{2}}\put(100,60){\circle*{2}}

\put(0,20){\line(1,0){100}}
\put(80,60){\line(1,-2){20}}
\put(0,20){\line(2,1){80}}
\put(60,40){\line(-3,-1){60}}
\put(60,40){\line(1,1){20}}

\put(-22,17){\scriptsize{$(0,0)$}}
\put(105,17){\scriptsize{$(d+d',k+k')$}}
\put(77,27){\scriptsize{$T$}}
\put(50,40){\scriptsize{$T'$}}

\put(72,65){\scriptsize{$(d,k)$}}
\put(102,48){\scriptsize{$(j_1-i_1,k_1)$}}
\put(15,7){\scriptsize{$(j_2-i_2,k_2)$}}

\put(100,50){\vector(-1,0){30}}
\put(35,13){\vector(0,1){19}}

\end{picture}

\noindent We may choose $T'$ to have minimal area, which means that we may apply \eqref{eqn:case 3}:
\begin{equation}
\label{eqn:mazare}
p_{[i_1;j_1)}^{(k_1)} \tp_{[i_2;j_2)}^{(k_2)} q^{\delta_{j_2}^{i_1} - \delta_{i_2}^{i_1}}  - \tp_{[i_2;j_2)}^{(k_2)} p_{[i_1;j_1)}^{(k_1)} q^{\delta_{j_2}^{j_1} - \delta_{i_2}^{j_1}} = \sum^{(s,t) \equiv (i_2,j_2)}_{s \leq j_1 \text{ and } i_1 \leq t} \frac {\be^{\lambda}_{[s,j_1)} e_{[i_1;t)}^{\lambda}}{q^{-1}-q}  \in \CE_{\lambda}^+ 
\end{equation}	
where the vectors $(j_1-i_1,k_1)$ and $(j_2-i_2,k_2)$ are as depicted in the figure above, and $[i_1;j_1) + [i_2;j_2) = \hdeg x$. As shown in \cite{S}, it suffices to prove that $i_1,j_1,i_2,j_2$ can be chosen in such a way that the RHS of \eqref{eqn:mazare} is equal to a \underline{non-zero} multiple of $x$ + sums of products of more than one primitive generator in $\CB_\lambda$. Indeed, once this is done, the induction hypothesis of \eqref{eqn:want 1} will imply that:
\begin{equation}
\label{eqn:sor}
p_{[i_1;j_1)}^{(k_1)} \tp_{[i_2;j_2)}^{(k_2)}y, \ \tp_{[i_2;j_2)}^{(k_2)} p_{[i_1;j_1)}^{(k_1)}y, \ zy  \in \bigotimes^\rightarrow_{\nu \in [\mu, \lambda] \cap \BQ} \CE^+_\nu
\end{equation}
for any $z \in \CE_\lambda^+$ that is a product of more than one primitive generator. Taking an appropriate linear combination of relations \eqref{eqn:sor} proves \eqref{eqn:want 1} for our given $x$, and thus establishes the induction step. \\
	
\noindent As we only need to prove the non-zero-ness underlined in the previous paragraph, we might as well apply the homomorphism $\Upsilon$ to \eqref{eqn:mazare} and prove this non-zero-ness in the algebra $\A$. The advantage to doing so is that we may apply the bialgebra pairing: because $\Upsilon(x)$ is a primitive generator of $\A$, it suffices to show that $i_1,j_1,i_2,j_2$ can be chosen so that the right-hand side of \eqref{eqn:mazare} has non-zero pairing with any primitive generator of $\CB^-_\lambda$ of degree opposite to that of $x$. We will prove this by considering the two kinds of primitive generators: \\
	
\begin{itemize}[leftmargin=*]
		
\item The simple ones, i.e. $\Upsilon(x) = P_{[i;j)}^{(k)}$ with $\gcd(k,j-i) = 1$ and $j \not \equiv i$ mod $n$. Then:
$$
\Upsilon(\text{RHS of \eqref{eqn:mazare}}) = \frac {E_{[i;j)}^{(k)}}{q^{-1}-q} 
$$
for any choice of $i = i_1 < j_1 = i_2 < j_2 = j$. Then:
$$
\left \langle \Upsilon(\text{RHS of \eqref{eqn:mazare}}), P_{-[i;j)}^{(-k)} \right \rangle = \frac {- q^{-1}\oq^{-\frac 1n}}{q^{-1}-q} 
$$
as a consequence of \eqref{eqn:main pair 3}. \\ 
		
\item The imaginary ones, i.e. $\Upsilon(x) = P_{l\bde, \hi}^{(k)}$. Then: 
$$
\Upsilon(\text{RHS of \eqref{eqn:mazare}}) = \frac {E_{[i_1;i_1+nl)}^{(k)} + \bE_{[j_1-nl;j_1)}^{(k)}}{q^{-1} - q} + ... 
$$	
for any choice of $i_1 < j_1$ and $i_2 < j_2$ such that $[i_1;j_1) + [i_2;j_2) = l\bde$, where the ellipsis denotes sums of products of two $E$'s and $\bE$'s. Since $P_{-l\bde,\hi}^{(-k)}$ is primitive, it pairs trivially with any such product, and therefore we have:
$$
\left \langle \Upsilon(\text{RHS of \eqref{eqn:mazare}}), P_{-l\bde,\hi}^{(-k)} \right \rangle = \frac {q^d\oq^{\frac dn} - q^{-d}\oq^{-\frac dn}}{q^{-1}-q} 
$$
where $d = \gcd(k,nl)$, as a consequence of \eqref{eqn:main pair 4} and \eqref{eqn:main pair 6}.  
		
\end{itemize} 

\end{proof}	 \end{proof}

\subsection{} 
\label{sub:rotate}

In the following equation, the first isomorphism was proved as part of the proof of Theorem \ref{thm:main}, and the second isomorphism was proved in \cite{Tor}:
$$
\CC^\pm \cong \A^\pm \cong \UUpm
$$
Therefore, the algebras $\CC^\pm$ are generated by the horizontal degree 1 elements:
\begin{equation}
\label{eqn:degree 1}
\left\{ p_{\pm [i;i+1)}^{(\pm k)} \right \}_{i \in \BZ/n\BZ}^{k \in \BZ}
\end{equation}
(this can also be shown directly from relations \eqref{eqn:case 1} and \eqref{eqn:case 3}, by an argument similar to the one in the proof of Theorem \ref{thm:main}). In the following Proposition, we will show that all the relations between generators of $\CC^+$ and $\CC^-$ can be deduced from the relations among the generators \eqref{eqn:degree 1} and the generators of $\CE_\infty$. \\

\begin{proposition}
\label{prop:equiv}

If $\CC^\pm = \Big \langle \CE_\mu^\pm \Big \rangle_{\mu \in \BQ} \Big/ 
\Big(  \text{relations \eqref{eqn:case 1}, \eqref{eqn:case 3}} \Big)$, then:
\begin{equation}
\label{eqn:qwerty}
\CC \cong \CC^+ \otimes \CE_\infty \otimes \CC^- \Big / \text{relations \eqref{eqn:special rel 2}--\eqref{eqn:special rel 4}}
\end{equation}
where the following are special cases of \eqref{eqn:case 1}--\eqref{eqn:case 4}:
\begin{align}
&\left[p_{\pm [i;i+1)}^{(k)}, p_{0\bde,r}^{(\pm d)} \right] = \pm p_{\pm [i;i+1)}^{(k\pm d)} \left(\delta_i^r \oq_\pm^{-\frac dn} - \delta_{i+1}^r \oq_\pm^{\frac dn} \right) \label{eqn:special rel 2} \\
&\left[p_{\pm [i;i+1)}^{(k)}, p_{0\bde,r}^{(\mp d)} \right] = \pm p_{\pm [i;i+1)}^{(k\mp d)} \left(\delta_i^r \oq_\mp^{-\frac dn} - \delta_{i+1}^r \oq_\mp^{\frac dn} \right) \barc^{\mp d}  \label{eqn:special rel 3}
\end{align}
and:
\begin{equation}
\label{eqn:special rel 4}
\left[p_{[i;i+1)}^{(k)}, p_{-[i';i'+1)}^{(- k')} \right] = \frac {\delta_{i'}^i}{q-q^{-1}} \cdot 
\end{equation}
$$
\left( \sum^{a,b \geq 0}_{a+b = k - k'} e_{0\bde,i+1}^{(a)} \frac {\psi_{i+1}}{\psi_i} \barc^{k'} \be_{0\bde,i}^{(b)} - \sum^{a,b \leq 0}_{a+b = k - k'} e_{0\bde,i+1}^{(a)} \frac {\psi_i}{\psi_{i+1}} \barc^{- k} \be_{0\bde,i}^{(b)} \right)
$$
for all $k,k' \in \BZ$, $d,d' \in \BN$ and $i,i' \in \BZ/n\BZ$. \\

\end{proposition}

\begin{proof} Let $\CC'$ be the algebra in the right-hand side of \eqref{eqn:qwerty}. We want to show that the natural surjective map:
\begin{equation}
\label{eqn:twohead}
\CC' \twoheadrightarrow \CC
\end{equation}
is an isomorphism. To this end, we will show that relation \eqref{eqn:case 4} holds in $\CC'$ (the analogous statement for relation \eqref{eqn:case 2} will be left as an exercise to the interested reader). Clearly, the map \eqref{eqn:twohead} restricts to isomorphisms:
\begin{equation}
\label{eqn:twohead pm}
{\CC'}^\pm \cong \CC^\pm 
\end{equation}
on the subalgebras generated by $p_{\pm [i;j)}^{(\pm k)}, p_{\pm l\bde, r}^{(\pm k')}$ for all $i<j$, $l>0$ and $k,k' \in \BZ$, because the two algebras in \eqref{eqn:twohead pm} have the same generators and relations. Since the half subalgebras $\CC^\pm$ are generated by degree one elements, then we can write:
\begin{align*}
&p_{\pm [i;j)}^{(\pm k)} = \sum \text{coefficient} \cdot p_{\pm [a_1;a_1+1)}^{(\pm b_1)} \dots p_{\pm [a_{j-i};a_{j-i}+1)}^{(\pm b_{j-i})} \\
&p_{\mp [i';j')}^{(\mp k')} = \sum \text{coefficient} \cdot p_{\mp [a'_1;a'_1+1)}^{(\mp b'_1)} \dots p_{\mp [a'_{j'-i'};a'_{j'-i'}+1)}^{(\mp b'_{j'-i'})}
\end{align*}
in ${\CC'}^\pm \cong \CC^\pm$, for various numbers $a_d, a'_d, b_d, b'_d$. Then the Leibniz rule implies that:
$$
\text{LHS of \eqref{eqn:case 4}} = \sum \text{coefficient} \cdot \sum_{d = 1}^{j-i} \sum_{d' = 1}^{j'-i'} \left( \dots \left[p_{\pm [a_d;a_d+1)}^{(\pm b_d)},  p_{\mp [a'_{d'};a'_{d'}+1)}^{(\mp b'_{d'})} \right] \dots \right)
$$
in $\CC'$. One can use \eqref{eqn:special rel 2}--\eqref{eqn:special rel 4} to rewrite the expression in the right-hand side with all products ordered as in \eqref{eqn:want}, i.e.:
\begin{equation}
\label{eqn:need}
\text{LHS of \eqref{eqn:case 4}} = \sum \text{coefficient} \cdot \dots  p_{[a_x;a_x+1)}^{(b_x)} \dots  p_{0 \bde, r_y}^{(\pm d_y)} \dots p_{-[a'_z;a'_z+1)}^{(-b'_z)} \dots
\end{equation}
as an identity in $\CC'$. However, the right-hand side of \eqref{eqn:need} is equal to the right-hand side of \eqref{eqn:case 4} in the algebra $\CC$, because the computation above could have been done in $\CC$ just as well as in $\CC'$. Because of the triangular decomposition \eqref{eqn:want}, we conclude that the right-hand side of \eqref{eqn:need} is equal to the right-hand side of \eqref{eqn:case 4} as elements of the vector space $\CC^+ \otimes \CE_\infty \otimes \CC^-$. Therefore, the right-hand of \eqref{eqn:need} is equal to the right-hand side of \eqref{eqn:case 4} in $\CC'$, precisely what we needed to show. 
\end{proof}

\section{An orthogonal presentation}
\label{sec:final}

\noindent In the present Section, we state and prove an alternate version of Theorem \ref{thm:intro main}, which will yield a different presentation of $\UU$, that will be used in \cite{Tale}. \\ 

\begin{itemize}[leftmargin=*]

\item In Subsection \ref{sub:orthogonal}, we consider a triangular decomposition $\CA \cong \CA^\uparrow \otimes \CB_0 \otimes \CA^\downarrow$ which is ``orthogonal" to the defining triangular decomposition $\CA = \CA^+ \otimes \CB_\infty \otimes \CA^-$ \\

\item In Subsection \ref{sub:notation f}, we define elements $F_{\pm [i;j)}^{(\pm k)}$, $\bF_{\pm [i;j)}^{(\pm k)}$ to replace $E_{\pm [i;j)}^{(\pm k)}$, $\bE_{\pm [i;j)}^{(\pm k)}$ \\

\item In Subsection \ref{sub:study f}, we give formulas for the coproduct of $F_{\pm [i;j)}^{(\pm k)}$ and $\bF_{\pm [i;j)}^{(\pm k)}$ \\

\item In Subsection \ref{sub:connect e and f}, we connect the elements $E_{\pm [i;j)}^{(\pm k)}$, $\bE_{\pm [i;j)}^{(\pm k)}$ with $F_{\pm [i;j)}^{(\pm k)}$, $\bF_{\pm [i;j)}^{(\pm k)}$ \\

\item In Subsection \ref{sub:enter o}, we define elements $o_{\pm [i;j)}^{(\pm k)}$, $o_{\pm l\bde, \hi}^{(\pm k')}$ to replace $p_{\pm [i;j)}^{(\pm k)}$, $p_{\pm l\bde, \hi}^{(\pm k')}$ \\

\item In Subsection \ref{sub:main new}, we state and prove Theorem \ref{thm:main new} \\

\item In Subsection \ref{sub:degree 1}, we consider an analogue of the situation of Subsection \ref{sub:rotate} \\

\item In the remainder of the present Section, we prove the formulas in Propositions \ref{prop:identity 1} and \ref{prop:identity 2}, which were invoked in Subsection \ref{sub:connect e and f} 

\end{itemize}

\subsection{}
\label{sub:orthogonal}

Looking at relations \eqref{eqn:case 1}--\eqref{eqn:case 4}, we observe that the commutation relations between the generators of the quantum toroidal algebra depend quite strongly on whether the degree $(\bd,k)$ of the generators in question lies in $\nn \times \BZ$ or in $-\nn \times \BZ$. This is because the triangular decomposition:
\begin{equation}
\label{eqn:triangular 1}
\CA = \CA^+ \otimes \CB_\infty \otimes \CA^-
\end{equation}
prefers the vertical direction, i.e. that of slope $\infty$. In the present Section, we will see that this is just a consequence of our choice of generators. More specifically, we will construct another decomposition, which prefers the horizontal direction:
\begin{equation}
\label{eqn:triangular 2}
\CA \cong \CA^\uparrow \otimes \CB_0 \otimes \CA^\downarrow
\end{equation}
and see that the commutation relations between the generators will now depend quite strongly on whether the degree $(\bd,k)$ of the generators in question lies in $\zz \times \BN$ or $\zz \times -\BN$. This presentation will be used in \cite{Tale} to compare the quantum toroidal algebra with a new type of shuffle algebra, that will be defined therein. \\

\begin{remark}
	
We expect infinitely many triangular decompositions as \eqref{eqn:triangular 1} and \eqref{eqn:triangular 2}, indexed by a rational number $\mu$, in which the role of the middle subalgebra is played by $\CB_\mu$. The question is to find good descriptions of the left/right subalgebras. \\

\end{remark}

\subsection{} 
\label{sub:notation f} 

Recall the elements \eqref{eqn:a}--\eqref{eqn:bb} of $\CA^+ \otimes \CB_\infty \otimes \CA^- = \CA$, and define:
\begin{align}
&F_{[i;j)}^{(k)} =  \psi_i A^{(k)}_{[i;j)} \frac 1{\psi_j} \barc^{-k} \cdot (-\oq_-^{\frac 2n})^{i-j} \label{eqn:two 1} \\
&\bF_{[i;j)}^{(k)} =  \psi_i \bA^{(k)}_{[i;j)} \frac 1{\psi_j} \barc^{-k}
 \label{eqn:two 2} \\
&F_{-[i;j)}^{(-k)} =  \frac 1{\psi_i} B^{(-k)}_{-[i;j)} \psi_j \barc^k \cdot (-\oq_+^{\frac 2n})^{i-j} \label{eqn:two 3} \\
&\bF_{-[i;j)}^{(-k)} =  \frac 1{\psi_i} \bB^{(-k)}_{-[i;j)} \psi_j \barc^k \label{eqn:two 4} 
\end{align}
for all $(i<j) \in \zzz$, $k \in \BN \sqcup 0$ (recall that $\oq_+ = \oq$ and $\oq_- = q^{-n} \oq^{-1}$), as well as:
\begin{align}
&F_{[i;j)}^{(-k)} = B^{(-k)}_{[i;j)} \cdot q^{j-i} \label{eqn:two 5} \\
&\bF_{[i;j)}^{(-k)} = \bB^{(-k)}_{[i;j)} \cdot q^{j-i} (-\oq_+^{\frac 2n})^{j-i}  \label{eqn:two 6} \\
&F_{-[i;j)}^{(k)} = A^{(k)}_{-[i;j)} \cdot q^{j-i}  \label{eqn:two 7} \\
&\bF_{-[i;j)}^{(k)} = \bA^{(k)}_{-[i;j)} \cdot q^{j-i}(-\oq_-^{\frac 2n})^{j-i}  \label{eqn:two 8} 
\end{align}
for all $(i<j) \in \zzz$, $k \in \BN$. Note that we have the following, for all $k \in \BN \sqcup 0$:
\begin{align}
&F_{\pm [i;j)}^{(\pm k)} = \psi_i^{\pm 1} \bE_{\pm [i;j)}^{(\pm k)}  \frac 1{\psi_j^{\pm 1}} \barc^{\mp k} \label{eqn:match 1} \\
&\bF_{\pm [i;j)}^{(\pm k)} = \psi_i^{\pm 1} E_{\pm [i;j)}^{(\pm k)}  \frac 1{\psi_j^{\pm 1}} \barc^{\mp k} \label{eqn:match 2}
\end{align}
When $k < 0$, we will connect the elements $E_{\pm [i;j)}^{(\pm k)}$ and $F_{\pm [i;j)}^{(\pm k)}$ in Propositions \ref{prop:e and f} and \ref{prop:minimal}. Define $F_{\pm [i;i)}^{(\pm k)}, \bF_{\pm [i;i)}^{(\pm k)}$ by formulas \eqref{eqn:match 1}--\eqref{eqn:match 2} and \eqref{eqn:cartan 1}--\eqref{eqn:cartan 4}. \\

\begin{definition}

Consider the subalgebras:
$$
\CA^\uparrow = \Big \langle F_{[i;j)}^{(k)} \Big \rangle_{(i,j) \in \zzz}^{k>0} \qquad \text{and} \qquad \CA^\downarrow = \Big \langle F_{[i;j)}^{(k)} \Big \rangle_{(i,j) \in \zzz}^{k<0}
$$
where we recall that $[i;j) = - [j;i)$ for all $i,j \in \BZ$. \\

\end{definition}

\subsection{} 
\label{sub:study f}

Let us write:
$$
F_{[i;j)}^\mu = F_{[i;j)}^{(k)} \quad \text{and} \quad \bF_{[i;j)}^\mu = \bF_{[i;j)}^{(k)}
$$
for all $(i,j) \in \zzz$ and $k \in \BZ$, where $\mu = \frac k{j-i}$. Then:
$$
F_{[i;j)}^\mu, \bF_{[i;j)}^\mu \in \CB_\mu 
$$
for all $(i,j) \in \zzz$, $\mu \in \BQ$. Moreover, if $i<j$, we have the coproduct formulas:
\begin{align}
&\Delta_\mu \left(F_{[i;j)}^\mu \right) = \sum_{s=i}^j F_{[i;s)}^\mu  \otimes F_{[s;j)}^\mu \frac {\psi_i}{\psi_s} \barc^{\mu(i-s)} \label{eqn:cop f 1} \\
&\Delta_\mu \left(\bF_{[i;j)}^\mu \right) = \sum_{s=i}^j \bF_{[s;j)}^\mu  \otimes \frac {\psi_s}{\psi_j} \barc^{\mu(s-j)} \bF_{[i;s)}^\mu  \label{eqn:cop f 2}  \\
&\Delta_\mu \left(F_{-[i;j)}^\mu \right) = \sum_{s=i}^j F_{-[s;j)}^\mu \frac {\psi_s}{\psi_i} \barc^{\mu(s-i)} \otimes F_{-[i;s)}^\mu \label{eqn:cop f 3} \\ 
&\Delta_\mu \left(\bF_{-[i;j)}^\mu \right) = \sum_{s=i}^j \frac {\psi_j}{\psi_s} \barc^{\mu(j-s)} \bF_{-[i;s)}^\mu \otimes \bF_{-[s;j)}^\mu \label{eqn:cop f 4}  
\end{align}
if $\mu \geq 0$, as a consequence of \eqref{eqn:match 1}--\eqref{eqn:match 2} and \eqref{eqn:cop1}--\eqref{eqn:cop2 anti}. Meanwhile:
\begin{align}
&\Delta_\mu \left(F_{[i;j)}^\mu \right) = \sum_{s=i}^j \frac {\psi_s}{\psi_i} \barc^{\mu(s-i)} F_{[s;j)}^\mu \otimes F_{[i;s)}^\mu \label{eqn:cop f 5} \\
&\Delta_\mu \left(\bF_{[i;j)}^\mu \right) = \sum_{s=i}^j \bF_{[i;s)}^\mu \frac {\psi_j}{\psi_s} \barc^{\mu(j-s)} \otimes \bF_{[s;j)}^\mu \label{eqn:cop f 6}  \\
&\Delta_\mu \left(F_{-[i;j)}^\mu \right) = \sum_{s=i}^j F_{-[i;s)}^\mu  \otimes \frac {\psi_i}{\psi_s} \barc^{\mu(i-s)} F_{-[s;j)}^\mu \label{eqn:cop f 7} \\ 
&\Delta_\mu \left(\bF_{-[i;j)}^\mu \right) = \sum_{s=i}^j \bF_{-[s;j)}^\mu \otimes \bF_{-[i;s)}^\mu \barc^{\mu(s-j)} \frac {\psi_s}{\psi_j} \label{eqn:cop f 8}  
\end{align}
if $\mu < 0$, which is an analogous exercise that we leave to the interested reader. \\

\subsection{} 
\label{sub:connect e and f} 

To go from Theorem \ref{thm:main} to its equivalent version Theorem \ref{thm:main new}, one needs to connect the elements $E,\bE$ with the elements $F,\bF$ defined above. This is achieved by the results in the present Subsection, which will be proved at the end of the paper. We consider any $\mu = \frac ba < 0$ with $\gcd(a,b) = 1$. \\

\begin{definition}
\label{def:yz}

For any $(i,j) , (i',j') \in \zzz$ such that $j-i+j'-i' > 0$, define:
\begin{align}
&Y_{\pm [i;j), \pm [i';j')}^\mu = \sum_{i \leq t \text{ and } s \leq j}^{(s,t) \equiv (i',j')} \bE_{\pm [s;j)}^\mu E_{\pm [i;t)}^\mu \label{eqn:formula y} \\
&\bY_{\pm [i;j), \pm [i';j')}^\mu = \sum_{i \leq t \text{ and } s \leq j}^{(s,t) \equiv (i',j')} E_{\pm [s;j)}^\mu \bE_{\pm [i;t)}^\mu \label{eqn:formula yy} \\
&Z_{\pm [i;j), \pm [i';j')}^\mu = \sum_{i \leq t \text{ and } s \leq j}^{(s,t) \equiv (i',j')} F_{\pm [i;t)}^\mu \bF_{\pm [s;j)}^\mu \label{eqn:formula z} \\
&\bZ_{\pm [i;j), \pm [i';j')}^\mu = \sum_{i \leq t \text{ and } s \leq j}^{(s,t) \equiv (i',j')} \bF_{\pm [i;t)}^\mu F_{\pm [s;j)}^\mu \label{eqn:formula zz}
\end{align}

\end{definition}

\begin{proposition}
\label{prop:identity 1}

For all $(i,j),(i',j') \in \zzz$ such that $j-i+j'-i'>0$ and:
\begin{equation}
\label{eqn:gcd 1}
a|b(j-i)\pm 1 \quad \text{and} \quad a | b(j'-i')\mp 1
\end{equation}
we have the identity:
\begin{equation}
\label{eqn:identity yz}
Y_{\pm [i;j), \pm [i';j')}^\mu = q^{\delta_{j'}^i - \delta_{i'}^j} \bZ_{\pm [i';j'), \pm [i;j)}^\mu
\end{equation}
Similarly, if $a | b(j-i)\mp 1$ and $a | b(j'-i')\pm 1$, then we have:
\begin{equation}
\label{eqn:identity yyzz}
\bY_{\pm [i;j), \pm [i';j')}^\mu = q^{\delta_{i'}^j - \delta_{j'}^i} Z_{\pm [i';j'), \pm [i;j)}^\mu 
\end{equation}

\end{proposition}

\begin{proposition}
\label{prop:identity 2}

For all $(i,j),(i',j') \in \zzz$ such that $j-i+j'-i'>0$ and:
\begin{equation}
\label{eqn:gcd 2}
k = \frac {b(j-i) + \e}a \in \BZ \quad \text{and} \quad k' = \frac {b(j'-i') - \e}a \in \BZ
\end{equation}
(for any $\e \in \{-1,1\}$) we have the identity:
\begin{equation}
\label{eqn:identity yyy}
\bY^\mu_{\pm [i;j), \pm [i';j')} = \sum_{x,x' \in \BZ/n\BZ} q^{\delta_{i'+x'}^i - \delta_{j'+x'}^j} Y^\mu_{\pm [i'+x';j'+x'), \pm [i+x;j+x)} 
\end{equation}
$$
\left[q^{-\delta_i^j} \delta_{x}^0  + \delta_i^j (q - q^{-1}) \frac {\oq_\mp^{\frac 2n \cdot \overline{-\e kx}}}{\oq_\mp^2-1} \right] \left[q^{-\delta_{i'}^{j'}} \delta_{x'}^0  + \delta_{i'}^{j'}  (q - q^{-1}) \frac {\oq_\pm^{\frac 2n \cdot \overline{\e k'x'}}}{\oq_\pm^2-1} \right]
$$
There is also an analogous relation involving $Z$ and $\bZ$, which we will not need. \\

\end{proposition}

\begin{remark} 
\label{rem:works}

Note that the assumption $\mu < 0$ could be removed, if we defined $F^\mu, \bF^\mu$ by the analogues of formulas \eqref{eqn:two 5}--\eqref{eqn:two 8} instead of \eqref{eqn:two 1}--\eqref{eqn:two 4} for $\mu \geq 0$. \\

\end{remark}

\subsection{} 
\label{sub:enter o}

Recall the algebras $\CE_\mu$ of \eqref{eqn:abstract slope}, and replace their simple and imaginary generators by the following elements, for all $i \leq j$, $l \geq 0$ and all applicable $k,k',\hi$: 
\begin{align}
&o_{\pm [i;j)}^{(\pm k)} = \psi_i^{\pm 1} p_{\pm [i;j)}^{(\pm k)} \frac 1{\psi_j^{\pm 1}} \barc^{\mp k} \label{eqn:tp 1} \\
&o_{\pm l\bde, \hi}^{(\pm k')} = p_{\pm l\bde, \hi}^{(\pm k')} c^{\mp l} \barc^{\mp k'} \label{eqn:tp 2}
\end{align}
if $k,k' \geq 0$, while:
\begin{align}
&o_{\pm [i;j)}^{(\pm k)} =  \sum_{x \in \BZ/n\BZ} p_{\pm [i+x;j+x)}^{(\pm k)} \left[q^{-\delta_j^i} \delta_x^0 +  \delta_j^i (q-q^{-1}) \frac {\oq_\pm^{\frac 2n \cdot \overline{-kx}}}{\oq_\pm^2-1} \right] \stackrel{\eqref{eqn:tilde p}}= \tp_{\pm [i;j)}^{(\pm k)} \label{eqn:tp 3} \\
&o_{\pm l\bde, \hi}^{(\pm k')} = \sum_{y \in \BZ/g\BZ} p_{\pm l\bde, r+y}^{(\pm k')} \left[q^{-d} \delta_y^0  + (q^d - q^{-d}) \frac {\oq_\pm^{\frac {2d}n \cdot \widehat{\frac {- k' y}d}}}{\oq_\pm^{\frac {2dg}n}-1} \right] \label{eqn:tp 4}
\end{align}
if $k,k' < 0$, where $d = \gcd(k',nl)$, $g = \gcd(n,\frac {nl}d)$ and $\widehat{x}$ is defined as the element in the set $\{1,...,g\}$ which is congruent to $x$ modulo $g$. We will use the notation:
$$
\to_{\pm [i;j)}^{(\pm k)} = \sum_{x \in \BZ/n\BZ} o_{\pm [i+x;j+x)}^{(\pm k)} \left[q^{-\delta_j^i} \delta_x^0 +  \delta_j^i (q-q^{-1}) \frac {\oq_\pm^{\frac 2n \cdot \overline{-k x}}}{\oq_\pm^2-1} \right] = 
$$
\begin{equation}
\label{eqn:tp 5}
\stackrel{\eqref{eqn:tilde p}}= \psi_i^{\pm 1} \tp_{\pm [i;j)}^{(\pm k)} \frac 1{\psi_j^{\pm 1}} \barc^{\mp k}
\end{equation}
if $k,k' \geq 0$, while we define:
\begin{equation}
\label{eqn:tp 6}
\to_{\pm [i;j)}^{(\pm k)} = \sum_{x \in \BZ/n\BZ} o_{\pm [i+x;j+x)}^{(\pm k)} \left[q^{-\delta_j^i} \delta_x^0 +  \delta_j^i (q-q^{-1}) \frac {\oq_\mp^{\frac 2n \cdot \overline{k x}}}{\oq_\mp^2-1} \right] \stackrel{\eqref{eqn:tilde p equivalent}}= p_{\pm [i;j)}^{(\pm k)}
\end{equation}
if $k,k'<0$. Finally, let us write:
\begin{equation}
\label{eqn:stipulate}
f_{\pm [i;j)}^{(\pm k)} \in \CE_{\frac k{j-i}}  \quad \text{and} \quad  \bf_{\pm [i;j)}^{(\pm k)} \in \CE_{\frac k{j-i}}
\end{equation}
for the images of the elements \eqref{eqn:two 1}--\eqref{eqn:two 8} under the isomorphisms \eqref{eqn:mini}. Therefore, we have the natural analogues of Definition \ref{def:yz}, Proposition \ref{prop:identity 1} and Proposition \ref{prop:identity 2} when $E,\bE,F,\bF \in \CB_\mu$ are replaced by the symbols $e,\be,f,\bf \in \CE_\mu$. \\

\subsection{} 
\label{sub:main new}

With the notation above in mind, we are ready to state and prove the following analogue of Theorem \ref{thm:main}. Specifically, both \eqref{eqn:generators} and \eqref{eqn:generators new} are presentations of the same quantum toroidal algebra, but while the relations in the former distinguish the vertical line (i.e. $i=j$), the relations in the latter distinguish the horizontal line (i.e. $k=0$). \\

\begin{theorem}
\label{thm:main new} 
	
We have an algebra isomorphism:
\begin{equation}
\label{eqn:generators new}
\UU \cong \boxed{\DD := \Big \langle \CE_{\mu} \Big \rangle_{\mu \in \BQ \sqcup \infty} \Big/\text{relations \eqref{eqn:case 1 new}--\eqref{eqn:case 4 new}}}
\end{equation}
where the defining relations in $\DD$ are of the four types below. \\

\noindent \underline{Type 1:} for all \footnote{We only allow $k = 0$ in the following if $i<j$ and $k' = 0$ if $l>0$} $(i,j) \in \zzz$, $l \in \BZ$ and $k,k' \geq 0$ such that:
$$
d:=  \det \begin{pmatrix} k & k' \\ j-i & nl \end{pmatrix}
$$ 
satisfies $|d| = \gcd(k',nl)$, we set (letting $g = \gcd \left(n, \frac {nl}d \right)$) for any $r \in \BZ/g\BZ$:
\begin{equation}
\label{eqn:case 1 new}
\Big[ o_{\pm[i;j)}^{(\pm k)}, o_{\pm l\bde, \hi}^{(\pm k')} \Big] = \pm o_{\pm [i;j+ln)}^{(\pm k \pm k')} \left(\delta_{i \text{ mod }g}^r \oq_\pm^{\frac dn} - \delta_{j \text{ mod }g}^r \oq_\pm^{-\frac dn} \right)
\end{equation}
\text{ }

\noindent \underline{Type 2}: under the same assumptions as in Type 1, we set:
\begin{equation}
\label{eqn:case 2 new}
\begin{aligned}
&\Big[ o_{\pm [i;j)}^{(\pm k)}, o_{\mp l\bde, \hi}^{(\mp k')} \Big] = \pm (c^{\pm l} \barc^{\pm k'})^{\frac {d}{|d|}} o_{\pm [i;j-nl)}^{(\pm k \mp k')} \left(\delta_{i \text{ mod }g}^r \oq_\mp^{\frac dn} - \delta_{j \text{ mod }g}^r \oq_\mp^{-\frac dn} \right) \\
&\Big[ \to_{\pm [i;j)}^{(\pm k)}, o_{\mp l\bde, \hi}^{(\mp k')} \Big] = \pm \left( \frac {\psi_j^{\pm 1} \barc^{\pm k}}{\psi_i^{\pm 1} q^{1-\delta_j^i}}\right)^{\frac {d}{|d|}} o^{(\pm k \mp k')}_{\pm[i;j-nl)} \left(\delta_{i \text{ mod }g}^r \oq_\mp^{\frac dn} - \delta_{j \text{ mod }g}^r \oq_\mp^{- \frac dn} \right)
\end{aligned} 
\end{equation}
with the first equation in \eqref{eqn:case 2 new} holding if $k > k'$ (or $k=k'$ and $j-i > nl$) and the second equation holding if $k < k'$ (or $k=k'$ and $j-i < nl$). \\

\noindent \underline{Type 3:} for all \footnote{We only allow $k = 0$ in the following if $i<j$ and $k' = 0$ if $i'<j'$ \label{foot 2}} $(i, j), (i', j') \in \zzz$ and $k,k' \geq 0$ such that:
$$
\det \begin{pmatrix} k & k' \\ j-i & j'-i' \end{pmatrix} = \gcd(k+k',j-i+j'-i')
$$
we set:
\begin{equation}
\label{eqn:case 3 new}
o_{\pm [i;j)}^{(\pm k)} \to_{\pm [i';j')}^{(\pm k')} q^{\delta_{j'}^i - \delta_{i'}^i} - \to_{\pm [i';j')}^{(\pm k')} o_{\pm [i;j)}^{(\pm k)} q^{\delta_{j'}^j - \delta_{i'}^j} = \sum^{(s,t) \equiv (i';j')}_{i \leq t \text{ and } s \leq j} \frac {f^\mu_{\pm [s,j)} \bf_{\pm [i;t)}^\mu}{q^{-1}-q} 
\end{equation}
where $\mu = \frac {k+k'}{j - i + j' - i'}$. \\

\noindent \underline{Type 4:} for all \textsuperscript{\emph{\ref{foot 2}}} $(i,j), (i',j') \in \zzz$ and $k,k' \geq 0$ such that:
$$
\det \begin{pmatrix} k & k' \\ j-i & j'-i' \end{pmatrix} = \gcd(k-k',j-i-j'+i')
$$
we set:
\begin{equation}
 \label{eqn:case 4 new}
\left[o_{\pm [i;j)}^{(\pm k)}, o_{\mp [i';j')}^{(\mp k')} \right] = \frac 1{q-q^{-1}}  \begin{cases} \displaystyle \sum^{(s,t) \equiv (i',j')}_{i \leq s \leq t \leq j} \bf^{\mu}_{\pm [t,j)} \frac {\psi_{j'}^{\pm 1} \barc^{\pm k'}}{\psi_{i'}^{\pm 1}} f_{\pm [i;s)}^{\mu} &\text{if } k > k' \\ \\
\displaystyle \sum^{(s,t) \equiv (i,j)}_{i' \leq s \leq t \leq j'} f_{\mp [t,j')}^{\mu} \frac{\psi_j^{\pm 1}  \barc^{\pm k}}{\psi_i^{\pm 1}} \bf_{\mp [i';s)}^{\mu} &\text{if } k \leq k'
\end{cases}
\end{equation}
where $\mu = \frac {k-k'}{j - i - j' + i'}$. \\ 

\end{theorem}

\begin{proof} We will prove that relations \eqref{eqn:case 1}, \eqref{eqn:case 2}, \eqref{eqn:case 3}, \eqref{eqn:case 4} are equivalent to relations \eqref{eqn:case 1 new}, \eqref{eqn:case 2 new}, \eqref{eqn:case 3 new}, \eqref{eqn:case 4 new}, respectively. We will prove the first and the third of these equivalences, and leave the second and fourth as exercises for the interested reader. Let us start by proving the equivalence of \eqref{eqn:case 1} and \eqref{eqn:case 1 new}. \\

\begin{itemize}[leftmargin=*]

\item when $i \leq j$ and $l \geq 0$, the $p$ and $o$ generators are connected by formulas \eqref{eqn:tp 1} and \eqref{eqn:tp 2}, hence relations \eqref{eqn:case 1} and \eqref{eqn:case 1 new} are equivalent in virtue of the fact that $\psi_i$ and $\psi_j$ commute with $p_{\pm l\bde, \hi}^{(\pm k')}$, and $\barc$ commutes with everything. \\

\item when $i > j$ and $l < 0$, we have:
$$
\left [o_{\pm [i;j)}^{(\pm k)}, o_{\pm l\bde, r}^{(\pm k')} \right] \stackrel{\eqref{eqn:tp 3}, \eqref{eqn:tp 4}}= \sum_{x \in \BZ/n\BZ} \sum_{y \in \BZ/g\BZ} \left [p_{\pm [i+x;j+x)}^{(\pm k)}, p_{\pm l\bde, r+y}^{(\pm k')} \right] \cdot
$$
$$
\left(q^{-\delta_j^i} \delta_x^0 +  \delta_j^i (q-q^{-1}) \frac {\oq_\mp^{\frac 2n \cdot \overline{kx}}}{\oq_\mp^2-1} \right) \left( q^{-|d|} \delta_{y}^0  + (q^{|d|} - q^{-|d|}) \frac {\oq_\mp^{\frac {2|d|}n \cdot \widehat{\frac {k'y}{|d|}}}}{\oq_\mp^{\frac {2|d|g}n}-1} \right) 
$$
$$
\stackrel{\eqref{eqn:case 1}}= \sum_{x \in \BZ/n\BZ} \sum_{y \in \BZ/g\BZ} \pm p_{\pm [i+x;j+nl+x)}^{(\pm k \pm k')} \left( \delta_{i+x \text{ mod }g}^{r+y} \oq_\mp^{-\frac dn} - \delta_{j+x \text{ mod }g}^{r+y} \oq_\mp^{\frac dn} \right)
$$
$$
\left(q^{-\delta_j^i} \delta_x^0 +  \delta_j^i (q-q^{-1}) \frac {\oq_\mp^{\frac 2n \cdot \overline{kx}}}{\oq_\mp^2-1} \right) \left( q^{-|d|} \delta_{y}^0  + (q^{|d|} - q^{-|d|}) \frac {\oq_\mp^{\frac {2|d|}n \cdot \widehat{\frac {k'y}{|d|}}}}{\oq_\mp^{\frac {2|d|g}n}-1} \right) = 
$$
$$
= \sum_{x \in \BZ/n\BZ} \pm p_{\pm [i+x;j+nl+x)}^{(\pm k \pm k')} \squiggly{\left(q^{-\delta_j^i} \delta_x^0 +  \delta_j^i (q-q^{-1}) \frac {\oq_\mp^{\frac 2n \cdot \overline{kx}}}{\oq_\mp^2-1} \right)} 
$$
$$
\squiggly{\left[ q^{-|d|} \left( \oq_\mp^{-\frac dn}  \delta_{i+x \text{ mod }g}^r - \oq_\mp^{\frac dn} \delta_{j+x \text{ mod }g}^r \right) + (q^{|d|} - q^{-|d|}) \left( \frac {\oq_\mp^{\frac {2|d|}n \cdot \widehat{\frac {k'(i+x-r)}{|d|}}-\frac dn}}{\oq_\mp^{\frac {2|d|g}n}-1} - \frac {\oq_\mp^{\frac {2|d|}n \cdot \widehat{\frac {k'(j+x-r)}{|d|}}+\frac dn}}{\oq_\mp^{\frac {2|d|g}n}-1} \right) \right]}
$$
To complete the proof of relation \eqref{eqn:case 1 new}, we need to show that the expression with the squiggly underline is equal to:
$$
\left(q^{-\delta_j^i} \delta_x^i +  \delta_j^i (q-q^{-1}) \frac {\oq_\mp^{\frac 2n \cdot \overline{(k+k')(x-i)}}}{\oq_\mp^2-1} \right) \left(\delta_{i \text{ mod }g}^r \oq_\pm^{\frac dn} - \delta_{j \text{ mod }g}^r \oq_\pm^{-\frac dn} \right)
$$
for all $x \in \{1,...,n\}$. This is an elementary consequence of the assumption that $knl - k'(j-i) = d$ and $|d| = \gcd(k',nl)$, and one deals with the cases of $i\not \equiv j$ mod $n$ and $i \equiv j$ mod $n$ separately (in the latter case, the assumptions imply $n|k'$, hence $n|d$, hence $g=1$). We leave the details to the interested reader. \\

\end{itemize}

\noindent When $j-i$ and $l$ have different signs (here we consider the number 0 to be among the positive integers), there are only two cases when our assumption holds: \\

\begin{itemize}

\item when  $i=j$, $k=1$ and $nl|k'$ (which implies $g=1$), relation \eqref{eqn:case 1 new} reads:
$$
\Big[ o_{\pm[i;i)}^{(\pm 1)}, o_{\pm l\bde, 1}^{(\pm k')} \Big] = \pm o_{\pm [i;i+ln)}^{(\pm 1 \pm k')} \left( \oq_\pm^l - \oq_\pm^{-l} \right)
$$
If we apply the substitutions \eqref{eqn:tp 1}, \eqref{eqn:tp 3}, \eqref{eqn:tp 4}, the formula above becomes equivalent to:
\begin{multline*}
\left[ p_{\pm[i;i)}^{(\pm 1)} \barc^{\mp 1}, p_{\pm l\bde, 1}^{(\pm k')} \frac {q^{nl}q_{\mp}^l - q^{-nl}q_{\mp}^{-l}}{q_{\mp}^l-q_{\mp}^{-l}} \right] = \\ = \pm  \left( \oq_\pm^l - \oq_\pm^{-l} \right) \sum_{x \in \BZ/n\BZ} p_{\pm [i+x;i+x+ln)}^{(\pm 1 \pm k')} \left[q^{-1}\delta_x^0 + (q-q^{-1})\frac {\oq_\mp^{\frac 2n \cdot \overline{x}}}{\oq_\mp^2-1} \right]
\end{multline*}
By using $q^n \oq_\mp = \oq_\pm^{-1}$, the formula above is equivalent to:
$$
\left[ p_{\pm[i;i)}^{(\pm 1)}, p_{\pm l\bde, 1}^{(\pm k')} \right] = \pm \barc^{\pm 1} \left( \oq_\mp^{-l} - \oq_\mp^{l} \right) \sum_{x \in \BZ/n\BZ} p_{\pm [i+x;i+x+ln)}^{(\pm 1 \pm k')} \left(q^{-1}\delta_x^0 + (q-q^{-1})\frac {\oq_\mp^{\frac 2n \cdot \overline{x}}}{\oq_\mp^2-1} \right)
$$
If we apply Lemma \ref{lem:elem}, then the formula above is equivalent to:
$$
\left[ \sum_{x \in \BZ/n\BZ} p_{\pm[i+x;i+x)}^{(\pm 1)} \left(q^{-1}\delta_x^0 + (q-q^{-1})\frac {\oq_\pm^{\frac 2n \cdot \overline{-x}}}{\oq_\pm^2-1} \right), p_{\pm l\bde, 1}^{(\pm k')} \right] = \pm \barc^{\pm 1} \left( \oq_\mp^{-l} - \oq_\mp^{l} \right) p_{\pm [i;i+ln)}^{(\pm 1 \pm k')} 
$$
which is precisely the second option in \eqref{eqn:case 2} for $l<0$. \\

\item when $l=0$ and $i=j+1$ (which implies $g=n$), relation \eqref{eqn:case 1 new} reads:
$$
\Big[ o_{\pm[i;i-1)}^{(\pm k)}, o_{\pm 0\bde, \hi}^{(\pm k')} \Big] = \pm o_{\pm [i;i-1)}^{(\pm k \pm k')} \left(\delta_{i}^r \oq_\pm^{\frac {k'}n} - \delta_{i-1}^r \oq_\pm^{-\frac {k'}n} \right)
$$
Using \eqref{eqn:tp 2} and \eqref{eqn:tp 3}, the relation above becomes equivalent to:
$$
\Big[ p_{\mp[i-1;i)}^{(\pm k)}, p_{\pm 0\bde, \hi}^{(\pm k')} \barc^{\mp k'} \Big] = \pm p_{\mp [i-1;i)}^{(\pm k \pm k')} \left(\delta_{i}^r \oq_\pm^{\frac {k'}n} - \delta_{i-1}^r \oq_\pm^{-\frac {k'}n} \right)
$$
The formula above is precisely the first option of \eqref{eqn:case 2}. \\

\end{itemize}

\noindent Let us now prove the equivalence of \eqref{eqn:case 3} and \eqref{eqn:case 3 new}. \\

\begin{itemize}[leftmargin=*]

\item when $i \leq j$ and $i' \leq j'$, we have (to keep our formulas legible, we set $\barc=1$ in what follows, since it is clear that we will encounter the same powers of $\barc$ in the left and right-hand sides of the equations below):
$$
o_{\pm [i;j)}^{(\pm k)} \to_{\pm [i';j')}^{(\pm k')} q^{\delta_{j'}^i - \delta_{i'}^i} - \to_{\pm [i';j')}^{(\pm k')} o_{\pm [i;j)}^{(\pm k)} q^{\delta_{j'}^j - \delta_{i'}^j} \stackrel{\eqref{eqn:tp 1}, \eqref{eqn:tp 5}}= 
$$
$$
= \psi_i^{\pm 1} p_{\pm [i;j)}^{(\pm k)} \frac {\psi_{i'}^{\pm 1}}{\psi_j^{\pm 1}} \tp_{\pm [i';j')}^{(\pm k')} \frac 1{\psi_{j'}^{\pm 1}} q^{\delta_{j'}^i - \delta_{i'}^i} - \psi_{i'}^{\pm 1} \tp_{\pm [i';j')}^{(\pm k')} \frac {\psi_{i}^{\pm 1}}{\psi_{j'}^{\pm 1}} p_{\pm [i;j)}^{(\pm k)} \frac 1{\psi_{j}^{\pm 1}} q^{\delta_{j'}^j - \delta_{i'}^j} \stackrel{\eqref{eqn:rel 1}}=
$$
$$
= \psi_i^{\pm 1} \psi_{i'}^{\pm 1} p_{\pm [i;j)}^{(\pm k)} \tp_{\pm [i';j')}^{(\pm k')} \frac 1{\psi_{j}^{\pm 1} \psi_{j'}^{\pm 1}} q^{\delta_{j'}^{i} - \delta_{j'}^{j}} - \psi_i^{\pm 1} \psi_{i'}^{\pm 1} \tp_{\pm [i';j')}^{(\pm k')} p_{\pm [i;j)}^{(\pm k)} \frac 1{\psi_{j}^{\pm 1} \psi_{j'}^{\pm 1}} q^{\delta_{i'}^i - \delta_{i'}^j} \stackrel{\eqref{eqn:case 3}}=
$$
$$
= q^{\delta_{i'}^i - \delta_{j'}^j} \cdot \psi_i^{\pm 1} \psi_{i'}^{\pm 1} \left[ \sum^{(s,t) \equiv (i',j')}_{i \leq t \text{ and } s \leq j} \frac {\be^\mu_{\pm [s,j)} e_{\pm [i;t)}^\mu}{q^{-1}-q} \right]  \frac 1{\psi_{j}^{\pm 1} \psi_{j'}^{\pm 1}} = 
$$
$$
= \sum^{(s,t) \equiv (i',j')}_{i \leq t \text{ and } s \leq j} \frac {\psi_s^{\pm 1} \be^\mu_{\pm [s,j)} \frac {\psi_i^{\pm 1}}{\psi_j^{\pm 1}} e_{\pm [i;t)}^\mu \frac 1{\psi_t^{\pm 1}}}{q^{-1}-q} \stackrel{\eqref{eqn:match 1}, \eqref{eqn:match 2}}= \sum^{(s,t) \equiv (i',j')}_{i \leq t \text{ and } s \leq j} \frac {f^\mu_{\pm [s,j)} \bf_{\pm [i;t)}^\mu}{q^{-1}-q}
$$
(indeed, the last formula holds because \eqref{eqn:stipulate} stipulates that the $e,\be$'s are to the $f,\bf$'s as the $E,\bE$'s are to the $F,\bF$'s, and the latter are related by \eqref{eqn:match 1}--\eqref{eqn:match 2}). \\

\item when $i > j$ and $i' > j'$, we have: 
$$
o_{\pm [i;j)}^{(\pm k)} \to_{\pm [i';j')}^{(\pm k')} q^{\delta_{j'}^i - \delta_{i'}^i} - \to_{\pm [i';j')}^{(\pm k')} o_{\pm [i;j)}^{(\pm k)} q^{\delta_{j'}^j - \delta_{i'}^j} \stackrel{\eqref{eqn:tp 3}, \eqref{eqn:tp 6}}=
$$
$$
= \tp_{\mp [j;i)}^{(\mp (- k))} p_{\mp [j';i')}^{(\mp (-k'))} q^{\delta_{j'}^i - \delta_{i'}^i} - p_{\mp [j';i')}^{(\mp (-k'))} \tp_{\mp [j;i)}^{(\mp (- k))} q^{\delta_{j'}^j - \delta_{i'}^j} \stackrel{\eqref{eqn:tilde p},\eqref{eqn:tilde p equivalent}}=
$$
$$
= \sum_{x,x' \in \BZ/n\BZ} \left[ p_{\mp [j+x;i+x)}^{(\mp (- k))} \tp_{\mp [j'+x';i'+x')}^{(\mp (-k'))} q^{\delta_{j'}^i - \delta_{i'}^i} - \tp_{\mp [j'+x';i'+x')}^{(\mp (-k'))} p_{\mp [j+x;i+x)}^{(\mp (- k))} q^{\delta_{j'}^j - \delta_{i'}^j} \right] \cdot 
$$
\begin{equation}
\label{eqn:miel}
\left[q^{-\delta_j^i} \delta_x^0 +  \delta_j^i (q-q^{-1}) \frac {\oq_\mp^{\frac 2n \cdot \overline{kx}}}{\oq_\mp^2-1} \right] \left[q^{-\delta_{j'}^{i'}} \delta_{x'}^{0} +  \delta_{j'}^{i'} (q-q^{-1}) \frac {\oq_\pm^{\frac 2n \cdot \overline{-k'x'}}}{\oq_\pm^2-1} \right]
\end{equation}
Note that the identity:
$$
\delta_{j'}^i + \delta_{i'}^j - \delta_{i'}^i - \delta_{j'}^j = \delta_{i'+x'}^{j+x} + \delta_{j'+x'}^{i+x} - \delta_{j'+x'}^{j+x} - \delta_{i'+x'}^{i+x}
$$
holds for all $x,x'$ that have non-zero coefficient in \eqref{eqn:miel} (this is because $i\not \equiv j \Rightarrow n|x$ and $i'\not \equiv j' \Rightarrow n|x'$). Therefore, formula \eqref{eqn:miel} equals:
$$
\sum_{x,x' \in \BZ/n\BZ} \left[ p_{\mp [j+x;i+x)}^{(\mp (- k))} \tp_{\mp [j'+x';i'+x')}^{(\mp (-k'))} q^{\delta_{i'+x'}^{j+x} - \delta_{j'+x'}^{j+x}} - \tp_{\mp [j'+x';i'+x')}^{(\mp (-k'))} p_{\mp [j+x;i+x)}^{(\mp (- k))} q^{\delta_{i'+x'}^{i+x} - \delta_{j'+x'}^{i+x}} \right] \cdot 
$$
$$
q^{\delta_{j'}^j - \delta_{i'}^j - \delta_{i'+x'}^{i+x} + \delta_{j'+x'}^{i+x}} \left[q^{-\delta_j^i} \delta_x^0 +  \delta_j^i (q-q^{-1}) \frac {\oq_\mp^{\frac 2n \cdot \overline{kx}}}{\oq_\mp^2-1} \right] \left[q^{-\delta_{j'}^{i'}} \delta_{x'}^{0} +  \delta_{j'}^{i'} (q-q^{-1}) \frac {\oq_\pm^{\frac 2n \cdot \overline{-k'x'}}}{\oq_\pm^2-1} \right]
$$
By invoking \eqref{eqn:case 3}, the expression above equals:
$$
\sum_{x,x' \in \BZ/n\BZ} \left[ \sum^{(s,t) \equiv (j'+x',i'+x')}_{j+x \leq t \text{ and } s \leq i+x} \frac {\be^\mu_{\mp [s,i+x)} e_{\mp [j+x;t)}^\mu}{q^{-1}-q}  \right] q^{\delta_{j'}^j - \delta_{i'}^j - \delta_{i'+x'}^{i+x} + \delta_{j'+x'}^{i+x}} \cdot 
$$
$$
 \left[q^{-\delta_j^i} \delta_x^0 +  \delta_j^i (q-q^{-1}) \frac {\oq_\mp^{\frac 2n \cdot \overline{kx}}}{\oq_\mp^2-1} \right] \left[q^{-\delta_{j'}^{i'}} \delta_{x'}^{0} +  \delta_{j'}^{i'} (q-q^{-1}) \frac {\oq_\pm^{\frac 2n \cdot \overline{-k'x'}}}{\oq_\pm^2-1} \right]
$$
Note that the identity:
$$
\delta_{j'}^j - \delta_{i'}^j - \delta_{i'+x'}^{i+x} + \delta_{j'+x'}^{i+x} = \delta_i^{j'} - \delta_j^{i'} + \delta_{j+x}^{j'+x'} - \delta_{i+x}^{i'} 
$$
holds for all $x,x'$ that have non-zero coefficient in the expression above. Thus, we may apply \eqref{eqn:identity yyy} to conclude that the expression above equals:
$$
\frac {q^{\delta_{j'}^i - \delta_{i'}^j} \bY_{\mp [j';i'), \mp [j;i)}}{q^{-1}-q} \stackrel{\eqref{eqn:identity yyzz}}= \frac {Z_{\mp [j;i), \mp [j';i')}}{q^{-1}-q}
$$
(the notations $\bY$ and $Z$ are defined in \eqref{eqn:formula yy} and \eqref{eqn:formula z}). The right-hand side of the expression above precisely equals the right-hand side of \eqref{eqn:case 3 new}, as required. \\

\end{itemize} 

\noindent When $i-j > 0 \geq i' - j'$, our hypothesis also holds in the following cases: \\

\begin{itemize}

\item when $j-i|k+1$, $i'=j'$ and $k'=1$, relation \eqref{eqn:case 3 new} reads:
$$
o_{\pm [i;j)}^{(\pm k)} \to_{\pm [i';i')}^{(\pm 1)} - \to_{\pm [i';i')}^{(\pm 1)} o_{\pm [i;j)}^{(\pm k)}  = \sum^{s \equiv i'}_{i \leq s \leq j} \frac {f^\mu_{\pm [s,j)} \bf_{\pm [i;s)}^\mu}{q^{-1}-q} 
$$
We may use the substitutions \eqref{eqn:tp 1} and \eqref{eqn:tp 3} to rewrite the left-hand side, and \eqref{eqn:identity yyzz} to rewrite the right-hand side as:
$$
\tp_{\mp [j;i)}^{(\pm k)} \tp_{\pm [i';i')}^{(\pm 1)} - \tp_{\pm [i';i')}^{(\pm 1)} \tp_{\mp [j;i)}^{(\pm k)}  = q^{\delta_i^{i'} - \delta_j^{i'}} \barc^{\pm 1} \sum^{s \equiv i'}_{i \leq s \leq j} \frac {e^\mu_{\mp [j,s)} \be_{\mp [s;i)}^\mu}{q^{-1}-q} 
$$
If we convert the $\tp$'s into $p$'s using formula \eqref{eqn:tilde p}, and apply \eqref{eqn:identity yyy} to the right-hand side, the formula above reduces to \eqref{eqn:case 4}. \\

\item when $i=j+1$ and $i' = j'-1$, relation \eqref{eqn:case 3 new} reads:
$$
o_{\pm [j+1;j)}^{(\pm k)} \to_{\pm [i';i'+1)}^{(\pm k')} q^{\delta_{i'}^{j} - \delta_{i'}^{j+1}} - \to_{\pm [i';i'+1)}^{(\pm k')} o_{\pm [j+1;j)}^{(\pm k)} q^{\delta_{i'+1}^{j} - \delta_{i'}^{j}} = \delta_{i'}^j \sum_{a+b = k+k'} 
\frac {f^{(a)}_{\pm [j;j)} \of^{(b)}_{\pm [j+1;j+1)}}{q^{-1}-q} 
$$
With the substitutions \eqref{eqn:match 1}, \eqref{eqn:match 2}, \eqref{eqn:tp 1}, \eqref{eqn:tp 3}, this formula reads:
$$
p_{\mp [j;j+1)}^{(\pm k)} \psi_{i'}^{\pm 1} p_{\pm [i';i'+1)}^{(\pm k')} \frac {\barc^{\mp k'}}{\psi_{i'+1}^{\pm 1}} q^{\delta_{i'+1}^{j+1} - \delta_{i'}^{j+1}} - \psi_{i'}^{\pm 1} p_{\pm [i';i'+1)}^{(\pm k')} \frac {\barc^{\mp k'}}{\psi_{i'+1}^{\pm 1}} p_{\mp [j;j+1)}^{(\pm k)} q^{\delta_{i'+1}^{j} - \delta_{i'}^{j}} = 
$$
$$
= \delta_{i'}^j \sum_{a+b = k+k'} 
\frac {\be^{(a)}_{\pm [j;j)} e^{(b)}_{\pm [j+1;j+1)} \barc^{\mp k \mp k'}}{q^{-1}-q} 
$$
If we move $\psi_{i'}^\pm$ (respectively $\psi_{i'+1}^{\pm 1}$) to the left (respectively right) of the formula above using \eqref{eqn:rel 1}, then we obtain precisely \eqref{eqn:case 4}. \\

\end{itemize}

\end{proof}

\subsection{} 
\label{sub:degree 1}

By analogy with Proposition \ref{prop:equiv}, we have the following: \\

\begin{proposition}
\label{prop:equiv new}

Let $\DD^\pm\subset \DD$ be the subalgebras generated by $o_{[i;j)}^{(k)}$ and $o_{l\bde,\hi}^{(k')}$ with $k,k' \in \BN$ and all $i,j,l,\hi$, modulo relations \eqref{eqn:case 1 new} and \eqref{eqn:case 3 new}. Then we have:
\begin{equation}
\label{eqn:qwerty new}
\DD \cong \DD^+ \otimes \CE_0 \otimes \DD^- \Big / \text{relations \eqref{eqn:special rel 1 new}--\eqref{eqn:special rel 5 new}}
\end{equation}
where the following are special cases of \eqref{eqn:case 1 new}--\eqref{eqn:case 4 new}:
\begin{align}
&\left[ o_{\pm [i;j)}^{(\pm 1)}, o_{\pm l\bde, 1}^{(0)} \right] = \pm o_{\pm [i;j+nl)}^{(\pm 1)} \left( \oq_\pm^l - \oq_\pm^{-l} \right) \label{eqn:special rel 1 new} \\
&\left[ o_{\pm [i;j)}^{(\pm 1)}, o_{\mp l\bde, 1}^{(0)} \right] = \pm o_{\pm [i;j-nl)}^{(\pm 1)} c^{\pm l}  \left( \oq_\mp^l - \oq_\mp^{-l} \right) \label{eqn:special rel 2 new}
\end{align}
\begin{multline}
\label{eqn:special rel 3 new}
o_{\pm [i;j)}^{(\pm 1)} o_{\pm [s;s+1)}^{(0)} q^{\delta_{s+1}^i - \delta_s^i} - o_{\pm [s;s+1)}^{(0)} o_{\pm [i;j)}^{(\pm 1)} q^{\delta_{s+1}^j - \delta_s^j} = \\ = \pm \left( \delta_{s+1}^i \cdot \oq_\mp^{-\frac 1n}  o_{\pm [i-1;j)}^{(\pm 1)} - \delta_s^j \cdot \oq_\mp^{\frac 1n} o_{\pm [i;j+1)}^{(\pm 1)} \right)
\end{multline} 
\begin{equation}
\label{eqn:special rel 4 new}
\left[ o_{\pm [i;j)}^{(\pm 1)}, o_{\mp [s;s+1)}^{(0)} \right] = \pm \left( \delta_s^i \cdot \oq_\mp^{\frac 1n} o_{\pm [i+1;j)}^{(\pm 1)} \frac {\psi_{i+1}^{\pm 1}}{\psi_i^{\pm 1}} - \delta_{s+1}^j \cdot \oq_\mp^{-\frac 1n} \frac {\psi_j^{\pm 1}}{\psi_{j-1}^{\pm 1}} o_{\pm [i;j-1)}^{(\pm 1)}  \right)
\end{equation}
and:
\begin{equation}
\label{eqn:special rel 5 new}
\left[o_{[i;j)}^{(1)}, o_{[i';j')}^{(-1)} \right] = \frac 1{q^{-1} - q} 
\end{equation}
$$
\left(\sum_{\left \lceil \frac {i-j'}n \right \rceil \leq k \leq \left \lfloor \frac {j-i'}n \right \rfloor} f_{[i'+nk,j)}^{\mu} \frac{\psi_{j'}}{\psi_{i'} \barc} \bf_{[i;j'+nk)}^{\mu}  - \sum_{\left \lceil \frac {j'-i}n \right \rceil \leq k \leq \left \lfloor \frac {i'-j}n \right \rfloor} f_{- [j+nk,i')}^{\mu} \frac{\psi_j \barc}{\psi_i} \bf_{- [j';i+nk)}^{\mu}\right)
$$
for all $(i,j), (i',j') \in \zzz$, $l \in \BZ$, $s \in \BZ/n\BZ$. \\

\end{proposition}

\noindent The proof of the Proposition above is completely analogous to that of Proposition \ref{prop:equiv}, so we leave it as an exercise to the interested reader. \\

\subsection{} In the remainder of this paper, we will prove Propositions \ref{prop:identity 1} and \ref{prop:identity 2}. As we have seen in Lemma \ref{lem:unique}, a frequent strategy for proving identities of shuffle elements in $\CB_\mu$, such as \eqref{eqn:identity yz}, is the following two-step process: one first shows that the LHS and RHS have the same intermediate terms under the coproduct $\Delta_\mu$, and then one shows that the LHS and RHS have the same value under linear maps which ``detect" primitive elements. Such linear maps are the $\alpha_{\pm [i;j)}$ of \eqref{eqn:alpha}--\eqref{eqn:beta}, however, these never give nice formulas on both $Y,\bY$ and $Z,\bZ$ of \eqref{eqn:formula y}--\eqref{eqn:formula zz}. To remedy this issue, we will consider a different choice of linear maps, and we will show in Lemma \ref{lem:min} that they detect primitive elements. Fix generic parameters $x_1,...,x_n \in \BF$ and introduce the following linear maps (inspired by the similar construction of \cite{FT}):
\begin{equation}
\label{eqn:def rho}
\rho : \CA^\pm \rightarrow \BF
\end{equation}
given for any $R^\pm(...,z_{i1},z_{i2},...,z_{ik_i},...)$ of homogeneous degree $h$ by the formula:
\begin{equation}
\label{eqn:rho}
\rho (R^\pm) = \frac {R^\pm \left(...,x_i,x_iq^2,...,x_iq^{2(k_i-1)},...\right) q^{-\frac {h|k|}n}}{\prod_{1\leq i,i'\leq n} \prod_{1\leq a \leq k_i, 1 \leq a' \leq k_{i'}}^{a>a'}\zeta \left( \frac {x_iq^{2a}}{x_{i'} q^{2a'}} \right)}
\end{equation}
where $x_i$ has color $i$. We have the following analogue of Proposition \ref{prop:pseudo}: \\

\begin{proposition} 
\label{prop:rho mult}

If $\deg R_1 = (\bk, h_1)$ and $\deg R_2 = (l\bde, h_2)$, we have:
\begin{equation}
\label{eqn:rho plus}
\rho(R_1R_2) = \rho(R_1)\rho(R_2) q^{\frac {h_1 nl - h_2 |k|}n}
\end{equation}
while if $\deg R_2 = (-l\bde, h_2)$ and $\deg R_1 = (-\bk,h_1)$, we have:
\begin{equation}
\label{eqn:rho minus}
\rho(R_2R_1) = \rho(R_1)\rho(R_2) q^{\frac {h_1 nl - h_2 |k|}n}
\end{equation}
	
\end{proposition}

\begin{proof} We will only prove \eqref{eqn:rho plus}, and leave the analogous \eqref{eqn:rho minus} as an exercise to the interested reader. By definition, $\rho(R_1R_2)$ involves summing over partitions:
$$
\{x_i, x_i q^2,...,x_i q^{2(k_i+l-1)}\}_{1 \leq i \leq n} = V_1 \sqcup V_2
$$
multiplying together the evaluations of $R_1$, $R_2$ at the set of variables $V_1$, $V_2$ (respectively) and then multiplying the result by $\prod_{y \in V_1} \prod_{z \in V_2} \zeta(y/z)$. However, because $\zeta(q^{-2} \text{ of color }0) = 0$, the only non-zero contribution produced by the procedure above happens for:
$$
V_1 = \{x_iq^{2l},...,x_iq^{2(k_i+l-1)} \}, \quad V_2 = \{x_i,x_iq^2,...,x_iq^{2(l-1)}\}
$$
We conclude that:
$$
(R_1R_2) \Big|_{z_{ia} \mapsto x_i q^{2(a-1)}} = R_1 \Big|_{z_{ia} \mapsto x_i q^{2(l+a-1)}} \cdot R_2 \Big|_{z_{ia} \mapsto x_i q^{2(a-1)}} \cdot \prod_{i,i' = 1}^n \prod_{l \leq a < k_i +l}^{0 \leq a' < l} \zeta \left( \frac {x_i q^{2a}}{x_{i'}q^{2a'}} \right) 
$$
Since $R_1|_{z_{ia} \mapsto x_i q^{2(l+a-1)}} = R_1 |_{z_{ia} \mapsto x_i q^{2(a-1)}} \cdot q^{2h_1 l}$, the formula above implies \eqref{eqn:rho plus}.
	
\end{proof}

\begin{proposition}
\label{prop:rho computations}

For any $(i<j) \in \zzz$ and $h \in \BZ$, let $\mu = \frac h{j-i}$. We have:
\begin{align} 
&\rho (A^{(h)}_{\pm [i;j)}) = \frac {q^{\alpha_{i,j}^{\mu}} \prod_{a=i}^{j-1} \left( x_{\bara} \oq^{\frac {2\bara}n} \right)^{\left \lceil \frac {h(a-i+1)}{j-i} \right \rceil - \left \lceil \frac {h(a-i)}{j-i} \right \rceil}}{\prod_{a=i+1}^{j-1} \left( 1-\frac {x_{a-1}q^2}{x_a} \right)} \left[ \frac {1 - \frac {x_{i-1}}{x_i}}{1 - \frac {x_{i-1}q^2}{x_i}} \right]^{\delta_i^j} \label{eqn:rho a} \\
&\rho (\bA^{(h)}_{\pm [i;j)}) = \frac {q^{\overline{\alpha}_{i,j}^{\mu}} \prod_{a=i}^{j-1} \left( x_{\bara} \oq^{\frac {2\bara}n} \right)^{\left \lfloor \frac {h(a-i+1)}{j-i} \right \rfloor - \left \lfloor \frac {h(a-i)}{j-i} \right \rfloor}}{\prod_{a=i+1}^{j-1} \left( 1-\frac {x_a}{x_{a-1}q^2} \right)} \left[ \frac {1 - \frac {x_i}{x_{i-1}}}{1 - \frac {x_i}{x_{i-1}q^2}} \right]^{\delta_i^j} \label{eqn:rho aa} \\ 
&\rho (B^{(h)}_{\pm [i;j)}) = \frac {q^{\beta_{i,j}^{\mu}} \prod_{a=i}^{j-1} \left( x_{\bara} \oq^{\frac {2\bara}n} \right)^{\left \lfloor \frac {h(a-i+1)}{j-i} \right \rfloor - \left \lfloor \frac {h(a-i)}{j-i} \right \rfloor}}{ \prod_{a=i+1}^{j-1} \left( 1-\frac {x_a}{x_{a-1}q^2} \right)} \label{eqn:rho b} \\ 
&\rho (\bB^{(h)}_{\pm [i;j)}) = \frac {q^{\overline{\beta}_{i,j}^{\mu}} \prod_{a=i}^{j-1} \left( x_{\bara} \oq^{\frac {2\bara}n} \right)^{\left \lceil \frac {h(a-i+1)}{j-i} \right \rceil - \left \lceil \frac {h(a-i)}{j-i} \right \rceil}}{\prod_{a=i+1}^{j-1} \left( 1-\frac {x_{a-1}q^2}{x_a} \right)} \label{eqn:rho bb} 
\end{align} 
where we extend the notation $x_i$ to all $i\in \BZ$ by the rule $x_{i+n}\oq^2 = x_i$, and define:
\begin{align*}
&\alpha_{i,j}^\mu = (2\mu(j-i)+1) \left \lfloor \frac {j-i}n \right \rfloor - 2\sum_{k=1}^{\left \lfloor \frac {j-i}n \right \rfloor} \left \lceil \mu n k \right \rceil - \frac {\mu(j-i)^2}n \\
&\overline{\alpha}_{i,j}^\mu = (2\mu(j-i)-1) \left \lfloor \frac {j-i}n \right \rfloor - 2\sum_{k=1}^{\left \lfloor \frac {j-i}n \right \rfloor} \left \lfloor \mu n k \right \rfloor - \frac {\mu(j-i)^2}n \\
&\beta_{i,j}^\mu = (2\mu(j-i)+1) \left( \left \lceil \frac {j-i}n \right \rceil - 1 \right) + 1 - 2\sum_{k=1}^{\left \lceil \frac {j-i}n \right \rceil - 1} \left \lceil \mu n k \right \rceil - \frac {\mu(j-i)^2}n - j+i \\
&\overline{\beta}_{i,j}^\mu = (2\mu(j-i)-1) \left( \left \lceil \frac {j-i}n \right \rceil - 1 \right) - 1 - 2\sum_{k=1}^{\left \lceil \frac {j-i}n \right \rceil - 1} \left \lfloor \mu n k \right \rfloor - \frac {\mu(j-i)^2}n + j-i
\end{align*}

\end{proposition}

\begin{proof} We will prove \eqref{eqn:rho a}, and leave the other three formulas as analogous exercises to the interested reader. Recall that we have:
\begin{equation}
\label{eqn:new A}
A^{(h)}_{\pm [i;j)} = \text{Sym} \left( a^{(h)}_{\pm [i;j)} \right) 
\end{equation}
where:
\begin{equation}
\label{eqn:new a}
a^{(h)}_{\pm [i;j)} = \frac {\prod_{a=i}^{j-1} (z_a\oq^{\frac {2a}n})^{\left \lceil \frac {h(a-i+1)}{j-i} \right \rceil - \left \lceil \frac {h(a-i)}{j-i} \right \rceil}}{\left(1 - \frac {z_{i}q^2}{z_{i+1}}  \right) ... \left(1 - \frac {z_{j-2}q^2 }{z_{j-1}} \right)} \prod_{i\leq a < b < j} \zeta \left( \frac {z_b}{z_a} \right) 
\end{equation}
Applying the linear map $\rho$ to the rational function \eqref{eqn:new A} entails specializing $A^{(h)}_{\pm [i;j)}$ at:
\begin{equation}
\label{eqn:spec}
z_a \mapsto x_{\bar{a}} \oq^{2- 2 \left \lceil \frac an \right \rceil} q^{2 \left \lfloor \frac {a-i}n \right \rfloor} \qquad \forall a \in \{i,...,j-1\}
\end{equation}
The fact that $\zeta(q^{-2} \text{ of color }0) = 0$ means that we obtain the same result by specializing the rational function \eqref{eqn:new a} acoording to \eqref{eqn:spec}. Thus, we have:
\begin{equation}
\label{eqn:ash}
\rho \left( A^{(h)}_{\pm [i;j)} \right) \stackrel{\eqref{eqn:rho}}=  \frac {a^{(h)}_{\pm [i;j)} \Big|_{\text{evaluation \eqref{eqn:spec}}} \cdot q^{-\frac {h(j-i)}n}}{\prod_{1\leq s,s' \leq n} \prod_{1\leq u \leq k_s, 1 \leq u' \leq k_{s'}}^{u>u'}\zeta \left( \frac {x_sq^{2u}}{x_{s'} q^{2u'}} \right)}
\end{equation}
where $\bk = [i;j)$ and $k_s$ denotes the number of elements in the set $\{i,...,j-1\}$ congruent to $s$ modulo $n$. It is easy to see that when we specialize the $\zeta$ factors of \eqref{eqn:new a} according to \eqref{eqn:spec}, we pick up precisely the product of $\zeta$ factors in \eqref{eqn:ash}, but with ``$u>u'$" replaced by ``$u>u'$ \text{ or } $u=u'$, $s \equiv i-1$, $s' \equiv i$ mod $n$". Hence:
$$
\rho \left( A^{(h)}_{\pm [i;j)} \right) = \frac {\prod_{a=i}^{j-1} (x_{\bar{a}}  \oq^{\frac {2\bar{a}}n} q^{2 \left \lfloor \frac {a-i}n \right \rfloor})^{\left \lceil \frac {h(a-i+1)}{j-i} \right \rceil - \left \lceil \frac {h(a-i)}{j-i} \right \rceil}}{q^{\frac {h(j-i)}n} \prod_{a=i+1}^{j-1}\left(1 - \frac {x_{\overline{a-1}}q^{2-2\delta_a^i}\oq^{2\delta_a^1}}{x_{\bar{a}}}  \right)} \cdot \prod_{i \leq a \leq j-n}^{a \equiv i} \zeta \left(\frac {x_{\overline{a-1}} \oq^{2\delta_a^1}}{x_{\bara}} \right)
$$
The number of $\zeta$ factors is equal to $\left \lfloor \frac {j-i}n \right \rfloor$, while the number of linear factors in the denominator of the expression above is equal to $\left \lfloor \frac {j-i-1}n \right \rfloor$. These two numbers are equal unless $i \equiv j$ modulo $n$, in which case the number of $\zeta$ factors is one more than the other number. Therefore, plugging in the explicit formula for $\zeta$ yields:
\begin{equation}
\label{eqn:diez}
\rho \left( A^{(h)}_{\pm [i;j)} \right) = q^\# \cdot \frac {\prod_{a=i}^{j-1} (x_{\bar{a}} \oq^{\frac {2\bar{a}}n})^{\left \lceil \frac {h(a-i+1)}{j-i} \right \rceil - \left \lceil \frac {h(a-i)}{j-i} \right \rceil}}{\prod_{a=i+1}^{j-1}\left(1 - \frac {x_{\overline{a-1}}\oq^{2\delta_a^1}}{x_{\bar{a}}}  \right)} \left[ \frac {1 - \frac {x_{i-1}}{x_i}}{1 - \frac {x_{i-1}q^2}{x_i}} \right]^{\delta_i^j}
\end{equation}
where:
$$
\# = \sum_{a=i}^{j-1} 2 \left \lfloor \frac {a-i}n \right \rfloor \left( \left \lceil \frac {h(a-i+1)}{j-i} \right \rceil - \left \lceil \frac {h(a-i)}{j-i} \right \rceil \right) + \left \lfloor \frac {j-i}n \right \rfloor - \frac {h(j-i)}n
$$
It remains to compute the number $\#$. To this end, let us write $j = i + nl + r$ and $a = i + nk + s$ for $r,s \in \{0,...,n-1\}$. Thus, we have:
$$
\# = \sum_{k=0}^{l - 1} 2 k \left( \left \lceil \frac {hn(k+1)}{j-i} \right \rceil - \left \lceil \frac {hnk}{j-i} \right \rceil \right) + 2 l \left(h - \left \lceil \frac {hnl}{j-i} \right \rceil \right) + l - \frac {h(j-i)}n = \alpha_{i,j}^{\mu}
$$
precisely as prescribed by \eqref{eqn:rho a}. 

\end{proof}

\subsection{} 
\label{sub:not final} 

Let us introduce the following notation:
\begin{equation}
\label{eqn:basic 1}
\fS_{i,j} = \prod_{a=i}^{j-1} \left(1 - \frac {x_{a-1}q^2}{x_a} \right) \quad \text{and} \quad \fT_{i,j} = \prod_{a=i}^{j-1} \left(1 - \frac {x_a}{x_{a-1}q^2} \right)  
\end{equation}
as well as:
\begin{align}
&\sigma^{\mu}_{i,j} = \prod_{a=i}^{j-1} \left( x_{\bara} \oq^{\frac {2\bara}n} \right)^{\left \lceil \mu(a-i+1) \right \rceil - \left \lceil \mu(a-i) \right \rceil} \label{eqn:basic 2} \\
&\tau^{\mu}_{i,j} = \prod_{a=i}^{j-1} \left( x_{\bara} \oq^{\frac {2\bara}n} \right)^{\left \lfloor \mu(a-i+1) \right \rfloor - \left \lfloor \mu(a-i) \right \rfloor} \label{eqn:basic 3}
\end{align}
for all $\mu \in \BQ$ such that $\mu(j-i) \in \BZ$. The following identities are obvious:
\begin{equation}
\label{eqn:basic 4}
\fS_{i+n,j+n} = \fS_{i,j} \qquad \text{and} \qquad \fT_{i+n,j+n} = \fT_{i,j}
\end{equation}
and:
\begin{align} 
&\fS_{i,j+n} = \fS_{i,j} \cdot \mathfrak{s}, \qquad \text{where} \qquad \mathfrak{s} = \prod_{a=1}^{n} \left(1 - \frac {x_{a-1}q^2}{x_a} \right) \label{eqn:basic 5} \\
&\fT_{i,j+n} = \fT_{i,j} \cdot \mathfrak{t}, \qquad \ \text{where} \qquad \mathfrak{t} = \prod_{a=1}^{n} \left(1 - \frac {x_a}{x_{a-1}q^2} \right) \label{eqn:basic 6}
\end{align} 
Meanwhile, if we write $\mu = \frac ba$ with $\gcd(a,b) = 1$ and set $g = \gcd(n,a)$, then:
\begin{equation}
\label{eqn:basic 7}
\sigma^{\mu}_{i+\frac {na}g,j+\frac {na}g} = \sigma^{\mu}_{i,j} \quad \text{and} \quad \tau^{\mu}_{i+\frac {na}g,j+\frac {na}g} = \tau^{\mu}_{i,j}
\end{equation}
and:
\begin{align}
&\sigma_{i,j+\frac {na}g}^\mu = \sigma_{i,j}^\mu \cdot s_{\mu, j}, \quad  s_{\mu, j} = \prod_{a=j}^{j+\frac {na}g - 1} \left( x_{\bara} \oq^{\frac {2\bara}n} \right)^{\left \lceil \mu(a-j+1) \right \rceil - \left \lceil \mu(a-j) \right \rceil} \label{eqn:basic 8} \\
&\tau_{i,j+\frac {na}g}^\mu = \tau_{i,j}^\mu \cdot t_{\mu,j}, \quad  t_{\mu,j} = \prod_{a=j}^{j+\frac {na}g - 1} \left( x_{\bara} \oq^{\frac {2\bara}n} \right)^{\left \lfloor \mu(a-j+1) \right \rfloor - \left \lfloor \mu(a-j) \right \rfloor} \label{eqn:basic 9} 
\end{align}
(it is easy to see that $s_{\mu,j}$ and $t_{\mu,j}$ only depend on $j$ mod $g$). Finally, we have:
\begin{align}
&\alpha_{i,j}^{\mu} = \overline{\beta}_{i,j}^{\mu} + 1 - \delta_i^j+ 2 \left \lfloor \frac {(j-i)g}{na} \right \rfloor - j + i \label{eqn:basic 10} \\
&\overline{\alpha}_{i,j}^{\mu} = \beta_{i,j}^{\mu} - 1 + \delta_i^j -  2 \left \lfloor \frac {(j-i)g}{na} \right \rfloor + j - i \label{eqn:basic 11}
\end{align} 
and the identities:
\begin{align}
&\alpha_{i,j-\frac {na}g}^{\mu} = \alpha_{i,j}^{\mu} + \frac {\mu n a}g  - 1 & &\beta_{i,j-\frac {na}g}^{\mu} = \beta_{i,j}^{\mu} + \frac {\mu n a}g + \frac {na}g - 1 \label{eqn:basic 12} \\
&\overline{\alpha}_{i,j-\frac {na}g}^{\mu} = \overline{\alpha}_{i,j}^{\mu} + \frac {\mu n a}g  + 1 & &\overline{\beta}_{i,j-\frac {na}g}^{\mu} = \overline{\beta}_{i,j}^{\mu} + \frac {\mu n a}g - \frac {na}g + 1 \label{eqn:basic 13}
\end{align}
All the formulas above are elementary, and we leave them to the interested reader. \\

\subsection{} The main reason the maps $\rho$ are useful to us is that they ``detect" primitive elements, according to the following: \\

\begin{lemma}
\label{lem:min}

For all $\mu \in \BQ$, any element of $\CB_\mu$ which is primitive for $\Delta_\mu$ and is annihilated by $\rho$ (for generic parameters $x_1,...,x_n$), vanishes. \\

\end{lemma}

\begin{proof} It is well-known (\cite{Tor}) that primitive elements have horizontal degree $l\bde$ for some $l \in \BN$, so we will focus on the graded piece of $\CB_\mu$ of degree:
\begin{equation}
\label{eqn:degree min}
(l\bde, k)
\end{equation}
where $\frac k{nl} = \mu$. Let us consider all $n$--tuples of partitions of $l\bde$:
$$
\bla = (\lambda^{(1)},...,\lambda^{(n)}) \qquad \text{where} \qquad \lambda^{(i)} = (\lambda^{(i)}_1 \geq \lambda^{(i)}_2 \geq ... ) \vdash l
$$
Take the partial ordering on $n$--tuples of partitions where we set $\bmu \geq \bla$ if, for all $i\in \{1,...,n\}$, we have $\mu^{(i)} \geq \lambda^{(i)}$ with respect to the dominance ordering on usual partitions. Then we consider the linear maps:
$$
\ph_{\bla} : V \rightarrow \BF(...,x_{i,1},x_{i,2},...)
$$
given by:
$$
\ph_\bla (R) = R(...,x_{i,1},x_{i,1}q^2,...,x_{i,1} q^{2(\lambda^{(i)}_1-1)},x_{i,2}, x_{i,2} q^2,...,x_{i,2} q^{2(\lambda^{(i)}_2-1)},...)
$$
(as $i$ runs over $\{1,....,n\}$, and $x_{i,j}$ are variables of color $i$) where $V$ denotes the vector space of primitive elements of $\CB_\mu$ of degree \eqref{eqn:degree min}. The subspaces:
$$
V_\bla = \bigcap_{\bmu > \bla} \text{Ker }\ph_\bmu
$$ 
form a filtration of $V$. Therefore, to prove the Lemma, it suffices to show that:
\begin{equation}
\label{eqn:claim}
\ph_\bla(V_\bla) = 0
\end{equation}
for all $n$--tuples of partitions $\bla \neq \bla_{\text{max}} = ((l), (l),...,(l))$. Indeed, once we have \eqref{eqn:claim}, let us consider a shuffle element $R \in V$ as in the statement of the Lemma. Our assumption on $R$ implies that $R \in \text{Ker }\ph_{\bla_{\text{max}}}$. Then \eqref{eqn:claim} allows one to prove (by induction on $\bla$ in decreasing order) that $R \in V_\bla$ for all $n$--tuples of partitions $\bla$. Hence $R \in V_{\bla_{\text{min}}}$, where $\bla_{\text{min}} = ((1^l), (1^l),...,(1^l))$, and then \eqref{eqn:claim} implies that $\ph_{\bla_{\text{min}}}(R) = 0$. However, $\ph_{\bla_{\text{min}}}$ is just the identity map, so we conclude that $R = 0$. \\

\noindent It remains to prove \eqref{eqn:claim}, so let us choose an arbitrary $R \in V_\bla$. By \eqref{eqn:shuf}, we have:
\begin{equation}
\label{eqn:ph}
\ph_\bla(R) = \frac {r(...,x_{i,1},x_{i,2},...)}{\prod_{i=1}^n \prod_{\alpha,\beta} (x_{i,\alpha} - x_{i+1,\beta} q^{...})^{\lambda_{\alpha}^{(i)} \lambda_{\beta}^{(i+1)}}}
\end{equation}
In the formula above and throughout, $q^{...}$ simply refers to an integer power of $q$, that will not be crucial to our argument. Moreover, the notation $(x-x'q^{...})^k$ will refer to a product of $k$ factors of the form $x-x'q^{...}$. The powers of $q$ which appear in these factors might be different, but this fact does not change the validity of our argument in any way, and so we prefer this slightly imprecise notation in order to ensure legibility. For any natural numbers $d,d'$, relation \eqref{eqn:wheel} implies that:
$$
r(...,\underbrace{x,xq^2,...,xq^{2(d-1)}}_{\text{of color }i},...,\underbrace{x',x'q^2,...,x'q^{2(d'-1)}}_{\text{of color }i+1},...)
$$
is divisible by $(x-x'q^{...})^{dd'-\min(d,d')}$ (indeed, assume without loss of generality that $d' \leq d$, so each of the $d'$ variables of color $i+1$ participates in $d-1$ vanishing conditions \eqref{eqn:wheel} with the $d$ variables of color $i$). Hence, \eqref{eqn:ph} may be written as:
\begin{equation}
\label{eqn:phi}
\ph_\bla(R) = \frac {s(...,x_{i,1},x_{i,2},...)}{\prod_{i=1}^n \prod_{\alpha,\beta} (x_{i,\alpha} - x_{i+1,\beta} q^{...})^{\min(\lambda_{\alpha}^{(i)}, \lambda_{\beta}^{(i+1)})}}
\end{equation}
for some Laurent polynomial $s$. We recall that the integer powers of $q$ which appear in the formula above are arbitrary, and it is no problem for our argument if $(x-x'q^{...})^m$ actually refers to a ratio $(x-x'q^{...})^{a}/(x-x'q^{...})^b$, for some $a-b=m$. However, the fact that $R \in V_\bla$ means that $\ph_\bmu(R) = 0$ for all $\bmu > \bla$, which implies that the analogues of formula \eqref{eqn:phi} for partitions $\bmu > \bla$ vanish. This implies that the Laurent polynomial $s$ vanishes whenever:
\begin{equation}
\label{eqn:specialization}
x_{i,\alpha} \in \left\{ x_{i,\beta} q^2,...,x_{i,\beta} q^{2\lambda^{(i)}_\beta} \right\} \bigsqcup \left \{ x_{i,\beta} q^{-2\lambda^{(i)}_\alpha},..., x_{i,\beta} q^{2(\lambda^{(i)}_\beta - \lambda^{(i)}_\alpha -1)} \right\}
\end{equation}
for any $i$ and any $\alpha < \beta$. Therefore, \eqref{eqn:phi} implies that:
\begin{equation}
\label{eqn:fi}
\ph_\bla(R) = t(...,x_{i,1},x_{i,2},...) \prod_{i=1}^n \frac {\prod_{\alpha < \beta} (x_{i,\alpha} - x_{i,\beta} q^{...})^{2 \lambda_{\beta}^{(i)}}}{\prod_{\alpha,\beta} (x_{i,\alpha} - x_{i+1,\beta} q^{...})^{\min(\lambda_{\alpha}^{(i)}, \lambda_{\beta}^{(i+1)})}}
\end{equation}
for some Laurent polynomial $t$. We wish to show that $t = 0$. To this end, the fact that $R$ has homogeneous degree $k$ implies that $t$ is homogeneous of degree:
\begin{equation}
\label{eqn:x}
\deg_{\{x_{i,\alpha}\}^{i \in \{1,...,n\}}_{\alpha \geq 1}} (t) = k + \sum_{i=1}^n \left[  \sum_{\alpha,\beta} \min(\lambda_{\alpha}^{(i)}, \lambda_{\beta}^{(i+1)}) - \sum_{\alpha} 2\lambda^{(i)}_\alpha (\alpha-1) \right]
\end{equation}
The fact that $R \in \CB_\mu \subset \CA_{\leq \mu}$ implies that:
$$
\deg_{\{x_{i,1}\}^{i \in \{1,...,n\}}} (t) < \frac {k \sum_{i=1}^n \lambda^{(i)}_{1}}{nl} + \sum_{i=1}^n \left[ \min(\lambda_{1}^{(i)}, \lambda_{1}^{(i+1)}) + \right.
$$
\begin{equation}
\label{eqn:y}
\left. + \sum_{\alpha \geq 2} \min(\lambda_{\alpha}^{(i)}, \lambda_{1}^{(i+1)}) + \sum_{\alpha \geq 2} \min(\lambda_{1}^{(i)}, \lambda_{\alpha}^{(i+1)}) - \sum_{\alpha \geq 2} 2\lambda^{(i)}_\alpha  \right] 
\end{equation}
$$
\deg_{\{x_{i,\alpha}\}^{i \in \{1,...,n\}}_{\alpha \geq 2}} (t) < \frac {k \sum_{i=1}^n \sum_{\alpha \geq 2} \lambda^{(i)}_{\alpha}}{nl} + \sum_{i=1}^n \left[ \sum_{\alpha,\beta \geq 2} \min(\lambda_{\alpha}^{(i)}, \lambda_{\beta}^{(i+1)}) + \right. 
$$
\begin{equation}
\label{eqn:z}
\left. \sum_{\alpha \geq 2} \min(\lambda_{\alpha}^{(i)}, \lambda_{1}^{(i+1)}) + \sum_{\alpha \geq 2} \min(\lambda_{1}^{(i)}, \lambda_{\alpha}^{(i+1)}) - \sum_{\alpha \geq 2} 2\lambda^{(i)}_\alpha (\alpha-1) \right]
\end{equation}
where the inequalities are strict because $R$ is primitive (here we use the fact that $\bla \neq \bla_{\text{max}}$). If $t\neq 0$, then we would have:
$$
\text{LHS of \eqref{eqn:x}} \leq \text{LHS of \eqref{eqn:y}} + \text{LHS of \eqref{eqn:z}} 
$$ 
and therefore:
$$
0 < \sum_{i=1}^n \left[ \sum_{\alpha \geq 2} \min(\lambda_{1}^{(i)}, \lambda_{\alpha}^{(i+1)}) + \min(\lambda_{1}^{(i+1)}, \lambda_{\alpha}^{(i)}) - \sum_{\alpha \geq 2} 2\lambda^{(i)}_\alpha \right]
$$
Since the right-hand side of the expression above is $\leq 0$ (simply because $\min(x,y) \leq y$), we obtain a contradiction. Hence $t = 0 \Rightarrow \ph_\bla(R) = 0$, and thus \eqref{eqn:claim} is proved.
	
\end{proof} 

\subsection{} By comparing \eqref{eqn:cop1}--\eqref{eqn:cop2 anti} with \eqref{eqn:cop f 5}--\eqref{eqn:cop f 8}, we conclude that when $\mu < 0$:
\begin{equation}
\label{eqn:shuf e}
E_{\pm [i;j)}^\mu \quad \text{and} \quad \bE_{\pm [i;j)}^\mu
\end{equation}
enjoy the same coproduct formulas as:
\begin{equation}
\label{eqn:shuf f}
F_{\pm [i;j)}^\mu q^{\delta_j^i - 1} \quad \text{and} \quad \bF_{\pm [i;j)}^\mu q^{1-\delta_j^i}
\end{equation}
respectively. Using the formulas of Proposition \ref{prop:rho computations}, we will show the following: \\

\begin{proposition} 
\label{prop:efg}

With the notation of Subsection \ref{sub:not final}, we have for all $\mu < 0$:
\begin{align}
&\rho \left( E_{\pm [i;j)}^\mu \right) = \sum_{k=0}^{\left \lfloor \frac {(j-i)g}{na} \right \rfloor} \rho \left( F_{\pm \left[i; j - \frac {nak}g \right)}^\mu q^{\delta_j^i - 1} \right) \gamma_{\pm k, \mu, j}  \label{eqn:rho 1} \\
&\rho \left( \bE_{\pm [i;j)}^\mu \right) = \sum_{k=0}^{\left \lfloor \frac {(j-i)g}{na} \right \rfloor} \rho \left( \bF_{\pm \left[i; j - \frac {nak}g \right)}^\mu q^{1-\delta_j^i} \right) \overline{\gamma}_{\pm k, \mu, j} \label{eqn:rho 2}
\end{align}
where $\gamma_{0, \mu, j} = \overline{\gamma}_{0,\mu,j} = 1$ and: 
\begin{align*} 
&\gamma_{k, \mu, j} = (1-q^2)\left[\frac {t_{\mu,j}}{{\mathfrak{t}}^{\frac ag} q^{\frac {\mu n a}g + 1}} \right]^k & &\gamma_{- k, \mu, j} = (1-q^{2})\left[\frac {s_{\mu,j}}{{\mathfrak{s}}^{\frac ag} q^{\frac {\mu n a}g - \frac {na}g+1}} \right]^k \\
&\overline{\gamma}_{k, \mu, j} = (1-q^{-2})\left[\frac {(-\oq_-^{-\frac 2n})^{\frac {na}g}s_{\mu,j}}{\mathfrak{s}^{\frac ag} q^{\frac {\mu n a}g - 1}} \right]^k & &\overline{\gamma}_{-k, \mu, j} = (1-q^{-2})\left[\frac {(-\oq_+^{-\frac 2n})^{\frac {na}g}t_{\mu,j}}{\mathfrak{t}^{\frac ag} q^{\frac {\mu n a}g + \frac {n a}g - 1}} \right]^k
\end{align*} 
$\forall \ k>0$. Just like $s_{\mu,j}$, $t_{\mu,j}$, the quantities $\gamma_{\pm k, \mu, j}, \overline{\gamma}_{\pm k, \mu, j}$ only depend on $j$ mod $g$. \\

\end{proposition} 

\begin{proof} Recall the formulas \eqref{eqn:e plus}--\eqref{eqn:be minus} and \eqref{eqn:two 5}--\eqref{eqn:two 8}:
\begin{align*} 
&E_{[i;j)}^{\mu} = \bA^\mu_{[i;j)} \qquad & &F_{[i;j)}^{\mu} = B^{\mu}_{[i;j)} \cdot q^{j-i} \\
&\bE_{[i;j)}^\mu = A^\mu_{[i;j)} \cdot (-\oq_-^{\frac 2n})^{i-j} \qquad & &\bF_{[i;j)}^{\mu} = \bB^{\mu}_{[i;j)} \cdot q^{j-i} (-\oq_+^{\frac 2n})^{j-i}  \\
&E_{-[i;j)}^{\mu} = \bB^\mu_{-[i;j)} \qquad & &F_{-[i;j)}^{\mu} = A^{\mu}_{-[i;j)} \cdot q^{j-i} \\
&\bE_{-[i;j)}^{\mu} = B^\mu_{-[i;j)} \cdot (-\oq_+^{\frac 2n})^{i-j} \qquad & &\bF_{-[i;j)}^{\mu} = \bA^{\mu}_{-[i;j)} \cdot q^{j-i}(-\oq_-^{\frac 2n})^{j-i}  
\end{align*}
Let us prove \eqref{eqn:rho 1} when the sign is $+$, and leave the other cases as analogous exercises to the interested reader. By \eqref{eqn:rho aa}, we have:
\begin{equation}
\label{eqn:lhs}
\rho \left( E_{[i;j)}^\mu \right) = \rho \left( \bA_{[i;j)}^\mu \right) = \frac {q^{\overline{\alpha}_{i,j}^\mu} \tau_{i,j}^\mu}{\fT_{i,j}} \left(1 - \frac {x_i q^{2\delta_i^j}}{x_{i-1} q^2}  \right) 
\end{equation}
Similarly, formula \eqref{eqn:rho b} implies that:
$$
\rho \left( F_{[i;j)}^\mu q^{\delta_j^i - 1} \right) = \rho \left( B_{[i;j)}^\mu q^{j-i+\delta_j^i - 1} \right) = \frac {q^{\beta_{i,j}^\mu+j-i+\delta_j^i-1} \tau_{i,j}^\mu}{\fT_{i,j}} \left(1 - \frac {x_i}{x_{i-1} q^2}  \right) 
$$
Using formulas \eqref{eqn:basic 6}, \eqref{eqn:basic 9} and \eqref{eqn:basic 12}, we obtain:
$$
\rho \left( F_{\left[i;j- \frac {nak}g \right)}^\mu q^{\delta_j^i - 1} \right)  = \frac {q^{\beta_{i,j}^\mu+j-i+k \left(\frac {\mu n a}g - 1 \right) + \delta_j^i-1} \tau_{i,j}^\mu}{\fT_{i,j}} \left[\frac {{\mathfrak{t}}^{\frac ag}}{t_{\mu,j}} \right]^k \underbrace{\left(1 - \frac {x_i}{x_{i-1} q^2}  \right)}_{\text{does not appear for }k = \frac {(j-i)g}{na}} 
$$
Summing over all $k$ from 0 to $\left \lfloor \frac {(j-i)g}{na} \right \rfloor$, we conclude that the RHS of \eqref{eqn:rho 1} equals:
$$
\sum_{k=0}^{\left \lfloor \frac {(j-i)g}{na} \right \rfloor}  \frac {q^{\beta_{i,j}^\mu + j - i + k \left(\frac {\mu n a}g - 1 \right) +\delta_j^i-1} \tau_{i,j}^\mu}{\fT_{i,j}} \left[\frac {{\mathfrak{t}}^{\frac ag}}{t_{\mu,j}} \right]^k \gamma_{k,\mu,j} \underbrace{\left(1 - \frac {x_i}{x_{i-1} q^2} \right)}_{\text{does not appear for }k = \frac {(j-i)g}{na}}  
$$
The fact that the expression above equals \eqref{eqn:lhs} (which would imply \eqref{eqn:rho 1}) is an immediate consequence of \eqref{eqn:basic 11} and the following elementary identities:
$$
q^{ - 2 \left \lfloor \frac {(j-i)g}{na} \right \rfloor} = \sum_{k=0}^{\left \lfloor \frac {(j-i)g}{na} \right \rfloor} \left[q^{\frac {\mu n a}g - 1} \frac {{\mathfrak{t}}^{\frac ag}}{t_{\mu,j}} \right]^k \gamma_{k,\mu,j}
$$
if $i \not \equiv j$ modulo $n$, while: 
$$
q^{ - 2 \frac {(j-i)g}{na} } \left(1 - \frac {x_i}{x_{i-1}} \right) = \sum_{k=0}^{\frac {(j-i)g}{na}} \left[q^{\frac {\mu n a}g - 1} \frac {{\mathfrak{t}}^{\frac ag}}{t_{\mu,j}} \right]^k \gamma_{k,\mu,j} \underbrace{\left(1 - \frac {x_i}{x_{i-1} q^2} \right)}_{\text{does not appear for }k = \frac {(j-i)g}{na}}  
$$
if $i \equiv j$ modulo $n$ (which implies that $\frac {na}g | j-i$). We leave these identities as exercises. \\

\end{proof} 

\subsection{} For any reduced fraction $\mu = \frac ba$ with $g=\gcd(n,a)$, recall that $\CB_\mu$ has $g$ countable families of primitive (with respect to the coproduct $\Delta_\mu$) elements:
$$
P_{\pm k \bde, \hi}^\mu \in \CB_\mu 
$$
for every $k \in \BN$ and $\hi \in \BZ/g\BZ$. For all collections of scalars $\alpha_{\pm k, \hi}$, the expression:
\begin{equation}
\label{eqn:g}
1 + \sum_{k=1}^\infty \frac {G_{\pm k}^\mu}{x^k} = \exp \left[ \sum_{k=1}^\infty \sum_{r \in \BZ/g\BZ} \frac {P_{\pm k\bde,\hi}^\mu \alpha_{\pm k, \hi}}{kx^k} \right]
\end{equation}
is the generating series of a family of group-like elements of $\CB_\mu$. We will write
\begin{equation}
\label{eqn:gg}
1 + \sum_{k=1}^\infty \frac {\bG_{\pm k}^\mu}{x^k} = \exp \left[- \sum_{k=1}^\infty \sum_{r \in \BZ/g\BZ} \frac {P_{\pm k\bde,\hi}^\mu \alpha_{\pm k, \hi}}{kx^k} \right]
\end{equation}
for the inverse of the generating series \eqref{eqn:g}. \\

\begin{proposition}
\label{prop:e and f}

For any $\mu = \frac ba < 0$ with $\gcd(a,b) = 1$, $\gcd(n,a) = g$, we have:
\begin{align} 
&E_{\pm [i;j)}^\mu = \sum_{k=0}^{\left \lfloor \frac {(j-i)g}{na} \right \rfloor} F_{\pm \left[i; j - \frac {nak}g \right)}^\mu q^{\delta_j^i - 1} G_{\pm k, \widehat{i}}^\mu \label{eqn:e and f 1} \\
&\bE_{\pm [i;j)}^\mu = \sum_{k=0}^{\left \lfloor \frac {(j-i)g}{na} \right \rfloor} \bF_{\pm \left[i; j - \frac {nak}g \right)}^\mu q^{1-\delta_j^i} \bG_{\pm k, \widehat{i \pm v}}^\mu \label{eqn:e and f 2}
\end{align}
where $\widehat{i}$ denotes the residue of $i$ modulo $g$, $v \in \BZ/g\BZ$ is uniquely defined by:
$$
bv \equiv 1 \text{ mod }a
$$
and for all $i$, $G_{\pm k,\widehat{i}}^\mu, \bG_{\pm k,\widehat{i}}^\mu$ are of the form \eqref{eqn:g}, \eqref{eqn:gg} for certain scalars $\alpha_{\pm k, \hi}$. \\

\end{proposition}

\begin{proof} The natural analogue of Proposition \ref{prop:unique} applied to the tensor product $U_{q}(\dot{\fgl}_{\frac ng})^{\otimes g} \cong \CB_\mu$ implies that we have \eqref{eqn:e and f 1} and \eqref{eqn:e and f 2}, where:
\begin{align}
&1 + \sum_{k=1}^\infty G_{\pm k,\widehat{i}}^\mu = \exp \left[ \sum_{k=1}^\infty \sum_{r \in \BZ/g\BZ} \frac {P_{\pm k\bde,r}^\mu \alpha_{\pm k, \widehat{i}, r}}k \right] \label{eqn:g to p} \\
&1 + \sum_{k=1}^\infty \bG_{\pm k,\widehat{i}}^\mu = \exp \left[\sum_{k=1}^\infty \sum_{r \in \BZ/g\BZ} \frac {P_{\pm k\bde,r}^\mu \overline{\alpha}_{\pm k, \widehat{i}, r}}k \right]\label{eqn:gg to pp}
\end{align}
for certain coefficients $\{\alpha_{\pm k, r,r'}, \overline{\alpha}_{\pm k, r,r'}\}^{k\in \BN}_{r,r' \in \BZ/g\BZ}$. Thus, it remains to show that:
\begin{equation}
\label{eqn:alphas}
\overline{\alpha}_{\pm k, r,r'} = - \alpha_{\pm k, r,r'}
\end{equation}
for all $k, r, r'$. Let us apply the map $\rho$ to the identities \eqref{eqn:e and f 1}--\eqref{eqn:e and f 2}, and recall (from Proposition \ref{prop:rho mult}) that $\rho$ is multiplicative on the subalgebra of elements of any fixed slope. With this in mind, Proposition \ref{prop:efg} implies that:
$$
\rho \left(G_{\pm k, \widehat{i}}^\mu \right) = \gamma_{\pm k, \mu, i} \qquad \text{and} \qquad \rho \left(G_{\pm k, \widehat{i\pm v}}^\mu \right) = \overline{\gamma}_{\pm k, \mu, i}
$$
We may now plug these formulas into \eqref{eqn:g to p} and \eqref{eqn:gg to pp}, and obtain:
\begin{align*}
&1 + \sum_{k=1}^\infty \frac {\gamma_{\pm k, \mu, i}}{x^k} = \frac {x - q^2 y_\pm}{x-y_\pm} = \exp \left[ \sum_{k=1}^\infty \sum_{r \in \BZ/g\BZ} \frac {\rho\left(P_{\pm k\bde,\hi}^\mu\right) \alpha_{\pm k, \widehat{i}, r}}{kx^k} \right] \\
&1 + \sum_{k=1}^\infty \frac {\overline{\gamma}_{\pm k, \mu, i \mp v}}{x^k} = \frac {x - q^{-2} z_\pm}{x - z_\pm} = \exp \left[ \sum_{k=1}^\infty \sum_{r \in \BZ/g\BZ} \frac {\rho\left(P_{\pm k\bde,\hi}^\mu\right) \overline{\alpha}_{\pm k, \widehat{i}, r}}{kx^k} \right]
\end{align*}
for $y_+ = \frac {t_{\mu,i}}{{\mathfrak{t}}^{\frac ag} q^{\frac {\mu n a}g + 1}}$, $y_- =\frac {s_{\mu,i}}{{\mathfrak{s}}^{\frac ag} q^{\frac {\mu n a}g - \frac {na}g+1}}$, $z_+ = \frac {(-\oq_-^{-\frac 2n})^{\frac {na}g}s_{\mu,i - v}}{\mathfrak{s}^{\frac ag} q^{\frac {\mu n a}g - 1}}$, $z_- = \frac {(-\oq_+^{-\frac 2n})^{\frac {na}g}t_{\mu,i + v}}{\mathfrak{t}^{\frac ag} q^{\frac {\mu n a}g + \frac {n a}g - 1}}$. \\

\noindent Taking the logarithms of the above power series in $x^{-1}$, we conclude that:
\begin{align}
&\sum_{r \in \BZ/g\BZ} \rho\left(P_{\pm k\bde,\hi}^\mu\right) \alpha_{\pm k, \widehat{i}, r} = (1-q^{2k})y_\pm^k \label{eqn:dbtr} \\
&\sum_{r \in \BZ/g\BZ} \rho\left(P_{\pm k\bde,\hi}^\mu\right) \overline{\alpha}_{\pm k, \widehat{i}, r} = (1-q^{-2k})z_\pm^k \label{eqn:cgis}
\end{align}
$\forall \ k > 0$. Since $\mathfrak{s} = \mathfrak{t} \cdot (-q^2)^n \oq^2$ and $s_{\mu, i} = t_{\mu, i+v}$, we conclude that $z_\pm = q^2 y_\pm$, hence the right-hand sides of \eqref{eqn:dbtr} and \eqref{eqn:cgis} are opposites of each other. Since this holds for all $\widehat{i}$ mod $g$, Lemma \ref{lem:min} implies \eqref{eqn:alphas} for all $r,r'$. 

\end{proof} 

\begin{proposition}
\label{prop:minimal}

If $\mu = \frac ba$ with $\gcd(a,b) = 1$, then for all $j-i=a$ we have:
\begin{align}
&E_{\pm [i;j)}^\mu = F_{\pm [i;j)}^\mu \cdot q^{-1} \label{eqn:minimal e 1} \\
&\bE_{\pm [i;j)}^\mu = \bF_{\pm [i;j)}^\mu \cdot q \label{eqn:minimal f 1}
\end{align} 
if $n \nmid a$, while:
\begin{align}
&E_{\pm [i;j)}^\mu = \sum_{u=1}^n F_{\pm [u;j-i+u)}^\mu \cdot q^{-1} \left( q^{-1}\delta_u^i + (q - q^{-1})\frac {\oq_\mp^{\frac 2n \cdot \overline{\pm b(u-i)}}}{\oq_\mp^2-1} \right) \label{eqn:minimal e 2} \\
&\bE_{\pm [i;j)}^\mu = \sum_{u=1}^n  \bF_{\pm [u;j-i+u)}^\mu \cdot q \left(q^{-1} \delta_u^i + (q - q^{-1})\frac {\oq_\mp^{\frac 2n \cdot \overline{\pm b(u-i)}}}{\oq_\mp^2-1} \right) \label{eqn:minimal f 2}
\end{align}
if $n | a$. \\

\end{proposition}

\begin{proof} Our assumptions imply that the left and right-hand sides of \eqref{eqn:minimal e 1}--\eqref{eqn:minimal f 2} are primitive elements of $\CB_\mu$. Therefore, Lemma \ref{lem:min} says that it is sufficient to show that the map $\rho$ takes the same values on the left and right-hand sides of these equations. We will prove the $\pm = +$ case of both \eqref{eqn:minimal e 1} and \eqref{eqn:minimal e 2}, and leave all the remaining cases as exercises to the interested reader. By Proposition \ref{prop:rho computations}, we have:
\begin{equation}
\label{eqn:lhs minimal}
\rho \left( E_{[i;j)}^{\mu} \right) = \rho \left( \bA_{[i;j)}^{\mu} \right) = \frac {q^{\overline{\alpha}_{i,j}^\mu} \tau_{i,j}^\mu}{\fT_{i,j}} \left(1 - \frac {x_i q^{2\delta_j^i}}{x_{i-1}q^2} \right)
\end{equation}
and:
\begin{equation}
\label{eqn:rhs minimal}
\rho \left( F_{[i;j)}^{\mu} q^{\delta_j^i-1} \right) = \rho \left( B_{[i;j)}^{\mu} q^{j-i+\delta_j^i-1} \right) = \frac {q^{\beta_{i,j}^\mu+j-i+\delta_j^i-1} \tau_{i,j}^\mu}{\fT_{i,j}} \left(1 - \frac {x_i}{x_{i-1} q^2} \right)
\end{equation} 
Under the assumption that $j-i=a$, relation \eqref{eqn:basic 11} reads:
$$
\overline{\alpha}_{i,j}^\mu = \beta_{i,j}^\mu+j-i+\delta_j^i-1 - \begin{cases} 0 &\text{if }n \nmid a \\ 2 &\text{if }n|a \end{cases}
$$
Therefore, we conclude that \eqref{eqn:lhs minimal} is equal to \eqref{eqn:rhs minimal} when $n \nmid a$, and this establishes \eqref{eqn:minimal e 1}. For the remainder of this proof, we will deal with the case $n|a$, when formula \eqref{eqn:rhs minimal} implies that:
$$
\rho \left( F_{[u;j-i+u)}^{\mu} \right) =\frac {q^{\overline{\alpha}_{u,j-i+u}^\mu} \tau_{u,j-i+u}^\mu}{\fT_{u,j-i+u}} \left(q^2 - \frac {x_u}{x_{u-1}} \right)
$$
for all integers $u$. We have $\fT_{i,j} = \fT_{u,j-i+u}$, $\overline{\alpha}_{u,j-i+u}^\mu = \overline{\alpha}_{i,j}^\mu$ and:
$$
\tau_{u,j-i+u}^\mu = \prod_{s=0}^{a-1} \left(x_{\overline{u+s}} \oq^{\frac 2n \cdot \overline{u+s}} \right)^{\left \lfloor \mu(s+1) \right \rfloor - \left \lfloor \mu s \right \rfloor}
$$
However, we have the following elementary identity, which holds for all integers $b,c$ and all residues $r \in \{1,...,n\}$:
$$
\sum_{0 \leq s < nc}^{s \equiv r \text{ mod }n} \left( \left \lfloor \frac {b(s+1)}{nc} \right \rfloor - \left \lfloor \frac {bs}{nc} \right \rfloor \right) = \gcd(b,c) \left( \left \lfloor \frac {b(r+1)}{n\gcd(b,c)} \right \rfloor - \left \lfloor \frac {br}{n\gcd(b,c)} \right \rfloor \right) 
$$
Back to our setting, letting $c = \frac an$ and recalling that $\gcd(a,b) = 1$, we infer that:
$$
\tau_{u,j-i+u}^\mu = \prod_{r=1}^n \left(x_{r} \oq^{\frac {2r}n} \right)^{\left \lfloor \frac {b(r-u+1)}n \right \rfloor - \left \lfloor \frac {b(r-u)}n \right \rfloor}
$$
Therefore, \eqref{eqn:minimal e 2} follows from the identity:
$$
\sum_{u=1}^n \left(q^2 - \frac {x_u}{x_{u-1}} \right) \prod_{r=1}^n \left(x_{r} \oq^{\frac {2r}n} \right)^{\left \lfloor \frac {b(r-u+1)}n \right \rfloor - \left \lfloor \frac {b(r-u)}n \right \rfloor} \left(q^{-2} \delta_u^i + (1 - q^{-2})\frac {\oq_-^{\frac 2n \cdot \overline{b(u-i)}}}{\oq_-^2-1} \right) = 
$$
$$
=  \left(1 - \frac {x_i}{x_{i-1}} \right) \prod_{r=1}^n \left(x_{r} \oq^{\frac {2r}n} \right)^{\left \lfloor \frac {b(r-i+1)}n \right \rfloor - \left \lfloor \frac {b(r-i)}n \right \rfloor}
$$
which holds for all $i,u \in \BZ$ and $b$ that is coprime with $n$. Since the aforementioned identity is elementary, we leave it to the interested reader. \\

\end{proof}

\subsection{} Let us now use the results above to prove Propositions \ref{prop:identity 1} and \ref{prop:identity 2}. \\

\begin{proof} \emph{of Proposition \ref{prop:identity 1}:} We will only prove \eqref{eqn:identity yz}, as \eqref{eqn:identity yyzz} is analogous, and hence left as an exercise to the interested reader. We have:
$$
Y_{\pm [i;j), \pm [i';j')}^\mu = \sum_{\frac {i-j'}n \leq l \leq \frac {j-i'}n} \bE_{\pm [i'+nl;j)}^\mu E_{\pm [i;j'+nl)}^\mu \stackrel{\eqref{eqn:e and f 1}, \eqref{eqn:e and f 2}}=
$$
\begin{equation}
\label{eqn:strung}
= \sum_{\frac {i-j'}n \leq l \leq \frac {j-i'}n} \sum_{k,k' \geq 0} \bF_{\pm [i'+nl;j-\frac {nak}g)}^\mu F_{\pm [i;j'+nl-\frac {nak'}g)}^\mu q^{\delta_{j'}^i - \delta_{i'}^j} \bG_{\pm k, \widehat{i' \pm v}}^\mu G_{\pm k', \widehat{i}}^\mu
\end{equation}
(in the latter formula, we used the fact that group-like elements are central in $\CB_\mu^+$). The assumption \eqref{eqn:gcd 1} implies that $a|j-i\pm v$ and $a|j'-i'\mp v$. Moreover, we assume that $i' \equiv j$ and $j' \equiv i$ modulo $g = \gcd(n,a)$, otherwise the right-hand sides of \eqref{eqn:formula y} and \eqref{eqn:formula zz} are 0 and the problem is vacuous. Therefore, we conclude that $i' \pm v \equiv i$ modulo $g$. By switching the order of the sums, formula \eqref{eqn:strung} equals:
$$
\sum_{t \geq 0} \sum_{\frac {i-j'}n \leq l \leq \frac {j-i'}n - \frac {at}g} \bF_{\pm [i'+nl;j-\frac {nat}g)}^\mu F_{\pm [i;j'+nl)}^\mu q^{\delta_{j'}^i - \delta_{i'}^j} \left[ \sum_{k = 0}^t  \bG_{\pm k, \widehat{i}}^\mu G_{\pm (t-k), \widehat{i}}^\mu \right]
$$
Because the series \eqref{eqn:g} and \eqref{eqn:gg} are inverses of each other, only the $t=0$ summand in the formula above survives, and it precisely equals the RHS of \eqref{eqn:identity yz}.

\end{proof} 

\begin{proof} \emph{of Proposition \ref{prop:identity 2}:} We will prove the case $\pm = +$, as the case $\pm = -$ is analogous. Due to Lemma \ref{lem:elem}, the required identity \eqref{eqn:identity yyy} is equivalent to:
\begin{equation}
\label{eqn:identity yyy equiv}
\sum_{x' \in \BZ/n\BZ} \bY_{[i;j),[i'+x';j'+x')} \left( \delta_{x'}^0 q^{-\delta_{j'}^{i'}} + (q-q^{-1}) \delta_{j'}^{i'} \frac {\oq_-^{\frac 2n \cdot \overline{-\e k'x'}}}{\oq_-^2-1} \right) = 
\end{equation}
$$
= q^{\delta_{i'}^i - \delta_{j'}^j} \sum_{x \in \BZ/n\BZ} Y_{[i';j'), [i+x;j+x)} \left( \delta_{x}^0 q^{-\delta_{j}^{i}} + (q - q^{-1}) \delta_{j}^{i} \frac {\oq_-^{\frac 2n \cdot \overline{- \e kx}}}{\oq_-^2-1} \right)
$$
which we will now prove. Because of formulas \eqref{eqn:cop1} and \eqref{eqn:cop1 anti}, we have:
$$
\Delta_\mu \left(\bY_{[i;j),[i';j')} \right) =  \sum_{\frac {i-j'}n \leq l \leq \frac {j-i'}n} \Delta_\mu \left( E_{[i'+nl;j)}^\mu \right) \Delta_\mu \left( \bE_{[i;j'+nl)}^\mu \right) =
$$
$$
= \sum_{(s,t) \in \zzz} \sum_{l,l'} E_{[s+nl';j)} \frac {\psi_s \psi_{j'}}{\psi_t \psi_{i'}} \bE_{[i;t+nl')} \otimes E_{[i'+nl;s)} \bE_{[t;j'+nl)}
$$
We make two remarks about the formula above: firstly, we set $\barc=1$ in the coproduct, to keep notation simple (it will not be important for this proof). Secondly, the sum in the latter expression involves picking an arbitrary representative $(s,t) \in \BZ^2$ of a certain residue class modulo $(n,n)$, and then summing over all $l,l'$ such that all the $E_{[x;y)}$ and $\bE_{[x;y)}$ which appear in the formula have $x \leq y$. After commuting all the $\psi$ symbols to the very left of the expression above, we conclude that:
\begin{equation}
\label{eqn:delta before}
\Delta_\mu \left(\bY_{[i;j),[i';j')} \right) = \sum^{(s,t) \in \zzz}_{i' - j' \leq s-t \leq j - i} \frac {\psi_s \psi_{j'}}{\psi_t \psi_{i'}} \bY_{[i;j),[s;t)} \otimes \bY_{[t;s),[i';j')} q^{\bar{*}}
\end{equation}
where $\overline{*} = 1 - \delta_s^j - \delta_s^t + \delta_t^j + \delta_s^{j'} - \delta_s^{i'} - \delta_j^{j'} + \delta_j^{i'}$. Therefore, we obtain:
$$
\Delta_\mu \left[ \sum_{x' \in \BZ/n\BZ} \bY_{ [i;j), [i'+x';j'+x')} \left( \delta_{x'}^0 q^{-\delta_{j'}^{i'}} + (q-q^{-1}) \delta_{j'}^{i'} \frac {\oq_-^{\frac 2n \cdot \overline{-\e k'x'}}}{\oq_-^2-1} \right) \right] = 
$$
$$
= \sum_{x' \in \BZ/n\BZ} \sum^{(s,t) \in \zzz}_{i' - j' \leq s-t \leq j - i} \frac {\psi_s \psi_{j'}}{\psi_t \psi_{i'}} \bY_{[i;j),[s;t)} \otimes \bY_{[t;s),[i'+x';j'+x')} q^{\overline{*}} \cdot
$$
\begin{equation}
\label{eqn:delta yy}
\left( \delta_{x'}^0 q^{-\delta_{j'}^{i'}} + (q-q^{-1}) \delta_{j'}^{i'} \frac {\oq_-^{\frac 2n \cdot \overline{-\e k'x'}}}{\oq_-^2-1} \right)
\end{equation}
\footnote{Note that the number $\overline{*}$ in \eqref{eqn:delta yy} is still given by the same formula as the same-named number in \eqref{eqn:delta yy}, because $i' \equiv j'$ modulo $n$ implies: 
$$
\delta_s^{j'} - \delta_s^{i'} - \delta_j^{j'} + \delta_j^{i'} = \delta_s^{j'+x'} - \delta_s^{i'+x'} - \delta_j^{j'+x'} + \delta_j^{i'+x'}
$$
} Similarly, we have:
$$
\Delta_\mu \left[q^{\delta_{i'}^i - \delta_{j'}^j} \sum_{x \in \BZ/n\BZ} Y_{[i';j'), [i+x;j+x)} \left( \delta_{x}^0 q^{-\delta_{j}^{i}} + (q - q^{-1}) \delta_{j}^{i} \frac {\oq_-^{\frac 2n \cdot \overline{- \e kx}}}{\oq_-^2-1} \right) \right] = 
$$
$$
\sum_{x \in \BZ/n\BZ} \sum^{(s,t) \in \zzz}_{i' - j' \leq s-t \leq j - i}  \frac {\psi_s \psi_{j'}}{\psi_t \psi_{i'}} Y_{[s;t),[i+x;j+x)} \otimes Y_{[i';j'), [t;s)} q^{* + \delta_{i'}^i - \delta_{j'}^j} \cdot 
$$
\begin{equation}
\label{eqn:delta y}
\left( \delta_{x}^0 q^{-\delta_{j}^{i}} + (q - q^{-1}) \delta_{j}^{i} \frac {\oq_-^{\frac 2n \cdot \overline{-\e kx}}}{\oq_-^2-1} \right)
\end{equation}
where $* = 1 - \delta_s^j - \delta^{i'}_s + \delta^{i'}_j + \delta_s^{i} - \delta_s^t - \delta_i^{i'} + \delta_{t}^{i'}$. Let us prove \eqref{eqn:identity yyy equiv} by induction on $j-i+j'-i'$ (the base case will be covered by the induction step). By applying the induction hypothesis, we see that the intermediate terms in the RHS of \eqref{eqn:delta yy} and \eqref{eqn:delta y} are both equal to the intermediate terms in the following expression: 
$$
\sum_{x \in \BZ/n\BZ} \sum^{(s,t) \in \zzz}_{i' - j' \leq s-t \leq j - i}  \frac {\psi_s \psi_{j'}}{\psi_t \psi_{i'}} \bY_{[i;j),[s;t)} \otimes Y_{[i';j'), [t+x;s+x)} q^{\#} \left( \delta_{x}^0 q^{-\delta_{t}^{s}} + (q - q^{-1}) \delta_{t}^{s} \frac {\oq_-^{\frac 2n \cdot \overline{-\e lx}}}{\oq_-^2-1} \right)
$$
where $\# = \overline{*} + \delta_t^{i'} - \delta_s^{j'} = * + \delta_{i'}^i - \delta_{j'}^j + \delta_j^t - \delta_s^i$ and $l = \frac {b(s-t)+\e}a$ (the latter expression is an integer, because unless $s-t \equiv j-i$ mod $a$, all the $Y$'s and $\bY$'s in \eqref{eqn:delta yy} and \eqref{eqn:delta y} are 0). Therefore, we conclude that the intermediate terms in \eqref{eqn:delta yy} and \eqref{eqn:delta y} are equal to each other. Therefore, it remains to show that the left and right-hand sides of \eqref{eqn:identity yyy equiv} have the same values under the linear maps \eqref{eqn:alpha}. For all $u, v$ satisfying $v-u = j-i+j'-i'$, Claim \ref{claim:combi} states that:
$$
\alpha_{[u;v)}\left[\sum_{x \in \BZ/n\BZ} Y_{[i';j'), [i+x;j+x)} \left( \delta_{x}^0 q^{-\delta_{j}^{i}} + (q - q^{-1}) \delta_{j}^{i} \frac {\oq_-^{\frac 2n \cdot \overline{- \e k x}}}{\oq_-^2-1} \right) \right] = 
$$
\begin{equation}
\label{eqn:mult y} 
= \left(q^{-1}-q \right) \left( \delta_{u}^{i} \delta_{j}^{i'} \oq_+^{\frac {j-i+j'-i'}{na}} q^{\delta_{i'}^{j} - \delta_{i'}^i} - \delta_{u}^{i'} \delta_{j'}^i \oq_+^{-\frac {j-i+j'-i'}{na}} q^{\delta_{j'}^j - \delta_{j'}^i} \right) 
\end{equation}
(technically, the proof of Claim \ref{claim:combi} only deals with the case $\e = -1$, but the case $\e = 1$ is analogous) and we also have the following analogous formula:
$$
\alpha_{[u;v)}\left[\sum_{x' \in \BZ/n\BZ} \bY_{[i;j),[i'+x';j'+x')} \left( \delta_{x'}^0 q^{-\delta_{j'}^{i'}} + (q-q^{-1}) \delta_{j'}^{i'} \frac {\oq_-^{\frac 2n \cdot \overline{- \e k' x'}}}{\oq_-^2-1} \right) \right] = 
$$
\begin{equation}
\label{eqn:mult yy} 
= \left( q^{-1}-q \right) \left( \delta_{u}^{i} \delta_{i'}^j \oq_+^{\frac {j-i+j'-i'}{na}} q^{\delta_{i'}^j - \delta_{j'}^j} - \delta_{u}^{i'} \delta_{j'}^i \oq_+^{-\frac {j-i+j'-i'}{na}} q^{\delta_{i'}^i - \delta_{j'}^i} \right)
\end{equation}
Formulas \eqref{eqn:mult y} and \eqref{eqn:mult yy} are equal up to $q^{\delta_{i'}^i - \delta_{j'}^j}$, thus proving \eqref{eqn:identity yyy equiv}.

\end{proof} 

\section{Index of notations}
\label{sec:index}

In the present Section, we will give a list of the most common notations in the present paper, and the location where they were first encountered. \\

\noindent $\delta_i^j$ \hfill Subsection \ref{sub:su} \\

\noindent $\delta_{i \text{ mod }g}^j$ \hfill \eqref{eqn:kron mod} \\

\noindent $\delta_{(i',j')}^{(i,j)}$ \hfill \eqref{eqn:kron mod double} \\

\noindent $\oq_\pm$ \hfill \eqref{eqn:parameters} \\

\noindent $x_i^\pm$ \hfill \eqref{eqn:special} \\

\noindent $\psi_s$ \hfill \eqref{eqn:special} \\

\noindent $\psi_{s+n} = c\psi_s$ \hfill \eqref{eqn:quasi per} \\

\noindent $\ph_s$ \hfill \eqref{eqn:psi} \\

\noindent $p_{\pm k}$ \hfill \eqref{eqn:heisenberg} \\

\noindent $e_{\pm [i;j)}$, $\be_{\pm [i;j)}$ \hfill \eqref{eqn:collection of generators 1}, \eqref{eqn:e bare} \\

\noindent $a_{s+n,d} = a_{s,d} \oq^{-2 d}$ \hfill \eqref{eqn:extend} \\

\noindent $\ph_s^{\pm}$ \hfill \eqref{eqn:convert 1}, \eqref{eqn:convert 2} \\

\noindent connection between $x_i^\pm$ and $e_{\pm [i;j)}$ \hfill \eqref{eqn:simple correspondence 1}--\eqref{eqn:simple correspondence 2} \\

\noindent connection between $p_{\pm k}$ and $e_{\pm [i;j)}$ \hfill Definition \ref{def:correspondence} \\

\noindent connection between $e_{\pm [i;j)}$ and $\be_{\pm [i;j)}$ \hfill \eqref{eqn:identities antipode} \\

\noindent $\Delta(\psi_s)$ \hfill \eqref{eqn:cop special 1} \\

\noindent $\Delta(x_i^\pm)$ \hfill \eqref{eqn:cop special 2}--\eqref{eqn:cop special 3} \\

\noindent $\Delta(p_{\pm k})$ \hfill \eqref{eqn:cop heis 1}--\eqref{eqn:cop heis 2} \\

\noindent $\Delta(e_{\pm [i;j)})$, $\Delta(\be_{\pm [i;j)})$ \hfill \eqref{eqn:cop quant 1}--\eqref{eqn:cop quant 2}, \eqref{eqn:cop antipode 1}--\eqref{eqn:cop antipode 2} \\

\noindent $\Delta(a_{s,\pm d})$ \hfill \eqref{eqn:coproduct 1}--\eqref{eqn:coproduct 2} \\

\noindent $\Delta(R^\pm)$ \hfill \eqref{eqn:coproduct 3}--\eqref{eqn:coproduct 4} \\

\noindent $\Delta_{\mu}(R^\pm)$ \hfill \eqref{eqn:cristian1}--\eqref{eqn:cristian2} \\

\noindent intermediate terms in the coproduct \hfill \eqref{eqn:intermediate terms 1}--\eqref{eqn:intermediate terms 2} \\

\noindent $\langle x_i^+, x_j^- \rangle = \frac {\delta_j^i}{q^{-1}-q}$ \hfill \eqref{eqn:pairing village} \\

\noindent $\langle \psi_s,\psi_{s'} \rangle = q^{-\delta_{s'}^s}$  \hfill \eqref{eqn:pairing village} \\

\noindent $\langle p_k, p_{-k} \rangle = \frac k{q^{-k}-q^k}$ \hfill \eqref{eqn:pairing town} \\

\noindent $\langle e_{[i;j)}, e_{-[i;j)} \rangle = 1 - q^{-2}$ \hfill \eqref{eqn:pair quantum 2} \\

\noindent $\langle a_{s,d}, a_{s',-d'} \rangle$ \hfill \eqref{eqn:pairshuf0} \\

\noindent $\langle R^+, R^- \rangle$ \hfill \eqref{eqn:pairshuf1} \\

\noindent $\su$, $\sug$, $\sul$, $\supm$ \hfill \eqref{eqn:special}, \eqref{eqn:positive}, \eqref{eqn:negative}, \eqref{eqn:notation su} \\

\noindent $\uui$, $\uuig$, $\uuil$, $\uuipm$ \hfill \eqref{eqn:heisenberg}, \eqref{eqn:positive heis}, \eqref{eqn:negative heis}, \eqref{eqn:notation heis}  \\

\noindent $\uu$, $\uupm$ \hfill \eqref{eqn:quantum as tensor},\eqref{eqn:notation uu} \\

\noindent $\tau$ \hfill \eqref{eqn:tau} \\

\noindent $\deg$ (the grading on quantum groups) \hfill \eqref{eqn:grading} \\

\noindent $\hdeg$ (horizontal grading on the shuffle algebra) \hfill \eqref{eqn:hdeg} \\

\noindent $\vdeg$ (vertical grading on the shuffle algebra) \hfill \eqref{eqn:vdeg} \\

\noindent $\langle \bk, \bl \rangle$ \hfill \eqref{eqn:bilinear form} \\

\noindent $|\bk|$ \hfill \eqref{eqn:length} \\

\noindent $\zeta \left(\frac {z_{ia}}{z_{jb}} \right)$ \hfill \eqref{eqn:def zeta} \\

\noindent $R * R'$ \hfill \eqref{eqn:mult} \\

\noindent $\CA^\pm$ \hfill Definition \ref{def:shuf}, \eqref{eqn:opposite shuffle} \\

\noindent $\CA$, $\CA^{\geq}$, $\CA^{\leq}$ \hfill \eqref{eqn:hmm}, \eqref{eqn:a geq}, \eqref{eqn:a leq} \\

\noindent $\CA_{\pm \bk}$, $\CA_{\pm \bk,d}$ \hfill \eqref{eqn:a plus hdeg}, \eqref{eqn:a minus hdeg},\eqref{eqn:a plus deg}, \eqref{eqn:a minus deg} \\

\noindent \text{naive slope} \hfill \eqref{eqn:naive slope} \\

\noindent slope \hfill Definition \ref{def:slope} \\

\noindent hinge \hfill Subsection \ref{sub:diagrams} \\

\noindent $\CA^\pm_{\leq \mu}$, $\CA_{\leq \mu|\pm \bk}$, $\CA_{\leq \mu|\pm \bk,d}$ \hfill \eqref{eqn:slope notation}, \eqref{eqn:slope notation hdeg}, \eqref{eqn:slope notation deg} \\

\noindent $\CB_\mu^\pm$ \hfill \eqref{eqn:sub} \\

\noindent $\CB_\mu^{\geq}$, $\CB_\mu^{\leq}$, $\CB_\mu$ \hfill \eqref{eqn:sub1}--\eqref{eqn:sub2}, \eqref{eqn:double sub}

$$
\boxed{X_{\pm [i;j)}^\mu = X_{\pm [i;j)}^{(\pm k)}}
$$
for $k = \mu(j-i)$ and $X \in \{A,\bA,B,\bB,E,\bE,F,\bF, P, \tP, e, \be, f, \bf, p, \tp, o, \to\}$ \hfill \eqref{eqn:often write} \\

\noindent $A_{\pm [i;j)}^{\mu}$, $\bA_{\pm [i;j)}^{\mu}$, $B_{\pm [i;j)}^{\mu}$, $\bB_{\pm [i;j)}^{\mu}$ \hfill \eqref{eqn:a}, \eqref{eqn:ba}, \eqref{eqn:b}, \eqref{eqn:bb} \\

\noindent $E_{\pm [i;j)}^{\mu}$, $\bE_{\pm [i;j)}^{\mu}$ \hfill \eqref{eqn:e plus}--\eqref{eqn:be minus} \\

\noindent $F_{\pm [i;j)}^{\mu}$, $\bF_{\pm [i;j)}^{\mu}$ \hfill \eqref{eqn:two 1}--\eqref{eqn:two 8} \\

\noindent $P_{\pm [i;j)}^{(\pm k)}$, $P^{(\pm k')}_{\pm l \bde, \hi}$ \hfill \eqref{eqn:p plus}, \eqref{eqn:p minus}, \eqref{eqn:def p} \\

\noindent $\tP_{\pm [i;j)}^{(\pm k)}$ \hfill \eqref{eqn:tilde P}, \eqref{eqn:tilde P equivalent} \\

\noindent $G_{\pm k}^\mu$, $\bG_{\pm k}^\mu$ \hfill \eqref{eqn:g}, \eqref{eqn:gg} \\

\noindent $Y_{\pm [i;j), \pm [i';j')}^\mu$, $\bY_{\pm [i;j), \pm [i';j')}^\mu$ \hfill \eqref{eqn:formula y}, \eqref{eqn:formula yy} \\

\noindent $Z_{\pm [i;j), \pm [i';j')}^\mu$,$\bZ_{\pm [i;j), \pm [i';j')}^\mu$ \hfill \eqref{eqn:formula z}, \eqref{eqn:formula zz} \\

\noindent connection between $E^\mu_{\pm [i;j)}$ and $\bE^\mu_{\pm [i;j)}$ \hfill \eqref{eqn:antipode} \\

\noindent $\Delta_\mu \left(E_{\pm [i;j)}^{\mu} \right)$, $\Delta_\mu \left(\bE_{\pm [i;j)}^{\mu} \right)$ \hfill \eqref{eqn:cop1}--\eqref{eqn:cop2 anti} \\

\noindent $\Delta_{\mu} \left(F_{\pm [i;j)}^{\mu} \right)$,$\Delta_{\mu} \left(\bF_{\pm [i;j)}^{\mu} \right)$ \hfill \eqref{eqn:cop f 1}--\eqref{eqn:cop f 8} \\

\noindent $\Delta_{\frac k{j-i}} \left(P_{\pm [i;j)}^{(\pm k)} \right)$, $\Delta_{\frac {k'}{nl}} \left(P^{(\pm k')}_{\pm l \bde, \hi}\right)$ \hfill \eqref{eqn:primitive 1}--\eqref{eqn:primitive 2}, \eqref{eqn:condition 1}--\eqref{eqn:condition 2} \\

\noindent $\alpha_{\pm [i;j)}(R^\pm)$ \hfill \eqref{eqn:alpha}, \eqref{eqn:beta} \\

\noindent $\alpha_{\pm [i';j')}(E_{\pm [i;j)}^{\mu})$, $\alpha_{\pm [i';j')}(E_{\pm [i;j)}^{\mu})$ \hfill \eqref{eqn:main pair 1}--\eqref{eqn:main pair 2} \\

\noindent $\alpha_{\pm [u;v)}\left(P^{(\pm k)}_{\pm [i;j)}\right)$, $\alpha_{\pm [u;u+nl)} \left(P^{(\pm k')}_{\pm l \bde, \hi}\right)$ \hfill \eqref{eqn:normalize simple}, \eqref{eqn:normalize imaginary} \\

\noindent connection of $\alpha_{\pm [i;j)}$ with the pairing \hfill \eqref{eqn:bonnie} \\

\noindent $\left \langle E_{[i;j)}^\mu, E_{-[i';j')}^\mu \right \rangle$ \hfill \eqref{eqn:pair1} \\

\noindent $\Big \langle P_{\pm [i;j)}^{(\pm k)}, E_{\mp [i';j')}^{(\mp k)} \Big \rangle$, $\Big \langle P_{\pm [i;j)}^{(\pm k)}, \bE_{\mp [i';j')}^{(\mp k)} \Big \rangle$ \hfill \eqref{eqn:main pair 3}, \eqref{eqn:main pair 5} \\

\noindent $\Big \langle P_{\pm l\bde, \hi}^{(\pm k)}, E_{\mp [i';j')}^{(\mp k)} \Big \rangle$, $\Big \langle P_{\pm l\bde, \hi}^{(\pm k)}, \bE_{\mp [i';j')}^{(\mp k)} \Big \rangle$ \hfill \eqref{eqn:main pair 4}, \eqref{eqn:main pair 6} \\

\noindent $\Big \langle P_{[i;j)}^{(k)}, P_{-[i';j')}^{(-k)} \Big \rangle$, $\Big \langle \tP_{[i;j)}^{(k)}, \tP_{-[i';j')}^{(-k)} \Big \rangle$ \hfill \eqref{eqn:pair simple}, \eqref{eqn:pair simple tilde} \\

\noindent $e_{\pm [i;j)}^{(\pm k)}$, $\be_{\pm [i;j)}^{(\pm k)}$ \hfill \eqref{eqn:assignment 2} \\

\noindent $f_{\pm [i;j)}^{(\pm k)}$, $\bf_{\pm [i;j)}^{(\pm k)}$ \hfill \eqref{eqn:stipulate} \\

\noindent $p_{\pm [i;j)}^{(\pm k)}$, $p^{(\pm k')}_{\pm l \bde, \hi}$ \hfill \eqref{eqn:assignment 1} \\

\noindent $\tp_{\pm [i;j)}^{(\pm k)}$ \hfill \eqref{eqn:tilde p}, \eqref{eqn:tilde p equivalent} \\

\noindent $o_{\pm [i;j)}^{(\pm k)}$, $o_{\pm l\bde, \hi}^{(\pm k')}$ \hfill \eqref{eqn:tp 1}--\eqref{eqn:tp 2} \\

\noindent $\to_{\pm [i;j)}^{(\pm k)}$ \hfill \eqref{eqn:tp 5}, \eqref{eqn:tp 6} \\

\noindent $\gamma_{ijk}^\pm(x)$ \hfill \eqref{eqn:gamma} \\

\noindent $\CC$, $\CC^\pm$ \hfill \eqref{eqn:generators}, \eqref{eqn:want} \\

\noindent $\DD$, $\DD^\pm$ \hfill \eqref{eqn:generators new}, \eqref{eqn:qwerty new} \\

\noindent $\rho$ \hfill \eqref{eqn:def rho}, \eqref{eqn:rho} \\

\noindent $\alpha_{i,j}^\mu$, $\overline{\alpha}_{i,j}^\mu$ \hfill Proposition \ref{prop:rho computations} \\

\noindent $\beta_{i,j}^\mu$, $\overline{\beta}_{i,j}^\mu$ \hfill Proposition \ref{prop:rho computations} \\

\noindent $\fS_{i,j}$ \hfill \eqref{eqn:basic 1} \\

\noindent $\fT_{i,j}$ \hfill \eqref{eqn:basic 1} \\

\noindent $\sigma^{\mu}_{i,j}$ \hfill \eqref{eqn:basic 2} \\

\noindent $\tau^{\mu}_{i,j}$ \hfill \eqref{eqn:basic 3} \\

\noindent $\mathfrak{s}$ \hfill \eqref{eqn:basic 5} \\

\noindent $\mathfrak{t}$ \hfill \eqref{eqn:basic 6} \\

\noindent The author declares that he has no conflict of interest.

\end{document}